%% file: main.tex
\documentclass[11pt]{amsart}

\usepackage{package}
\usepackage{command}
\usepackage{bibstyle}

\usepackage{setting}

\begin{document}


\input{section/Information} 

\input{section/Abstract} 

\maketitle
\tableofcontents

\input{section/Introduction} 
\input{section/Preliminaries} 
\input{section/Assumptions} 
\input{section/Thermodynamic_formalism} 
\input{section/Proof} 

\appendix 
\input{section/Appendix/Ruelle_lemma} 

\printbibliography

\end{document}

%% file: section/Information.tex

\title[Counting Prime Orbits in Shrinking Intervals]
{
    Counting Prime Orbits in Shrinking Intervals \\ for Expanding Thurston Maps    
}


\author{Zhiqiang~Li \and Xianghui~Shi}

\thanks{Li and Shi were partially supported by Beijing Natural Science Foundation (JQ25001 and 1214021) and National Natural Science Foundation of China (12471083, 12101017, 12090010, and 12090015).}

    
\address{Zhiqiang~Li, School of Mathematical Sciences \& Beijing International Center for Mathematical Research, Peking University, Beijing 100871, China}
\email{zli@math.pku.edu.cn}

\address{Xianghui~Shi, Beijing International Center for Mathematical Research, Peking University, Beijing 100871, China}
\email{xhshi@pku.edu.cn}


\subjclass[2020]{Primary: 37D20; Secondary: 37C25, 37C35, 37D35, 57M12}

\keywords{expanding Thurston map, postcritically-finite rational map, prime orbits, thermodynamic formalism, Ruelle operator, shrinking targets.}

%% file: section/Abstract.tex

\begin{abstract}
We establish a local central limit theorem for primitive periodic orbits of expanding Thurston maps, providing a fine-scale refinement of the Prime Orbit Theorem in the context of non-uniformly expanding dynamics.
Specifically, we count the number of primitive periodic orbits whose Birkhoff sums for a given potential lie within a family of shrinking intervals. 
For eventually positive, real-valued \holder continuous potentials that satisfy the strong non-integrability condition, we derive precise asymptotic estimates. 
In particular, our results apply to postcritically-finite rational maps whose Julia set is the whole Riemann sphere.
\end{abstract}






%% file: section/Introduction.tex
\section{Introduction}
\label{sec:Introduction}


Periodic orbits serve as the skeleton of chaotic dynamics, encoding essential information about the system's long-term behavior.
A fundamental objective in this field is counting these orbits, a problem analogous to the Prime Number Theorem in number theory.
Recently, the dynamical counterpart---the Prime Orbit Theorem---was successfully established for expanding Thurston maps \cite{li2024prime:dirichlet, li2024prime:split}, which serve as topological models for postcritically-finite rational maps.
While this answers the question of ``how many'' orbits exist, it leaves open the more subtle question of ``how they are distributed'' with respect to statistical observables.
Understanding this fine-scale distribution is crucial for characterizing the fluctuations inherent in the system.

In this paper, we investigate the asymptotic distribution of primitive periodic orbits restricted to shrinking intervals for expanding Thurston maps.
We count orbits whose Birkhoff sums for a given potential lie within a prescribed family of shrinking intervals.
we obtain a precise asymptotic formula for the number of these constrained orbits, which resemble a local central limit theorem.
Unlike the central limit theorem, which describes the distribution of Birkhoff sums on the scale of $\sqrt{n}$, the local central limit theorem can probe the density at a given point.
Our main result refines the coarse counting of the Prime Orbit Theorem \cite{li2024prime:dirichlet} for expanding Thurston maps.


The problem of counting orbits in shrinking intervals has been successfully addressed for uniformly expanding systems.
For instance, Petkov and Stoyanov \cite{petkovDistributionPeriodsClosed2012} investigated the distribution of closed orbits for hyperbolic flows, and Sharp and Stylianou \cite{sharpStatisticsMultipliersHyperbolic2022} studied the multipliers and holonomies for hyperbolic rational maps.
However, these results rely heavily on the hyperbolicity and smoothness of the underlying systems.
The context of non-uniformly expanding dynamics, particularly for branched covering maps, remains largely unexplored.
To the best of our knowledge, our work is the first to address this problem in such a setting.

\subsection{Main results}%
\label{sub:Main results}

Let $f \colon S^2 \to S^2$ be an expanding Thurston map and $\phi \colon S^2 \to \mathbb{R}$ be a real-valued \holder continuous function. 
The topological $2$-sphere $S^2$ is equipped with a \emph{visual metric} $d$ (see Section~\ref{sub:Thurston_maps} for details). 
A periodic orbit $\tau = \set[\big]{x, f(x), \dots, f^{n-1}(x)}$ (where $f^n(x) = x$) is called \emph{primitive} if $f^m(x) \ne x$ for each integer $m$ with $1 \leqslant m < n$. 
We denote the set of primitive periodic orbits of $f$ by $\mathfrak{P}(f)$. 
For each $\tau \in \mathfrak{P}(f)$, we write $\phi(\tau) \coloneqq \sum_{x \in \tau} \phi(x)$.

In this article, we investigate the asymptotic distribution of primitive periodic orbits subject to constraints on their Birkhoff sums. 
Specifically, for a given number $\alpha \in \real$ and a sequence $\sequen{I_{n}}$ of intervals contained in a compact set $K \subseteq \real$, we study the asymptotic behavior of 
\[
    \pi_{f, \phi}(n; \alpha, I_n) \coloneqq \operatorname{card} \set[\big]{ \tau \in \mathfrak{P}_n(f) \describe \phi(\tau) - n \alpha \in I_{n}}
\]
as $n \to +\infty$, where $\mathfrak{P}_n(f) \coloneqq \set[\big]{\tau \in \mathfrak{P}(f) \describe |\tau| = n}$. 

To obtain precise estimates, we impose specific conditions on the potential $\phi$ and the sequence $\sequen{I_{n}}$.
We assume that $\phi$ is \emph{eventually positive} and satisfies the \emph{strong non-integrability condition} (see Definitions~\ref{def:eventually positive functions} and~\ref{def:strong non-integrability condition}). 
Furthermore, denoting the length of $I_{n}$ by $|I_n|$, we assume that the sequence $\sequen[\big]{|I_{n}|^{-1}}$ exhibits sub-exponential growth, i.e., $\limsup_{n \to +\infty} \frac{1}{n} \log \parentheses[\big]{ |I_{n}|^{-1} } = 0$.

We write $A(n) \sim B(n)$ as $n \to +\infty$ if $\lim_{n \to +\infty} A(n) / B(n) = 1$.

\begin{theorem} \label{thm:main theorem}
    Let $f \colon S^2 \to S^2$ be an expanding Thurston map and $d$ be a visual metric on $S^2$ for $f$.
    Let $\beta \in (0, 1]$ and $\phi \in C^{0, \beta}(S^2, d)$ be an eventually positive real-valued \holder continuous function satisfying the $\beta$-strong non-integrability condition (with respect to $f$ and $d$). 
    Then there exists a unique positive number $\rootpressure > 0$ with topological pressure $P(f, -\rootpressure \phi) = 0$ and there exists $N_{f} \in \n$ depending only on $f$ such that for each $N \in \n$ with $N \geqslant N_{f}$, the following statement holds for the iterate $F \coloneqq f^{N}$ and the potential $\Phi \coloneqq \sum_{i = 0}^{N - 1} \phi \circ f^{i}$:

    Denote $\alpha \coloneqq \frac{\mathrm{d}}{\mathrm{d} t} P(F, t \Phi) |_{t = -\rootpressure}$ and $\sigma \coloneqq \sqrt{ \frac{\mathrm{d}^2}{\mathrm{d} t^2} P(F, t \Phi) |_{t = -\rootpressure } }$.
    Let $\sequen{I_{n}}$ be a sequence of intervals contained in a compact set $K \subseteq \real$ with $\sequen[\big]{|I_{n}|^{-1}}$ having sub-exponential growth.
    Then
    \[
        \pi_{F, \Phi}(n; \alpha, I_n) \sim \frac{ \int_{I_{n}} e^{\rootpressure t} \, \mathrm{d}t }{ \sqrt{2\pi} \, \sigma } \, \frac{e^{\rootpressure \alpha n}}{n^{3/2}} \qquad \text{as } n \to +\infty.
    \]
\end{theorem}


Recall that a postcritically-finite rational map is expanding if and only if it has no periodic critical points (see \cite[Proposition~2.3]{bonk2017expanding}). 
Therefore, when we restrict our attention to rational maps, we obtain the following corollary of Theorem~\ref{thm:main theorem} and Remark~\ref{rem:chordal metric visual metric qs equiv}.

\begin{corollary}\label{coro:main theorem for postcritically-finite rational maps}
    Let $f \colon \ccx \mapping \ccx$ be a postcritically-finite rational map without periodic critical points.
    Let $\sigma$ be the chordal metric or the spherical metric on the Riemann sphere $\ccx$, and $\phi \in C^{0, \holderexp} \parentheses[\big]{ \ccx, \sigma }$ be an eventually positive real-valued \holder continuous function with exponent $\holderexp \in (0, 1]$ satisfying the $\holderexp$-strong non-integrability condition (with respect to $f$ and a visual metric).
    Then there exists a unique positive number $\rootpressure > 0$ with topological pressure $P(f, -\rootpressure \phi) = 0$ and there exists $N_f \in \n$ depending only on $f$ such that for each $N \in \n$ with $N \geqslant N_f$, the following statements hold for $F \define f^N$ and $\Phi \define \sum_{i=0}^{N-1} \phi \circ f^i$:
    
    Denote $\alpha \coloneqq \frac{\mathrm{d}}{\mathrm{d} t} P(F, t \Phi) |_{t = -\rootpressure}$ and $\sigma \coloneqq \sqrt{ \frac{\mathrm{d}^2}{\mathrm{d} t^2} P(F, t \Phi) |_{t = -\rootpressure } }$.
    Let $\sequen{I_{n}}$ be a sequence of intervals contained in a compact set $K \subseteq \real$ with $\sequen[\big]{|I_{n}|^{-1}}$ having sub-exponential growth.
    Then
    \[
        \pi_{F, \Phi}(n; \alpha, I_n) \sim \frac{ \int_{I_{n}} e^{\rootpressure t} \, \mathrm{d}t }{ \sqrt{2\pi} \, \sigma } \, \frac{e^{\rootpressure \alpha n}}{n^{3/2}} \qquad \text{as } n \to +\infty.
    \]
\end{corollary}

\subsection{Strategy and organization}
Our approach relies on a combination of thermodynamic formalism and operator theory, specifically adapted to the branched covering setting.

The main technical obstacle in studying Thurston maps is the presence of critical points, which disrupts the functional analytic properties of the standard Ruelle transfer operator.
To overcome this, we employ the \emph{split Ruelle operators} introduced in \cite{li2024prime:split}.
The idea is to decompose the sphere into ``black'' and ``white'' tiles (based on a checkerboard coloring induced by an invariant Jordan curve) and define a pair of operators acting on functions supported on these tiles.
This construction effectively ``unfolds'' the singularities, allowing us to recover good spectral properties.

To obtain the precise asymptotics required for the local central limit theorem, we need to control the decay of the characteristic function of the Birkhoff sums.
In terms of operator theory, this translates to bounding the spectral radius of the twisted transfer operator $\mathbb{L}_{s\phi}$ as the complex parameter $s$ moves along the imaginary axis.
The detailed estimates are separated into three parts: the unbounded part, the bounded part, and the local part.
For the unbounded part, we employ Dolgopyat-type estimates for the split Ruelle operators established in \cite{li2024prime:split}.
This requires checking a strong non-integrability condition (Definition~\ref{def:strong non-integrability condition}), which in particular implies that the potential is not cohomologous to a constant.
For the bounded part, we employ Ruelle's estimate (see Appendix~\ref{sec:Appendix:Ruelle lemma}).
For the local part, we employ the complex Ruelle--Perron--Frobenius theorem \cite[Theorem~2]{pollicott1984complex} and arguments in \cite{pollicottZetaFunctionsPeriodic1990}.

Finally, to count orbits in intervals $I_n$, we approximate the indicator function $\indicator{I_n}$ by smooth test functions.
We then apply Fourier transforms to relate the smoothed count to partition functions, which allows us to apply the established decay estimates.

\smallskip

The paper is organized as follows.
In Section~\ref{sec:Preliminaries}, we fix our notation, review fundamental concepts from thermodynamic formalism, and recall key results from the theory of expanding Thurston maps.
In Section~\ref{sec:The Assumptions}, we collect the main assumptions used throughout the paper.
The technical core of the paper is Section~\ref{sec:Pressure function and partition function estimates}, where we employ previous results and derive crucial decay estimates for associated partition functions.
These estimates are then used in Section~\ref{sec:Proof of the main theorem} to prove Theorem~\ref{thm:main theorem}.

%% file: section/Preliminaries.tex

\section{Preliminaries}
\label{sec:Preliminaries}

\subsection{Notation}
\label{sub:Notation}

Let $\cx$ be the complex plane and $\ccx$ be the Riemann sphere. 
For each complex number $z \in \cx$, we denote by $\Re{z}$ the real part of $z$, and by $\Im{z}$ the imaginary part of $z$.
The symbol $\mathbf{i}$ stands for the imaginary unit in the complex plane $\cx$.
We follow the convention that $\n = \{1,2,3,\dots\}$ and $\n_0 = \{0\} \cup \n$. 
We denote by $\nonnegreal$ the set of non-negative real numbers.
The cardinality of a set $A$ is denoted by $\card{A}$.

Consider complex-valued functions $u$, $v$, and $w$ defined on $\real$ and $a \in \real \cup \{\pm\infty\}$. 
We write $u(x) \sim v(x)$ as $x \to a$ if $\lim_{x \to a} \frac{u(x)}{v(x)} = 1$, and write $u(x) = v(x) + \mathcal{O}(w(x))$ as $x \to a$ if $\limsup_{x \to a} \abs[\big]{\frac{u(x) - v(x)}{w(x)}} < +\infty$. 
We use the same notation for discrete variables.


Let $f \colon X \mapping X$ be a map on a set $X$. 
The inverse map of $f$ is denoted by $f^{-1}$. 
We write $f^n$ for the $n$-th iterate of $f$, and $f^{-n}\define (f^n)^{-1}$, for $n\in \n$. 
For $n \in \n$, we denote by
\[
    \myperiodpoint{f}{n} \define \set{ x \in X \describe f^{n}(x) = x }
\]
the set of fixed points of $f^n$.
For a real-valued function $\varphi \colon X \mapping \real$, we write 
\[
    S_n \varphi(x) = S^f_n \varphi(x) \define \sum_{j=0}^{n-1} \varphi(f^j(x))
\] 
for $x \in X$ and $n\in \n_0$. We omit the superscript $f$ when the map $f$ is clear from the context. We adopt the convention that $S_0 \varphi = 0$.

Let $(X,d)$ be a compact metric space. 
For each subset $Y \subseteq X$, we denote the diameter of $Y$ by $\diameter{d}{Y} \define \sup\{d(x, y) \describe \juxtapose{x}{y} \in Y\}$ and the characteristic function of $Y$ by $\indicator{Y}$. 
We denote by $C(X)$ (\resp $C(X, \cx)$) the space of real-valued continuous functions from $X$ to $\real$ (\resp complex-valued continuous functions from $X$ to $\cx$). 
For $\varphi \in C(X)$, we denote by $\supp{\varphi}$ the support of $\varphi$.
If we do not specify otherwise, we equip $C(X)$ and $C(X, \cx)$ with the uniform norm $\uniformnorm{\cdot}$.
For $\psi \in C(X, \cx)$, we denote
\[
    |\psi|_{\holderexp} \define \sup\set[\big]{ \abs{ \psi(x) - \psi(y) } \big/ d(x,y)^{\holderexp} \describe x, y \in X, \, x \ne y },
\]
and the \holder norm of $\psi$ is defined as $\| \psi \|_{C^{0,\holderexp}} \define |\psi|_{\holderexp} + \|\psi\|_{\infty}$.
The space of real-valued (\resp complex-valued) \holder continuous functions with an exponent $\holderexp \in (0, 1]$ on $(X, d)$ is denoted by $C^{0,\holderexp}(X,d)$ (\resp $C^{0,\holderexp}((X,d), \cx)$), which consists of continuous functions with finite \holder norm.

\subsection{Thermodynamic formalism}  
\label{sub:thermodynamic formalism}

We first review some basic concepts from ergodic theory and dynamical systems. 
For more detailed studies of these concepts, we refer the reader to \cite[Chapter~20]{katok1995introduction} and \cite[Chapter~9]{walters1982introduction}.

\smallskip

Let $(X,d)$ be a compact metric space and $g \colon X \mapping X$ a continuous map. Given $n \in \n$, 
\[
    d^n_g(x, y) \define \operatorname{max} \bigl\{  d \bigl(g^k(x), g^k(y) \bigr) \describe k \in \{0, \, 1, \, \dots, \, n-1 \} \!\bigr\}, \quad \text{ for } \juxtapose{x}{y} \in X,
\]
defines a metric on $X$. A set $F \subseteq X$ is \emph{$(n, \epsilon)$-separated} (with respect to $g$), for some $n \in \n$ and $\epsilon > 0$, if for each pair of distinct points $\juxtapose{x}{y} \in F$, we have $d^n_g(x, y) \geqslant \epsilon$.

For each real-valued continuous function $\psi \in C(X)$, the following limits exist and are equal, and we denote these limits by $P(g, \psi)$ (see~e.g.~\cite[Subsection~20.2]{katok1995introduction}):
\begin{equation}  \label{eq:def:topological pressure}
    \begin{split}
        P(g, \psi) \define &  \lim \limits_{\epsilon \to 0^{+}} \limsup\limits_{n \to +\infty} \frac{1}{n} \log  N_{d}(g, \psi, \varepsilon, n) 
        = \lim \limits_{\epsilon \to 0^{+}} \liminf\limits_{n \to +\infty} \frac{1}{n} \log  N_{d}(g, \psi, \varepsilon, n),
    \end{split}
\end{equation}
where \[
    N_{d}(g, \psi, \varepsilon, n) \define \sup \set[\Big]{\sum_{x \in E} \myexp[\big]{S_n\psi(x)} \describe E \subseteq X \text{ is } (n, \varepsilon)\text{-separated with respect to } g}.
\]
We call $P(g, \psi)$ the \emph{topological pressure} of $g$ with respect to the \emph{potential} $\psi$. 
The quantity $h_{\operatorname{top}}(g) \define P(g, 0)$ is called the \emph{topological entropy} of $g$. 
The topological pressure $P(g, \psi)$ depends only on the topology of $X$ (see~e.g.~\cite[Subsection~20.2]{katok1995introduction}).


Let $\mathcal{M}(X,g)$ be the set of $g$-invariant Borel probability measures on $X$.
Let $\mu \in \mathcal{M}(X, g)$. 
We say $g$ is \emph{ergodic} for $\mu$ (or $\mu$ is \emph{ergodic} for $g$) if for each set $A \in \mathcal{B}$ with $g^{-1}(A) = A$ we have $\mu(A) = 0$ or $\mu(A) = 1$.

For each real-valued continuous function $\psi \in C(X)$, the \emph{measure-theoretic pressure} $P_\mu(g, \psi)$ of $g$ for the measure $\mu \in \mathcal{M}(X, g)$ and the potential $\psi$ is
\begin{equation}  \label{eq:def:measure-theoretic pressure}
    P_\mu(g, \psi) \define  h_\mu (g) + \int \! \psi \,\mathrm{d}\mu,
\end{equation}
where $h_{\mu}(g)$ is the measure-theoretic entropy of $g$ for $\mu$.

The topological pressure is related to the measure-theoretic pressure by the so-called \emph{Variational Principle}.
It states that (see~e.g.~\cite[Theorem~20.2.4]{katok1995introduction})
\begin{equation}  \label{eq:Variational Principle for pressure}
    P(g, \psi) = \sup \{P_\mu(g, \psi) \describe \mu \in \mathcal{M}(X, g)\}
\end{equation}
for each $\psi \in C(X)$.
In particular, when $\psi$ is the constant function $0$,
\begin{equation}  \label{eq:Variational Principle for entropy}
    h_{\operatorname{top}}(g) = \sup\{h_{\mu}(g) \describe\mu \in \mathcal{M}(X, g)\}.
\end{equation}
A measure $\mu$ that attains the supremum in \eqref{eq:Variational Principle for pressure} is called an \emph{equilibrium state} for the map $g$ and the potential $\psi$. A measure $\mu$ that attains the supremum in \eqref{eq:Variational Principle for entropy} is called a \emph{measure of maximal entropy} of $g$.

Let $\widetilde{X}$ be another compact metric space. If $\mu$ is a measure on $X$ and the map $\pi \colon X \mapping \widetilde{X}$ is continuous, then the \emph{push-forward} $\pi_{*} \mu$ of $\mu$ by $\pi$ is the measure given by $\pi_{*}\mu(A) \define \mu \parentheses[\big]{ \pi^{-1}(A) }$ for all Borel sets $A \subseteq \widetilde{X}$.

\subsection{Thurston maps}%
\label{sub:Thurston_maps}

In this subsection, we go over some key concepts and results on Thurston maps, and expanding Thurston maps in particular. 
For a more thorough treatment of the subject, we refer to \cite{bonk2017expanding,li2017ergodic}.

Let $S^2$ denote an oriented topological $2$-sphere and $f \colon S^2 \mapping S^2$ be a branched covering map. 
We denote by $\deg_f(x)$ the local degree of $f$ at $x \in S^2$.
The \emph{degree} of $f$ is $\deg{f} = \sum_{x\in f^{-1}(y)} \deg_{f}(x)$ for $y\in S^2$ and is independent of $y$. 

A point  $x\in S^2$ is a \emph{critical point} of $f$ if $\deg_f(x) \geqslant 2$. 
The set of critical points of $f$ is denoted by $\crit{f}$. A point $y\in S^2$ is a \emph{postcritical point} of $f$ if $y = f^n(x)$ for some $x \in \crit{f}$ and $n\in \n$. 
The set of postcritical points of $f$ is denoted by $\post{f}$. 
We observe that $\post{f} = \post{f^{n}}$ for all $n \in \n$.

\begin{definition}[Thurston maps]
    A Thurston map is a branched covering map $f \colon S^2 \mapping S^2$ on $S^2$ with $\deg f \geqslant 2$ and $\card{ \post{f} }  < +\infty$.
\end{definition}

We now recall the notation for cell decompositions of $S^2$ as used in \cite{bonk2017expanding} and \cite{li2017ergodic}. A \emph{cell of dimension $n$} in $S^2$, $n \in \{1, \, 2\}$, is a subset $c \subseteq S^2$ that is homeomorphic to the closed unit ball $\overline{\mathbb{B}^n}$ in $\real^n$, where $\mathbb{B}^{n}$ is the open unit ball in $\real^{n}$. We define the \emph{boundary of $c$}, denoted by $\partial c$, to be the set of points corresponding to $\partial \mathbb{B}^n$ under such a homeomorphism between $c$ and $\overline{\mathbb{B}^n}$. The \emph{interior of $c$} is defined to be $\inte{c} = c \mysetminus \partial c$. For each point $x\in S^2$, the set $\{x\}$ is considered as a \emph{cell of dimension $0$} in $S^2$. For a cell $c$ of dimension $0$, we adopt the convention that $\partial c = \emptyset$ and $\inte{c} = c$. 

Let $f \colon S^2 \mapping S^2$ be a Thurston map, and $\mathcal{C}\subseteq S^2$ be a Jordan curve containing $\post{f}$. 
Then the pair $f$ and $\mathcal{C}$ induces natural cell decompositions $\mathbf{D}^n(f,\mathcal{C})$ of $S^2$, for each $n \in \n_0$, in the following way:

By the Jordan curve theorem, the set $S^2 \mysetminus \mathcal{C}$ has two connected components. We call the closure of one of them the \emph{white $0$-tile} for $(f,\mathcal{C})$, denoted by $X^0_{\white}$, and the closure of the other one the \emph{black $0$-tile} for $(f,\mathcal{C})$, denoted be $X^0_{\black}$. 
The set of $0$-\emph{tiles} is $\mathbf{X}^0(f, \mathcal{C}) \define \set[\big]{ X^0_{\black}, \, X^0_{\white} }$. 
The set of $0$-\emph{vertices} is $\mathbf{V}^0(f, \mathcal{C}) \define \post{f}$. 
We set $\overline{\mathbf{V}}^0(f, \mathcal{C}) \define \set[\big]{ \{x\} \describe x\in \mathbf{V}^0(f,\mathcal{C}) }$. 
The set of $0$-\emph{edges} $\mathbf{E}^0(f,\mathcal{C})$ consists of the closures of the connected components of $\mathcal{C} \mysetminus \post{f}$. 
Then we get a cell decomposition\[
    \mathbf{D}^0(f,\mathcal{C}) \define \mathbf{X}^0(f, \mathcal{C}) \cup \mathbf{E}^0(f,\mathcal{C}) \cup \overline{\mathbf{V}}^0(f,\mathcal{C})
\]
of $S^2$ consisting of \emph{cells of level }$0$, or $0$-\emph{cells}.

We can recursively define the unique cell decomposition $\mathbf{D}^n(f,\mathcal{C})$ for $n\in \n$, consisting of $n$-\emph{cells}, such that $f$ is cellular for $\parentheses[\big]{ \mathbf{D}^{n + 1}(f,\mathcal{C}), \mathbf{D}^n(f,\mathcal{C}) }$. 
See \cite[Lemma~5.12]{bonk2017expanding} for details. 
We denote by $\mathbf{X}^n(f,\mathcal{C})$ the set of $n$-cells of dimension 2, called $n$-\emph{tiles}; by $\mathbf{E}^n(f,\mathcal{C})$ the set of $n$-cells of dimension $1$, called $n$-\emph{edges}; by $\overline{\mathbf{V}}^n(f,\mathcal{C})$ the set of $n$-cells of dimension $0$; and by $\mathbf{V}^n(f,\mathcal{C})$ the set $\set[\big]{x \describe \{x\} \in \overline{\mathbf{V}}^n(f,\mathcal{C}) }$, called the set of $n$-\emph{vertices}. 

\smallskip

For $n\in \n_0$, we define the \emph{set of black $n$-tiles} as\[
    \mathbf{X}^n_{\black}(f,\mathcal{C}) \define \left\{ X \in \mathbf{X}^n (f,\mathcal{C}) \describe f^n(X) = X^0_{\black} \right\},
\]
and the \emph{set of white $n$-tiles} as\[
    \mathbf{X}^n_{\white}(f,\mathcal{C}) \define \left\{X\in \mathbf{X}^n(f,\mathcal{C}) \describe f^n(X) = X^0_{\white}\right\}.
\]

From now on, if the map $f$ and the Jordan curve $\mathcal{C}$ are clear from the context, we will sometimes omit $(f,\mathcal{C})$ in the notation above.

We can now give a definition of expanding Thurston maps.

\begin{definition}[Expansion]     \label{def:expanding_Thurston_maps}
    A Thurston map $f \colon S^2 \mapping S^2$ is called \emph{expanding} if there exists a metric $d$ on $S^2$ that induces the standard topology on $S^2$ and a Jordan curve $\mathcal{C} \subseteq S^2$ containing $\post{f}$ such that
    \begin{equation}    \label{eq:definition of expansion}
        \lim_{n \to +\infty} \max\{ \diameter{d}{X} \describe X \in \mathbf{X}^n(f,\mathcal{C}) \} = 0.
    \end{equation}
\end{definition}

For an expanding Thurston map $f$, we can fix a particular metric $d$ on $S^2$ called a \emph{visual metric for $f$}. 
For the existence and properties of such metrics, see \cite[Chapter~8]{bonk2017expanding}. 
For a visual metric $d$ for $f$, there exists a unique constant $\Lambda > 1$ called the \emph{expansion factor} of $d$ (see \cite[Chapter~8]{bonk2017expanding} for more details). 

\begin{remark}\label{rem:chordal metric visual metric qs equiv}
    If $f \colon \ccx \mapping \ccx$ is a rational expanding Thurston map, then a visual metric is quasisymmetrically equivalent to the chordal metric on the Riemann sphere $\ccx$ (see \cite[Theorem~18.1~(ii)]{bonk2017expanding}). 
    Here the chordal metric $\sigma$ on $\ccx$ is given by $\sigma (z, w) \define \frac{2\abs{z - w}}{\sqrt{1 + \abs{z}^2} \sqrt{1 + \abs{w}^2}}$ for all $\juxtapose{z}{w} \in \cx$, and $\sigma(\infty, z) = \sigma(z, \infty) \define \frac{2}{\sqrt{1 + \abs{z}^2}}$ for all $z \in \cx$. 
    Quasisymmetric embeddings of bounded connected metric spaces are \holder continuous (see \cite[Section~11.1 and Corollary~11.5]{heinonen2001lectures}). 
    Accordingly, the classes of \holder continuous functions on $\ccx$ equipped with the chordal metric and on $S^2 = \ccx$ equipped with any visual metric for $f$ are the same (up to a change of the \holder exponent).
\end{remark}

A Jordan curve $\mathcal{C} \subseteq S^2$ is \emph{$f$-invariant} if $f(\mathcal{C}) \subseteq \mathcal{C}$. 
We are interested in $f$-invariant Jordan curves that contain $\post{f}$, since for such a Jordan curve $\mathcal{C}$, we get a cellular Markov partition $\parentheses[\big]{ \mathbf{D}^{1}(f,\mathcal{C}), \mathbf{D}^{0}(f,\mathcal{C}) }$ for $f$.
According to Example~15.11 in \cite{bonk2017expanding}, such $f$-invariant Jordan curves containing $\post{f}$ need not exist. 
However, Bonk and Meyer \cite[Theorem~15.1]{bonk2017expanding} proved that there exists an $f^n$-invariant Jordan curve $\mathcal{C}$ containing $\post{f}$ for each sufficiently large $n$ depending on $f$. 
A slightly stronger version of this result was proved in \cite[Lemma~3.11]{li2016periodic}, and we record it below. 

\begin{lemma}[Bonk \& Meyer \cite{bonk2017expanding}; Li \cite{li2016periodic}]    \label{lem:invariant_Jordan_curve not join opposite sides}
    Let $f \colon S^2 \mapping S^2$ be an expanding Thurston map, and $\widetilde{\mathcal{C}} \subseteq S^2$ be a Jordan curve with $\post{f} \subseteq \widetilde{\mathcal{C}}$. Then there exists an integer $N(f, \widetilde{\mathcal{C}}) \in \n$ such that for each $n \geqslant N(f,\widetilde{\mathcal{C}})$ there exists an $f^n$-invariant Jordan curve $\mathcal{C}$ isotopic to $\widetilde{\mathcal{C}}$ rel.\ $\post{f}$ such that no $n$-tile in $\Tile{n}$ joins opposite sides of $\mathcal{C}$.
\end{lemma}

The phrase ``joining opposite sides'' has a specific meaning in our context.

\begin{definition}[Joining opposite sides]
    Fix a Thurston map $f$ with $\card{ \post{f} } \geqslant 3$ and an $f$-invariant Jordan curve $\mathcal{C}$ containing $\post{f}$. A set $K \subseteq S^2$ \emph{joins opposite sides} of $\mathcal{C}$ if $K$ meets two disjoint $0$-edges when $\card{ \post{f} } \geqslant 4$, or $K$ meets all three $0$-edges when $\card{ \post{f} } = 3$.
\end{definition}

Recall that $\card{ \post{f} } \geqslant 3$ for each expanding Thurston map $f$ \cite[Lemma~6.1]{bonk2017expanding}.

\subsection{Symbolic dynamics for expanding Thurston maps}%
\label{sub:Symbolic dynamics for expanding Thurston maps}

In this subsection, we briefly review the dynamics of one-sided subshifts of finite type.
We refer the reader to \cite{kitchens1997symbolic} for a beautiful introduction to symbolic dynamics.
For a discussion of results on subshifts of finite type in our context, see \cite{pollicottZetaFunctionsPeriodic1990,baladi2000positive}.
    
\def\sft{\Sigma^{+}_{A}}    \def\sopt{\sigma_{A}}  
\def\sequ#1{\{ #1_{i} \}_{i \in \n_0}}
Let $S$ be a finite non-empty set and $A \colon S \times S \mapping \{0, \, 1\}$ be a matrix whose entries are either $0$ or $1$.
We denote the \emph{set of admissible sequences defined by $A$} by\[
    \sft \define \set{ \{x_{i}\}_{i \in \n_0} \describe x_{i} \in S, \, A(x_{i}, x_{i + 1}) = 1, \, \text{for each } i \in \n_0 }.
\]
Given $\metricexpshiftspace \in (0, 1)$, we equip the set $\sft$ with a metric $\metriconshiftspace$ given by
\begin{equation}    \label{eq:def:metric on shift space}
    \metriconshiftspace( \{x_{i}\}_{i \in \n_{0}}, \{y_{i}\}_{i \in \n_{0}} ) = \metricexpshiftspace^{m}  \qquad \text{for } \{x_{i}\}_{i \in \n_{0}} \ne \{y_{i}\}_{i \in \n_{0}},
\end{equation}
where $m$ is the smallest non-negative integer such that $x_{m} \ne y_{m}$.
The topology on the metric space $\parentheses[\big]{ \sft, \metriconshiftspace }$ coincides with that induced from the product topology, and is therefore compact.

The \emph{left-shift operator} $\sopt \colon \sft \mapping \sft$ (defined by $A$) is given by\[
    \sopt(\sequ{x}) = \{ x_{i + 1} \}_{i \in \n_0} \qquad \text{for } \sequ{x} \in \sft.
\]
The pair $\parentheses[\big]{ \sft, \sopt }$ is called the \emph{one-sided subshift of finite type} defined by $A$.
The set $S$ is called the \emph{set of states} and the matrix $A \colon S \times S \mapping \{0, \, 1\}$ is called the \emph{transition matrix}.
In particular, if all entries of the transition matrix $A$ are $1$, we call $\sigma \colon \Sigma^{+} \mapping \Sigma^{+}$ a \emph{one-sided shift map} on $\card{S}$ symbols, where $\sigma = \sigma_{A}$ and $\Sigma^{+} = \sft = S^{\n_0}$.

For a complex-valued continuous function \( \psi \in C \parentheses[\big]{ \sft, \cx } \), the \emph{Ruelle operator} \( \mathcal{L}_{\psi} \colon C \parentheses[\big]{ \sft, \cx } \mapping C \parentheses[\big]{ \sft, \cx } \) is defined by 
\begin{equation} \label{eq:def:ruelle operator symbolic shift}
    (\mathcal{L}_{\psi} u)(x) \define \sum_{y \in \sopt^{-1}(x)} e^{\psi(y)} u(y)
\end{equation}
for \( u \in C \parentheses[\big]{ \sft, \cx } \) and \( x \in \sft \).

We say that a one-sided subshift of finite type $\parentheses[\big]{ \sft, \sopt }$ is \textit{topologically mixing} if there exists $N \in \n$ such that $A^n(x, y) > 0$ for all $n \geqslant N$ and $\juxtapose{x}{y} \in S$.

Let $X$ and $Y$ be topological spaces, and $G \colon X \mapping X$ and $g \colon Y \mapping Y$ be continuous maps.
We say that the topological dynamical system $(Y, g)$ is a \emph{factor} of the topological dynamical system $(X, G)$ if there is a surjective continuous map $\pi \colon X \mapping Y$ such that $\pi \circ G = g \circ \pi$.
Such a map $\pi \colon X \mapping Y$ is called a \emph{semi-conjugacy} or \emph{factor map}.

The following proposition, combining results from \cite[Propositions~3.31 and 5.5]{li2024prime:dirichlet}, constructs a \holder continuous semi-conjugacy from a one-sided subshift of finite type to an expanding Thurston map, using a coding based on tiles.

\begin{proposition}[Li \& Zheng \cite{li2024prime:dirichlet}]    \label{prop:one-sided subshift of finite type associated with expanding Thurston map}
    Let $f \colon S^2 \mapping S^2$ be an expanding Thurston map with a Jordan curve $\mathcal{C} \subseteq S^2$ satisfying $f(\mathcal{C}) \subseteq \mathcal{C}$ and $\post{f} \subseteq \mathcal{C}$. 
    Let $d$ be a visual metric on $S^2$ for $f$.
    We set $S_{\vartriangle} \define \Tile{1}$, and define a transition matrix $A_{\vartriangle} \colon S_{\vartriangle} \times S_{\vartriangle} \mapping \{0, \, 1\}$ by
    \[
        A_{\vartriangle}(X, X') \define
        \begin{cases}
            1 & \text{if } f(X) \supseteq X', \\
            0 & \text{otherwise}
        \end{cases} 
    \]
    for $\juxtapose{X}{X'} \in S_{\vartriangle}$.
    Then the dynamical system $\parentheses[\big]{ S^2, f }$ is a factor of the one-sided subshift of finite type $\parentheses[\big]{ \Sigma^{+}_{A_{\vartriangle}}, \sigma_{A_{\vartriangle}} }$ defined by the transition matrix $A_{\vartriangle}$ with a surjective and \holder continuous factor map $\pi_{\vartriangle} \colon \Sigma^{+}_{A_{\vartriangle}} \mapping S^2$ given by
    \begin{equation}    \label{eq:prop:one-sided subshift of finite type associated with expanding Thurston map:factor map}
        \pi_{\vartriangle}(\sequ{X}) = x, \quad \text{where } \{x\} = \bigcap_{i \in \n_0} f^{-i}(X_{i}).
    \end{equation}
    Here $\Sigma^{+}_{A_{\vartriangle}}$ is equipped with the metric $\metriconshiftspace$ defined in \eqref{eq:def:metric on shift space} for some constant $\metricexpshiftspace \in (0, 1)$, and $S^2$ is equipped with the visual metric $d$.
    Moreover, $\parentheses[\big]{ \Sigma^{+}_{A_{\vartriangle}}, \sigma_{A_{\vartriangle}} }$ is topologically mixing, $\pi_{\vartriangle}$ is injective on $\pi_{\vartriangle}^{-1}\parentheses[\big]{ S^2 \mysetminus \bigcup_{i \in \n_0} f^{-i}(\mathcal{C})}$, and $P(\sigma_{A_{\vartriangle}}, \psi \circ \pi_{\vartriangle}) = P(f, \psi)$ for each $\psi \in \holderspacesphere$.
\end{proposition}

\begin{remark}\label{rem:primitivity of transition matrix for expanding Thurston map}
    By \cite[Lemma~5.10]{li2018equilibrium}, the transition matrix $A_{\vartriangle}$ in Proposition~\ref{prop:one-sided subshift of finite type associated with expanding Thurston map} is primitive.
    Recall that a matrix $A$ is called primitive if there exists $n \in \n$ such that all entries of $A^{n}$ are positive.
\end{remark}

The symbolic model in Proposition~\ref{prop:one-sided subshift of finite type associated with expanding Thurston map} is based on coding by tiles.
We next recall from~\cite{li2024prime:dirichlet} the construction of two subshifts of finite type based on the dynamics on an invariant Jordan curve.

Let $f \colon S^2 \to S^2$ be an expanding Thurston map with a Jordan curve $\mathcal{C} \subseteq S^2$ satisfying $f(\mathcal{C}) \subseteq \mathcal{C}$ and $\mathrm{post}\,f \subseteq \mathcal{C}$.
Define the set of states $\mathcal{S}_{\singleedge} \define \{e \in \mathbf{E}^1(f, \mathcal{C}) \describe e \subseteq \mathcal{C}\}$ and the transition matrix $\ematrix \colon \mathcal{S}_{\singleedge} \times \mathcal{S}_{\singleedge} \to \{0, 1\}$ by
\begin{equation} \label{eq:def:A_I_transition}
    \ematrix(e_1, e_2) = 
    \begin{cases}
        1 & \text{if } f(e_1) \supseteq e_2; \\
        0 & \text{otherwise}
    \end{cases}
\end{equation}
for $e_1, e_2 \in \mathcal{S}_{\singleedge}$.

Define the set of states $\mathcal{S}_{\doubleedge} \define \set[\big]{ (e, c) \in \mathbf{E}^1(f, \mathcal{C}) \times \colours \describe e \subseteq \mathcal{C} }$. 
For each $(e, c) \in \mathcal{S}_{\doubleedge}$, we denote by $X^1(e, c) \in \mathbf{X}^1(f, \mathcal{C})$ the unique $1$-tile satisfying\footnote{The existence and uniqueness of such a tile $X^1(e, c)$ defined by \eqref{eq:def:1-tile_definition_for_edge_symbolic_factors} follow immediately from the properties of the cell decomposition (cf.~\cite[Proposition~5.16~(iii),~(v),~and~(vi)]{bonk2017expanding}) and the assumptions that $f(\mathcal{C}) \subseteq \mathcal{C}$ and $e \subseteq \mathcal{C}$.}
\begin{equation} \label{eq:def:1-tile_definition_for_edge_symbolic_factors}
    e \subseteq X^1(e, c) \subseteq X_c^0.
\end{equation}
We define the transition matrix $\eematrix \colon \mathcal{S}_{\doubleedge} \times \mathcal{S}_{\doubleedge} \to \{0, 1\}$ by
\begin{equation*} \label{eq:def:A_II_transition}
    \eematrix((e_1, c_1), (e_2, c_2)) = 
    \begin{cases}
        1 & \text{if } f(e_1) \supseteq e_2 \text{ and } f \parentheses[\big]{ X^1(e_1, c_1) }  \supseteq X^1(e_2, c_2); \\
        0 & \text{otherwise}
    \end{cases}
\end{equation*}
for $(e_1, c_1), (e_2, c_2) \in \mathcal{S}_{\doubleedge}$.

We will consider the one-sided subshift of finite type $\parentheses[\big]{ \Sigma_{\ematrix}^+, \sigma_{\ematrix} }$ defined by the transition matrix $\ematrix$, and $\parentheses[\big]{ \Sigma_{\eematrix}^+, \sigma_{\eematrix} }$ defined by the transition matrix $\eematrix$, where
\begin{align*}
    \Sigma_{\ematrix}^+ &= \{ \{e_i\}_{i \in \n_0} \describe e_i \in \mathcal{S}_{\singleedge} \text{ and } \ematrix(e_i, e_{i+1}) = 1 \text{ for each } i \in \n_0 \}, \\
    \Sigma_{\eematrix}^+ &= \{ \{(e_i, c_i)\}_{i \in \n_0} \describe (e_i, c_i) \in \mathcal{S}_{\doubleedge} \text{ and } \eematrix((e_i, c_i), (e_{i+1}, c_{i+1})) = 1 \text{ for each } i \in \n_0 \},
\end{align*}
and the maps $\sigma_{\ematrix} \colon \Sigma_{\ematrix}^+ \to \Sigma_{\ematrix}^+$ and $\sigma_{\eematrix} \colon \Sigma_{\eematrix}^+ \to \Sigma_{\eematrix}^+$ are the corresponding left-shift operators.

The following proposition from \cite[Proposition~6.1]{li2024prime:dirichlet} shows that the dynamics of $f$ on the invariant Jordan curve $\mathcal{C}$ is a factor of the edge-based symbolic dynamics defined above.

\begin{proposition}[Li \& Zheng \cite{li2024prime:dirichlet}] \label{prop:edge_symbolic_factors_properties}
    Let $f \colon S^2 \mapping S^2$ be an expanding Thurston map with a Jordan curve $\mathcal{C} \subseteq S^2$ satisfying $f(\mathcal{C}) \subseteq \mathcal{C}$ and $\mathrm{post}\,f \subseteq \mathcal{C}$.
    Let $d$ be a visual metric on $S^2$ for $f$.
    Let $\parentheses[\big]{ \Sigma_{A_{\vartriangle}}^+, \sigma_{A_{\vartriangle}} }$ be the one-sided subshift of finite type associated to $f$ and $\mathcal{C}$ defined in Proposition~\ref{prop:one-sided subshift of finite type associated with expanding Thurston map}, and let $\pi_{\vartriangle} \colon \Sigma_{A_{\vartriangle}}^+ \to S^2$ be the factor map defined in~\eqref{eq:prop:one-sided subshift of finite type associated with expanding Thurston map:factor map}.
    Fix $\metricexpshiftspace \in (0, 1)$ and equip the spaces $\Sigma_{\ematrix}^+$ and $\Sigma_{\eematrix}^+$ with the metric $\metriconshiftspace$ defined in \eqref{eq:def:metric on shift space}.
    We write $\mathbf{V}(f, \mathcal{C}) \define \bigcup_{i \in \n_0} \mathbf{V}^i(f, \mathcal{C})$.
    Then the following statements hold:
    \begin{enumerate}
        \smallskip

        \item 
            $\parentheses[\big]{ \Sigma_{\ematrix}^+, \sigma_{\ematrix} }$ is a factor of $\parentheses[\big]{ \Sigma_{\eematrix}^+, \sigma_{\eematrix} }$ with a Lipschitz continuous factor map $\eecoding \colon \Sigma_{\eematrix}^+ \to \Sigma_{\ematrix}^+$ defined by
            \begin{equation} \label{eq:pi_II_def}
                \eecoding(\{ (e_i, c_i) \}_{i \in \n_0}) = \{ e_i \}_{i \in \n_0}
            \end{equation}
            for $\{ (e_i, c_i) \}_{i \in \n_0} \in \Sigma_{\eematrix}^+$.
            Moreover, for each $\{ e_i \}_{i \in \n_0} \in \Sigma_{\ematrix}^+$, we have
            \[
                \card[\big]{\eecoding^{-1}(\{ e_i \}_{i \in \n_0})} = 2.
            \]
        
        \smallskip

        \item 
            $\parentheses{ \mathcal{C}, f|_{\mathcal{C}} }$ is a factor of $\parentheses[\big]{ \Sigma_{\ematrix}^+, \sigma_{\ematrix} }$ with a \holder continuous factor map $\ecoding \colon \Sigma_{\ematrix}^+ \to \mathcal{C}$ defined by
            \begin{equation} \label{eq:pi_I_def}
                \ecoding(\{ e_i \}_{i \in \n_0}) = x, \quad \text{where } \{x\} = \bigcap_{i \in \n_0} f^{-i}(e_i)
            \end{equation}
            for $\{ e_i \}_{i \in \n_0} \in \Sigma_{\ematrix}^+$.
            Moreover, for each $x \in \mathcal{C}$, we have
            \begin{equation} \label{eq:card_pi_I_inv}
                \card[\big]{ \ecoding^{-1}(x) } = 
                \begin{cases}
                    1 & \text{if } x \in \mathcal{C} \mysetminus \mathbf{V}(f, \mathcal{C}), \\
                    2 & \text{if } x \in \mathcal{C} \cap \mathbf{V}(f, \mathcal{C}).
                \end{cases}
            \end{equation}
            Thus, we have the following commutative diagram:
            \[
            \begin{tikzcd}
                \Sigma_{\eematrix}^+ \arrow[r, "\eecoding"] \arrow[d, "\sigma_{\eematrix}"] & \Sigma_{\ematrix}^+ \arrow[r, "\ecoding"] \arrow[d, "\sigma_{\ematrix}"] & \mathcal{C} \arrow[d, "f|_{\mathcal{C}}"] \\
                \Sigma_{\eematrix}^+ \arrow[r, "\eecoding"] & \Sigma_{\ematrix}^+ \arrow[r, "\ecoding"] & \mathcal{C}.
            \end{tikzcd}
            \]
    \end{enumerate}
\end{proposition}

\subsection{Ergodic theory of expanding Thurston maps}%
\label{sub:Ergodic theory of expanding Thurston maps}

We summarize the existence, uniqueness, and basic properties of equilibrium states for expanding Thurston maps in the following theorem.

\begin{theorem}[Li \cite{li2018equilibrium}]     \label{thm:properties of equilibrium state}
    Let $f \colon S^2 \mapping S^2$ be an expanding Thurston map and $d$ a visual metric on $S^2$ for $f$. 
    Let $\juxtapose{\phi}{\gamma} \in C^{0,\holderexp} \parentheses[\big]{ S^2,d }$ be real-valued \holder continuous functions with an exponent $\holderexp \in (0,1]$. 
    Then the following statements hold:
    \begin{enumerate}[label=\rm{(\roman*)}]
        \smallskip
        
        \item     \label{item:thm:properties of equilibrium state:existence and uniqueness}
        There exists a unique equilibrium state $\mu_{\phi}$ for the map $f$ and the potential $\phi$.

        \smallskip
        
        \item     \label{item:thm:properties of equilibrium state:derivative}
        For each $t \in \real$, we have $\frac{\mathrm{d}}{\mathrm{d}t}P(f,\phi + t \gamma) = \int \! \gamma \,\mathrm{d}\mu_{\phi + t\gamma}$.

        \smallskip

        \item     \label{item:thm:properties of equilibrium state:cohomologous}
        Let $\mu_{\gamma}$ be the unique equilibrium state for $f$ and $\gamma$.
        Then $\mu_{\phi} = \mu_{\gamma}$ if and only if there exist a constant $K \in \real$ and a continuous function $u \in C \parentheses[\big]{ S^2 }$ such that $\phi - \gamma = K + u \circ f - u$. 
    \end{enumerate}
\end{theorem}

Let $T \colon X \mapping X$ be a map on a topological space $X$ and $\psi \colon X \mapping \real$ be a real-valued function on $X$.
We say that $\psi$ is \emph{cohomologous to a constant in $C(X)$} if there exist $C \in \real$ and $u \in C(X)$ such that $\psi = C + u \circ T - u$.

\begin{lemma} \label{lem:tile symbolic coding preserves entropy of equilibrium state}
    Under the assumptions and notation of Proposition~\ref{prop:one-sided subshift of finite type associated with expanding Thurston map}, let $\potential \in \holderspacesphere$ and $\widetilde{\mu} \in \mathcal{M}\parentheses[\big]{ \Sigma^{+}_{A_{\vartriangle}}, \sigma_{A_{\vartriangle}} }$ be the equilibrium state for the map $\sigma_{A_{\vartriangle}}$ and the potential $\potential \circ \pi_{\vartriangle}$. 
    Denote $\mu \define (\pi_{\vartriangle})_{*}\widetilde{\mu}$.
    Then $h_{\mu}(f) = h_{\widetilde{\mu}}(\sigma_{A_{\vartriangle}})$ and $\mu$ is an equilibrium state for $f$ and $\potential$.
    Moreover, $\potential$ is cohomologous to a constant in $C \parentheses[\big]{ S^2 }$ with respect to $f$ if and only if $\potential \circ \pi_{\vartriangle}$ is cohomologous to a constant in $C \parentheses[\big]{ \Sigma^{+}_{A_{\vartriangle}} }$ with respect to $\sigma_{A_{\vartriangle}}$.
\end{lemma}
\begin{proof}
    Let $E \define \bigcup_{i \in \n_0} f^{-i}(\mathcal{C})$ and $\widetilde{E} \define \pi_{\vartriangle}^{-1}\parentheses{ E }$.
    To deduce that $h_{\mu}(f) = h_{\widetilde{\mu}}(\sigma_{A_{\vartriangle}})$, it suffices to show $\mu(E) = \widetilde{\mu}(\widetilde{E}) = 0$.
    Indeed, by Proposition~\ref{prop:one-sided subshift of finite type associated with expanding Thurston map}, the restriction of $\pi_{\vartriangle}$ to the subset $\pi_{\vartriangle}^{-1}\parentheses[\big]{ S^2 \mysetminus E} = \Sigma^{+}_{A_{\vartriangle}} \mysetminus \widetilde{E}$ is a bijection from $\Sigma^{+}_{A_{\vartriangle}} \mysetminus \widetilde{E}$ onto $S^2 \mysetminus E$.
    If $\mu(E) = \widetilde{\mu}(\widetilde{E}) = 0$, then $\pi_{\vartriangle}$ is a measurable isomorphism between the systems $\parentheses[\big]{ \Sigma_{A_{\vartriangle}}^+, \sigma_{A_{\vartriangle}},\widetilde{\mu} }$ and $\parentheses[\big]{ S^2, f, \mu }$ (modulo sets of measure zero). 
    Since measure-theoretic entropy is invariant under measurable isomorphism, it follows that $h_{\mu}(f) = h_{\widetilde{\mu}}(\sigma_{A_{\vartriangle}})$.

    We argue by contradiction and suppose that $\mu(E) > 0$.
    The equilibrium state $\widetilde{\mu}$ for the topologically mixing subshift $\parentheses[\big]{ \Sigma^{+}_{A_{\vartriangle}}, \sigma_{A_{\vartriangle}} }$ is ergodic (cf.~\cite[Proposition~1.14]{bowen1975equilibrium}).
    Since $\pi_{\vartriangle}$ is a factor map, the push-forward measure $\mu$ is also ergodic.
    Since $f(\mathcal{C}) \subseteq \mathcal{C}$, we have $f^{-1}(E) = E$.
    Then ergodicity implies $\mu(E) = 1$. 
    Furthermore, we can show that $\mu(\mathcal{C}) = 1$.
    Indeed, since $\mu$ is $f$-invariant and $f(\mathcal{C}) \subseteq \mathcal{C}$, we have $\mu(f^{-n}(\mathcal{C})) = \mu(\mathcal{C})$ and $\mathcal{C} \subseteq f^{-n}(\mathcal{C})$ for all $n \in \n_{0}$, which imply that $\mu(f^{-n}(\mathcal{C}) \smallsetminus \mathcal{C}) = 0$. 
    Thus $\mu(E \mysetminus \mathcal{C}) = 0$, i.e., $\mu(\mathcal{C}) = 1$.
    Consequently, $\widetilde{\mu}\parentheses[\big]{ \pi_{\vartriangle}^{-1}(\mathcal{C}) } = 1$. 
    This leads to a contradiction, as $\pi_{\vartriangle}^{-1}(\mathcal{C})$ is a closed proper subset of $\Sigma^{+}_{A_{\vartriangle}}$, while the measure $\widetilde{\mu}$ must be positive on every non-empty open set (as a consequence of the Gibbs property \cite[Theorem 1.2]{bowen1975equilibrium}).
    Therefore, we conclude that $0 = \mu(E) = \widetilde{\mu} \parentheses[\big]{ \widetilde{E} }$.
    
    We now verify that $\mu$ is an equilibrium state for $f$ and $\potential$. 
    Since $\mu = (\pi_{\vartriangle})_{*}\widetilde{\mu}$, we have $\int \! \potential \,\mathrm{d}\mu = \int \! \potential \circ \pi_{\vartriangle} \,\mathrm{d} \widetilde{\mu}$.
    Using the entropy preservation $h_{\mu}(f) = h_{\widetilde{\mu}}(\sigma_{A_{\vartriangle}})$ and the pressure equality $P(\sigma_{A_{\vartriangle}}, \potential \circ \pi_{\vartriangle}) = P(f, \potential)$ from Proposition~\ref{prop:one-sided subshift of finite type associated with expanding Thurston map}, we obtain that
    \[
        P(f, \potential) = P(\sigma_{A_{\vartriangle}}, \potential \circ \pi_{\vartriangle}) 
        = h_{\widetilde{\mu}}(\sigma_{A_{\vartriangle}}) + \int \! \potential \circ \pi_{\vartriangle} \,\mathrm{d} \widetilde{\mu} 
        = h_{\mu}(f) + \int \! \potential \,\mathrm{d}\mu.
    \]
    This implies that $\mu$ is an equilibrium state for $f$ and $\potential$.

    Finally, we establish the equivalence of the cohomological statements. 
    The forward direction is straightforward. 
    For the converse, assume that $\potential \circ \pi_{\vartriangle}$ is cohomologous to a constant in $C \parentheses[\big]{ \Sigma^{+}_{A_{\vartriangle}} }$. 
    By Proposition~\ref{prop:one-sided subshift of finite type associated with expanding Thurston map}, the subshift $\parentheses[\big]{ \Sigma^{+}_{A_{\vartriangle}}, \sigma_{A_{\vartriangle}} }$ is topologically mixing. 
    It is a classical result (see \cite[Theorem~1.28]{bowen1975equilibrium}) that for such a system, the equilibrium state of a potential cohomologous to a constant is the measure of maximal entropy. Thus, $\widetilde{\mu}$ is the measure of maximal entropy for $\sigma_{A_{\vartriangle}}$.
    The preservation of entropy implies that its pushforward $\mu = (\pi_{\vartriangle})_{*} \tilde{\mu}$ is a measure of maximal entropy for $f$.
    Since $\mu$ is an equilibrium state for $f$ and $\potential$, it follows from Theorem~\ref{thm:properties of equilibrium state}~\ref{item:thm:properties of equilibrium state:cohomologous} that $\potential$ is cohomologous to a constant in $C \parentheses[\big]{ S^2 }$.
    This completes the proof.
\end{proof}

The potentials that satisfy the following property are of particular interest in the considerations of Prime Orbit Theorems.

\begin{definition}[Eventually positive function]    \label{def:eventually positive functions}
    Let $g \colon X \mapping X$ be a map on a set $X$, and $\varphi \colon X \mapping \real$ be a real-valued function on $X$.
    Then $\varphi$ is \emph{eventually positive} if there exists $N \in \n$ such that $S_n \varphi(x) > 0$ for each $x \in X$ and each $n \in \n$ with $n \geqslant N$. 
\end{definition}

The following result is a consequence of the properties of the pressure function. For a proof, see \cite[Corollary~3.29]{li2024prime:dirichlet}.

\begin{lemma}[Li \& Zheng \cite{li2024prime:dirichlet}] \label{lem:pressure function:monotonicity and uniqueness of zero}
    Let $f \colon S^2 \mapping S^2$ be an expanding Thurston map, and $d$ be a visual metric on $S^2$ for $f$. 
    Let $\potential \in \holderspacesphere$ be an eventually positive real-valued \holder continuous function with an exponent $\holderexp \in (0, 1]$. 
    Then the function $t \mapping P(f, -t \potential)$, $t \in \real$, is strictly decreasing and there exists a unique number $\rootpressure \in \real$ such that $P(f, -\rootpressure \potential) = 0$. 
    Moreover, $\rootpressure > 0$.
\end{lemma}

We recall the strong non-integrability condition from \cite[Subsection~7.1]{li2024prime:split}. 

\begin{definition}[Strong non-integrability condition]    \label{def:strong non-integrability condition}
    Let $f \colon S^2 \mapping S^2$ be an expanding Thurston map and $d$ be a visual metric on $S^2$ for $f$.
    Fix $\holderexp \in (0, 1]$.
    Let $\potential \in \holderspacesphere$ be a real-valued \holder continuous function with an exponent $\holderexp$.
    \begin{enumerate}[label=\rm{(\arabic*)}]
        \smallskip
        \item We say that $\potential$ satisfies the \emph{$(\mathcal{C}, \holderexp)$-strong non-integrability condition} (with respect to $f$ and $d$), for a Jordan curve $\mathcal{C} \subseteq S^2$ with $\post{f} \subseteq \mathcal{C}$, if there exist
        \begin{enumerate}[label=\rm{(\alph*)}]
            \smallskip
            \item numbers $\juxtapose{N}{M} \in \n$, $\varepsilon \in (0, 1)$, 
            \smallskip
            \item $M$-tiles $Y^{M}_{\black} \in \cTile{M}{\black}$, $Y^{M}_{\white} \in \cTile{M}{\white}$
        \end{enumerate}
        such that for each $\colour \in \colours$, each integer $m \geqslant M$, and each $m$-tile $X \in \Tile{m}$ with $X \subseteq Y^{M}_{\colour}$, there exist two points $\juxtapose{x_{1}}{x_{2}} \in X$ with the following properties:
        \begin{enumerate}
            \smallskip
            \item $\min \set[\big]{ d \parentheses[\big]{ x_{1}, S^2 \mysetminus X }, d \parentheses[\big]{ x_{2}, S^2 \mysetminus X }, d(x_{1}, x_{2}) } \geqslant \diameter{d}{X}$, and
            \smallskip
            \item for each integer $n \geqslant N$, there exist two $(n + M)$-tiles $\juxtapose{X^{n + M}_{\colour, 1}}{X^{n + M}_{\colour, 2}} \in \Tile{n + M}$ such that $Y^{M}_{\colour} = f^{n} \parentheses[\big]{ X^{n + M}_{\colour, 1} } = f^{n} \parentheses[\big]{ X^{n + M}_{\colour, 2} }$ and 
            \[
                \abs{ S_{n}\potential(\varsigma_{1}(x_{1})) - S_{n}\potential(\varsigma_{2}(x_{1})) - S_{n}\potential(\varsigma_{1}(x_{2})) + S_{n}\potential(\varsigma_{2}(x_{2})) } \geqslant \varepsilon d( x_{1}, x_{2})^{\holderexp},
            \]
            where we write $\varsigma_{i} \define \parentheses[\big]{ f^{n}|_{X^{n + M}_{\colour, i}} }^{-1}$ for each $i \in \set{1, 2}$.
        \end{enumerate}
        \smallskip

        \item We say that $\potential$ satisfies the \emph{$\holderexp$-strong non-integrability condition} (with respect to $f$ and $d$) if $\potential$ satisfies the $(\mathcal{C}, \holderexp)$-strong non-integrability condition with respect to $f$ and $d$ for some Jordan curve $\mathcal{C} \subseteq S^2$ with $\post{f} \subseteq \mathcal{C}$.
        
        \smallskip

        \item We say that $\potential$ satisfies the \emph{strong non-integrability condition} (with respect to $f$ and $d$) if $\potential$ satisfies the $\holderexp'$-strong non-integrability condition with respect to $f$ and $d$ for some $\holderexp' \in (0, \holderexp]$.
    \end{enumerate}
\end{definition}

The strong non-integrability condition is independent of the choice of the Jordan curve $\mathcal{C}$ (see \cite[Lemma~7.2]{li2024prime:split}).

\begin{definition}[Partition function]    \label{def:partition function}
    Let $g \colon X \mapping X$ be a map on a topological space $X$.
    Let $\varphi \colon X \mapping \cx$ be a complex-valued function on $X$.
    Consider $n \in \n$.
    We define a \emph{partition function} $\partifun{\cdot}[n][g][\varphi] \colon \cx \mapping \cx$ (for $(g, \varphi)$) as
    \begin{equation}    \label{eq:def:partition function}
        \partifun{s}[n][g][\varphi] \define \sum_{x \in \myperiodpoint{g}{n}} e^{s S_{n} \varphi(x)}, \quad s \in \cx.
    \end{equation}
\end{definition}

The primary tool employed in \cite{li2018equilibrium} for developing the thermodynamic formalism for expanding Thurston maps is the Ruelle operator.
For our purposes in this paper, we require certain variants of the Ruelle operator, called \emph{split Ruelle operators}, which were introduced in \cite{li2024prime:split}.
The relevant notions are recorded below.

\begin{definition}[Partial split Ruelle operator]    \label{def:partial split Ruelle operator}
    Let $f \colon S^2 \mapping S^2$ be an expanding Thurston map, $\mathcal{C} \subseteq S^2$ a Jordan curve containing $\post{f}$, and $\psi \in C(S^2, \cx)$ a complex-valued continuous function.
    Let $n \in \n_0$, and $E \subseteq S^2$ a union of $n$-tiles in $\Tile{n}$.
    We define an operator $\mathcal{L}^{(n)}_{\psi, \colour, E} \colon C(E, \cx) \mapping C \parentheses[\big]{ \colourtile, \cx } $, for each $\colour \in \colours$, by
    \begin{equation}    \label{eq:def:partial split Ruelle operator}
        \mathcal{L}^{(n)}_{\psi, \colour, E}(u)(y) \define \sum_{\substack{X^n \in \tile{n}_{\colour} \\ X^n \subseteq E}} u\parentheses[\big]{ (f^n|_{X^n})^{-1}(y) } \myexp[\big]{ S_n\psi \parentheses[\big]{ (f^n|_{X^n})^{-1}(y) } }, 
    \end{equation}
    for $u \in C(E, \cx)$ and $y \in \colourtile$.
    When $E = \ccolourtile$ for some $\ccolour \in \colours$, we often write \[
        \mathcal{L}^{(n)}_{\psi, \colour, \ccolour} \define \mathcal{L}^{(n)}_{\psi, \colour, \ccolourtile}.
    \]
\end{definition}

For each $\colour \in \colours$, we define the projection $\pi_{\colour} \colon C \parentheses[\big]{ \blacktile, \cx } \times C \parentheses[\big]{ \whitetile, \cx } \mapping C \parentheses[\big]{ \colourtile, \cx }$ by
\[
    \pi_{\colour}\splfun \define u_{\colour}, \qquad \text{for } \splfun \in C \parentheses[\big]{ \blacktile, \cx } \times C \parentheses[\big]{ \whitetile, \cx }.
\]

\begin{definition}[Split Ruelle operators]    \label{def:split ruelle operator}
    Let $f \colon S^2 \mapping S^2$ be an expanding Thurston map with a Jordan curve $\mathcal{C} \subseteq S^2$ satisfying $f(\mathcal{C}) \subseteq \mathcal{C}$ and $\post{f} \subseteq \mathcal{C}$. 
    Let $d$ be a visual metric for $f$ on $S^2$, and $\psi \in C^{0, \holderexp}((S^2, d), \cx)$ a complex-valued \holder continuous function with an exponent $\holderexp \in (0, 1]$.
    The \emph{split Ruelle operator} $\mathbb{L}_{\psi} \colon C \parentheses[\big]{ X^0_{\black} } \times C \parentheses[\big]{ X^0_{\white} } \mapping C \parentheses[\big]{ X^0_{\black} } \times C \parentheses[\big]{ X^0_{\white} }$ is defined by
    \begin{equation}    \label{eq:def:split ruelle operator}
        \mathbb{L}_{\psi}\splfun \define \parentheses[\big]{ 
            \paroperator[\varphi]{1}{\black}{\black}(u_{\black}) + \paroperator[\varphi]{1}{\black}{\white}(u_{\white}), 
            \paroperator[\varphi]{1}{\white}{\black}(u_{\black}) + \paroperator[\varphi]{1}{\white}{\white}(u_{\white})
        }
    \end{equation}
    for $u_{\black} \in C \parentheses[\big]{ X^0_{\black}, \cx}$ and $u_{\white} \in C \parentheses[\big]{ X^0_{\white}, \cx}$.

    For all $n \in \n_0$, we write the operator norm
    \[
        \sploptnorm{\mathbb{L}^{n}_{\psi}}
        \define \sup \set[\Big]{ \norm[\big]{ \pi_{\colour}\parentheses[\big]{ \mathbb{L}^{n}_{\psi}\splfun } }_{C^{0, \holderexp}} \describe
        \colour \in \colours, \, u_{\ccolour} \in C^{0, \holderexp}\parentheses[\big]{ (\ccolourtile, d), \cx }, \, \norm{u_{\ccolour}}_{C^{0, \holderexp}} \leqslant 1, \, \ccolour \in \colours }.
    \]

\end{definition}

%% file: section/Assumptions.tex
\section{The Assumptions}
\label{sec:The Assumptions}

We state below the hypotheses under which we will develop our theory in most parts of this paper. We will repeatedly refer to such assumptions in the later sections. We emphasize again that not all assumptions are assumed in all the statements in this paper.

\begin{assumptions}
\quad
\begin{enumerate}[label=\textrm{(\arabic*)}]
    \smallskip

    \item \label{assumption:expanding Thurston map}
        $f \colon S^2 \mapping S^2$ is an expanding Thurston map.

    \smallskip

    \item \label{assumption:Jordan curve}
        $\mathcal{C} \subseteq S^2$ is a Jordan curve containing $\post{f}$ with the property that there exists an integer $n_{\mathcal{C}} \in \n$ such that $f^{n_{\mathcal{C}}}(\mathcal{C}) \subseteq \mathcal{C}$ and $f^m(\mathcal{C}) \not\subseteq \mathcal{C}$ for each $m \in \{ 1,2,\dots, n_{\mathcal{C}}-1 \}$.

    \smallskip

    \item \label{assumption:visual metric and expansion factor}
        $d$ is a visual metric on $S^2$ for $f$ with expansion factor $\Lambda > 1$. 

    \smallskip

    \item \label{assumption:holder exponent}
        $\beta \in (0, 1]$.

    \smallskip

    \item \label{assumption:holder potential}
        $\potential \in C^{0,\beta}(S^2,d)$ is an eventually positive real-valued \holder continuous function with exponent $\holderexp$.

    \item \label{assumption:unique zero}
        $\rootpressure \in \real$ is the unique positive real number satisfying $P(f, - \rootpressure \potential) = 0$.
    
    \smallskip
    
    \item \label{assumption:average}
        $\alpha \define \frac{\mathrm{d}}{\mathrm{d} t} P(f, t \potential) |_{t = -\rootpressure}$.  

    \smallskip

    \item \label{assumption:equilibrium state} 
        $\mu_{\phi}$ is the unique equilibrium state for the map $f$ and the potential $\phi$. 
\end{enumerate}
\end{assumptions}
    
Note that the uniqueness of $\rootpressure$ in \ref{assumption:unique zero} is guaranteed by Lemma~\ref{lem:pressure function:monotonicity and uniqueness of zero}. Furthermore, it follows from Theorem~\ref{thm:properties of equilibrium state}~\ref{item:thm:properties of equilibrium state:derivative} and Definition~\ref{def:eventually positive functions} that 
\[
    \alpha = \int \! \potential \,\mathrm{d}\mu_{-\rootpressure \potential} > 0.
\]
For a pair of $f$ in \ref{assumption:expanding Thurston map} and $\potential$ in \ref{assumption:holder potential}, we will say that a quantity depends on $f$ and $\potential$ if it depends on $\rootpressure$.

Observe that by Lemma~\ref{lem:invariant_Jordan_curve not join opposite sides}, for each $f$ satisfying (1), there exists at least one Jordan curve $\mathcal{C}$ satisfying (2). 
Since for a fixed $f$, the number $n_{\mathcal{C}}$ is uniquely determined by $\mathcal{C}$ in $(2)$, in the remaining part of the paper we will say that a quantity depends on $\mathcal{C}$ even if it also depends on $n_{\mathcal{C}}$.

Recall that the expansion factor $\Lambda$ of a visual metric $d$ on $S^2$ for $f$ is uniquely determined by $d$ and $f$. We will say that a quantity depends on $f$ and $d$ if it depends on $\Lambda$.

In the discussion below, depending on the conditions we will need, we will sometimes say ``Let $f$, $\mathcal{C}$, $d$, $\phi$ satisfy the Assumptions.'', and sometimes say ``Let $f$ and $\mathcal{C}$ satisfy the Assumptions.'', etc.



%% file: section/Thermodynamic_formalism.tex

\section{Pressure function and partition function estimates}
\label{sec:Pressure function and partition function estimates}

In this section, we employ thermodynamic formalism and Ruelle operators to study the dynamics of expanding Thurston maps and derive decay estimates needed for the main theorem.
We first investigate some differential and analytic properties of the topological pressure function in Subsection~\ref{sub:The topological pressure function}, then in Subsection~\ref{sub:Decay estimates of the partition function} we establish several decay estimates associated with the Ruelle operators and partition functions.

\subsection{The topological pressure function}%
\label{sub:The topological pressure function}

We first introduce some terminology and then discuss properties of the topological pressure function.


The following proposition establishes the analyticity and second derivatives of the topological pressure function in suitable function directions by using the symbolic coding constructed in Proposition~\ref{prop:one-sided subshift of finite type associated with expanding Thurston map} and classical results from symbolic dynamics.

\begin{proposition}    \label{prop:property of pressure function and derivative}
    Let $f$ and $d$ satisfy the Assumptions in Section~\ref{sec:The Assumptions}. 
    Let $\potential$, $u$, and $v$ be \holder continuous functions on $S^2$ with respect to $d$.
    Let $\mu_{\potential}$ be the unique equilibrium state for $f$ and $\potential$.
    Then the following statements hold:
    \begin{enumerate}
        \smallskip

        \item     \label{item:prop:property of pressure function and derivative:analytic}
            The function $P(f, \potential + \cdot \, u) \colon \real \mapping \real$ is real-analytic.
            Moreover, there exists an open set $W \subseteq \cx$ with $\real \subseteq W$ such that $P(f, \potential + \cdot \, u)$ extends to a complex-analytic function on $W$.

        \smallskip

        \item     \label{item:prop:property of pressure function and derivative:second derivative}
            \begin{equation} \label{eq:item:prop:property of pressure function and derivative:second derivative:second derivative of pressure function}
                \frac{\partial^2}{\partial s \partial t} P(f, \potential + t u + s v) \Big|_{t = s = 0} = \lim_{n \to +\infty} \frac{1}{n} \int \! \parentheses[\bigg]{ S_{n}u - n \int \! u \,\mathrm{d}\mu_{\potential} } \parentheses[\bigg]{ S_{n}v - n \int \! v \,\mathrm{d}\mu_{\potential} }  \,\mathrm{d} \mu_{\potential}.
            \end{equation}
            In particular, \[
                \sigma^{2}_{\mu_{\potential}}(u) \define
                \frac{\mathrm{d}^2}{\mathrm{d} t^2} P(f, \potential + t u) \Big|_{t = 0} = \lim_{n \to +\infty} \frac{1}{n} \int \! \parentheses[\bigg]{ S_{n}u - n\int \! u \,\mathrm{d}\mu_{\potential} }^{2} \,\mathrm{d} \mu_{\potential}.
            \]
            Moreover, $\sigma^{2}_{\mu_{\potential}}(u) = 0$ if and only if there exists $w \in C(S^2)$ such that $u = \int \! u \,\mathrm{d}\mu_{\potential} + w \circ f - w$.
    \end{enumerate}
\end{proposition}
\begin{proof}
    We reduce the general case to one where the map possesses an invariant Jordan curve, thereby permitting the use of symbolic dynamics.
    By Lemma~\ref{lem:invariant_Jordan_curve not join opposite sides}, we may take a sufficiently high iterate $F \define f^K$ that admits an $F$-invariant Jordan curve $\mathcal{C} \subseteq S^2$ containing the postcritical set $\post{F} = \post{f}$. 
    The map $F$ is then an expanding Thurston map for which $d$ is a visual metric (cf.~\cite[Proposition~8.3~(v)]{bonk2017expanding}). 
    We fix this integer $K$, map $F$, and curve $\mathcal{C}$.

    Applying Proposition~\ref{prop:one-sided subshift of finite type associated with expanding Thurston map} to $F$ and $\mathcal{C}$, we denote the corresponding one-sided subshift of finite type by $(\Sigma^{+}, \sigma)$ and the factor map by $\pi \colon \Sigma^{+} \mapping S^2$. 
    Fix $\metricexpshiftspace \in (0, 1)$ and equip the space $\Sigma^{+}$ with the metric $\metriconshiftspace$ defined in \eqref{eq:def:metric on shift space}.
    Then $\pi \circ \sigma = F \circ \pi$, and $P(F, \psi) = P(\sigma, \psi \circ \pi)$ for each \holder continuous function $\psi \colon S^2 \mapping \real$ (with respect to $d$).
    Here the lifted potential $\psi \circ \pi \colon \Sigma^{+} \mapping \real$ is \holder continuous with respect to the metric $\metriconshiftspace$ on $\Sigma^{+}$.

    Denote $\Phi \define S_K^f \phi$, $U \define S_K^f u$, and $V \define S_K^f v$. 
    Then $\Phi$, $U$, and $V$ are \holder continuous with respect to $d$ since $f$ is Lipschitz continuous with respect to $d$ (cf.~\cite[Lemma~3.12]{li2017ergodic}).

    Let $\mu_{F,\Phi}$ be the unique equilibrium state for $F$ and $\Phi$.
    Since $P(F, \Phi) = K P(f, \phi)$ (cf.~ \cite[Theorem~9.8~(i)]{walters1982introduction}), it follows from $P_{\mu_{\phi}}(F, \Phi) = K P_{\mu_{\phi}}(f, \phi)$ and the uniqueness of the equilibrium state that $\mu_{F,\Phi} = \mu_{\phi}$. 

    \smallskip
    \ref{item:prop:property of pressure function and derivative:analytic}
    By \cite[Proposition~4.7]{pollicottZetaFunctionsPeriodic1990}, the function $P(\sigma, \Phi \circ \pi + \cdot\, U \circ \pi) \colon \real \mapping \real$ is real-analytic.
    Since $P(F, \Phi + t U) = P(\sigma, \Phi \circ \pi + t U \circ \pi)$ and $P(F, \Phi + t U) = P\parentheses[\big]{ f^K, S_K^f(\phi + tu) } = K P(f, \phi + tu)$ for all $t \in \real$, we conclude that $t \mapsto P(f, \phi + t u)$ is real-analytic, and thus extends to a complex-analytic function on a neighborhood of $\real$.

    \smallskip
    \ref{item:prop:property of pressure function and derivative:second derivative}
    We first establish the formula \eqref{eq:item:prop:property of pressure function and derivative:second derivative:second derivative of pressure function}.
    By \cite[Theorem~1.22]{bowen1975equilibrium}, there exists a unique equilibrium state $\nu_{\Phi \circ \pi} \in \invmea[\Sigma^{+}][\sigma]$ for $\sigma$ and $\Phi \circ \pi$.
    Then it follows from \cite[Theorem~5.7.4]{przytycki2010conformal} that
    \begin{align*}
        &\frac{\partial^2}{\partial s \partial t} P(\sigma, \Phi \circ \pi + t U \circ \pi + s V \circ \pi) \Big|_{t = s = 0} \\
        &\qquad = \lim_{n \to +\infty} \frac{1}{n} \int \! \parentheses[\bigg]{ S_{n}^{\sigma}(U \circ \pi) - n \int \! U \circ \pi \,\mathrm{d}\nu_{\Phi \circ \pi} } \parentheses[\bigg]{ S_{n}^{\sigma}(V \circ \pi) - n \int \! V \circ \pi \,\mathrm{d}\nu_{\Phi \circ \pi} }  \,\mathrm{d} \nu_{\Phi \circ \pi}.
    \end{align*}
    Denote $\mu \define \pi_{*}\nu_{\Phi \circ \pi} \in \invmea[S^2][F]$.
    Then by Lemma~\ref{lem:tile symbolic coding preserves entropy of equilibrium state}, we have $\mu = \mu_{F,\Phi} = \mu_{\potential}$.
    Since $P(f, \phi + tu + sv) = \frac{1}{K} P(F, \Phi + tU + sV)$ and $\int \! \varphi \,\mathrm{d}\mu = \int \! \varphi \circ \pi \,\mathrm{d} \nu_{\potential \circ \pi}$ for all $\varphi \in C(S^{2})$, we get
    \[
        \frac{\partial^2}{\partial s \partial t} P(f, \phi + tu + sv) \Big|_{t = s = 0} 
        = \frac{1}{K} \lim_{n \to +\infty} \frac{1}{n} \int \! \parentheses[\big]{ S_{n}^F \widetilde{U} } \parentheses[\big]{ S_{n}^F \widetilde{V} }  \,\mathrm{d} \mu_{\phi}
        = \lim_{n \to +\infty} \frac{1}{n K} \int \! \parentheses[\big]{ S_{nK}^f \widetilde{u} } \parentheses[\big]{ S_{nK}^f \widetilde{v} }  \,\mathrm{d} \mu_{\potential},
    \]
    where $\widetilde{U} \define U - \int \! U \,\mathrm{d}\mu_{\phi}$ and $\widetilde{u} \define u - \int \! u \,\mathrm{d}\mu_{\phi}$ (and similarly for $\widetilde{V}$ and $\widetilde{v}$).
    To establish~\eqref{eq:item:prop:property of pressure function and derivative:second derivative:second derivative of pressure function}, it suffices to verify that $\lim_{m \to +\infty} \frac{1}{m} \int \parentheses[\big]{ S_m^f \widetilde{u} } \parentheses[\big]{ S_m^f \widetilde{v} } \,\mathrm{d}\mu_{\phi}$ exists.
    
    For each $m \in \n$, we write $m = nK + r$ with $r \in \atob{0}{K - 1}$.
    We write $S_m^f \widetilde{u} = S_{nK}^f \widetilde{u} + E_{m}(\widetilde{u})$, where the error term $E_{m}(\widetilde{u}) \coloneqq S_r^f (\widetilde{u} \circ f^{nK})$ satisfies $\uniformnorm{E_{m}(\widetilde{u})} \leqslant K \uniformnorm{\widetilde{u}}$.
    Then by the Cauchy--Schwarz inequality, we obtain
    \begin{align*}
        &\abs[\bigg]{ \frac{1}{m} \int \parentheses[\big]{ S_m^f \widetilde{u} } \parentheses[\big]{ S_m^f \widetilde{v} } \,\mathrm{d}\mu_{\phi} - \frac{1}{m} \int \parentheses[\big]{ S_{nK}^f \widetilde{u} } \parentheses[\big]{ S_{nK}^f \widetilde{v} } \,\mathrm{d}\mu_{\phi} }  \\
        &\peq= \frac{1}{m} \abs[\bigg]{ \int E_m(\widetilde{v}) S_{nK}^f \widetilde{u} \,\mathrm{d}\mu_{\phi} + \int E_m(\widetilde{u}) S_{nK}^f \widetilde{v} \,\mathrm{d}\mu_{\phi} + \int E_m(\widetilde{u}) E_m(\widetilde{v}) \,\mathrm{d}\mu_{\phi} } \\
        &\peq\leqslant \frac{1}{m} \parentheses[\Big]{ K \uniformnorm{\widetilde{v}} \norm[\big]{S_{nK}^f \widetilde{u}}_{L^2(\mu_{\phi})} + K \uniformnorm{\widetilde{u}} \norm[\big]{S_{nK}^f \widetilde{v}}_{L^2(\mu_{\phi})} + K^2 \uniformnorm{\widetilde{u}} \uniformnorm{\widetilde{v}} }.
    \end{align*}    
    Since $\lim_{n \to +\infty} \frac{1}{nK} \norm[\big]{S_{nK}^f \widetilde{u}}^2_{L^2(\mu_{\phi})} = \frac{\mathrm{d}^2}{\mathrm{d} t^2} P(f, \potential + t u) \Big|_{t = 0}$, we have $\norm[\big]{S_{nK}^f \widetilde{u}}_{L^2(\mu_{\phi})} = O(\sqrt{n})$ as $n \to \infty$ (and similarly for $\widetilde{v}$).
    Since $\lim_{m \to \infty} m/(nK) = 1$, we conclude that
    \[
        \lim_{m \to +\infty} \frac{1}{m} \int \parentheses[\big]{ S_m^f \widetilde{u} } \parentheses[\big]{ S_m^f \widetilde{v} } \,\mathrm{d}\mu_{\phi} = \lim_{n \to +\infty} \frac{1}{nK} \int \parentheses[\big]{ S_{nK}^f \widetilde{u} } \parentheses[\big]{ S_{nK}^f \widetilde{v} } \,\mathrm{d}\mu_{\phi}.
    \]

    Finally, the last statement follows from \cite[Theorem~1.2~(6)]{das2021thermodynamic}.
\end{proof}

Theorem~\ref{prop:property of pressure function and derivative} implies the analyticity and expansion of the pressure function with respect to the imaginary part of the potential. 

\begin{lemma} \label{lem:Taylor extension of pressure function in imaginary part}
    Let $f$, $d$, $\potential$, $\rootpressure$, $\alpha$ satisfy the Assumptions in Section~\ref{sec:The Assumptions}.
    Let $\mu_{-\rootpressure \potential}$ be the unique equilibrium state for the map $f$ and the potential $- \rootpressure \potential$. 
    We assume that $\phi$ is not cohomologous to a constant in $C(S^2)$.
    Denote $\normpotential \define \potential - \alpha$.
    Then there exist $\delta > 0$ and $C_{\delta} \geqslant 0$ such that the function $t \mapsto P(f, (-\rootpressure + \imaginary t) \normpotential) \colon (-\delta, \delta) \mapping \cx$ is complex-analytic, and for each $t \in (-\delta, \delta)$,
    \begin{equation}    \label{eq:lem:Taylor extension of pressure function in imaginary part:Taylor expansion of pressure function}
        \abs[\Big]{ P\parentheses[\big]{ f, (-\rootpressure + \imaginary t)\normpotential }  - P(f, -\rootpressure \normpotential) + \frac{1}{2}\sigma^{2} t^{2} } \leqslant C_{\delta}\abs{ t }^{3},
    \end{equation}
    where $\sigma > 0$ is given by
    \begin{equation}    \label{eq:def:variance of potential}
    \sigma^2 \define \frac{\mathrm{d}^2}{\mathrm{d} t^2} P(f, t \potential) \big|_{t = -\rootpressure } = \lim_{n \to +\infty} \frac{1}{n} \int \! ( S_{n}\potential - n \alpha )^{2} \,\mathrm{d} \mu_{-\rootpressure \potential}.
\end{equation}
\end{lemma}
\begin{proof}
    By Proposition~\ref{prop:property of pressure function and derivative}~\ref{item:prop:property of pressure function and derivative:analytic}, there exists $\delta_{1} > 0$ with $-\rootpressure \in B(-\rootpressure, \delta_{1}) \subseteq \cx$ such that $P(f, \cdot \, \potential) \colon B(-\rootpressure, \delta_{1}) \mapping \cx$ is complex-analytic.
    Recall that $\alpha = \frac{\mathrm{d}}{\mathrm{d} t} P(f, t \potential) |_{t = -\rootpressure}$.  
    Then it follows from Taylor's formula and \eqref{eq:def:variance of potential} that there exist constants $\delta \in (0, \delta_{1}]$ and $C_{\delta} \geqslant 0$ such that for each $t \in (-\delta, \delta)$,
    \[
         \abs[\Big]{ P(f, (-\rootpressure + \imaginary t)\potential) - \parentheses[\Big]{ P(f, -\rootpressure \potential) + \imaginary \alpha t - \frac{1}{2}\sigma^{2} t^{2} } } 
        \leqslant C_{\delta}\abs{ t }^{3}.
    \]
    This implies \eqref{eq:lem:Taylor extension of pressure function in imaginary part:Taylor expansion of pressure function} since $P(f, (-\rootpressure + \imaginary t)\potential) = P\parentheses[\big]{ f, (-\rootpressure + \imaginary t)\normpotential }  + (-\rootpressure + \imaginary t) \alpha$ and $P(f, -\rootpressure \potential) = P(f, -\rootpressure \normpotential) - \rootpressure \alpha$.
\end{proof}

\subsection{Decay estimates of the partition function}%
\label{sub:Decay estimates of the partition function}

In this subsection, we investigate partition functions and obtain several decay estimates to prove the main theorem.
Our approach is inspired by the recent work \cite{li2024prime:dirichlet,li2024prime:split,li2024prime:non-integrability} on prime orbit theorems for expanding Thurston maps.


\def\sft{\Sigma^{+}_{A}}    \def\sopt{\sigma_{A}}  
\def\sequ#1{\{ #1_{i} \}_{i \in \n_0}}

The following result proved in \cite[Proposition~6.1]{li2024prime:split} is a version of Ruelle's estimate adapted to our setting.
The idea of the proof originated from Ruelle \cite{ruelle1990extension}.

\begin{proposition}[Z.~Li \& T.~Zheng \cite{li2024prime:split}]    \label{prop:Ruelle's estimate for the shift associated with expanding Thurston map}
    Let $f$, $\mathcal{C}$, $d$, $\Lambda$, $\holderexp$, $\potential$, $\rootpressure$ satisfy the Assumptions in Section~\ref{sec:The Assumptions}.
    We assume that $f(\mathcal{C}) \subseteq \mathcal{C}$ and no $1$-tile in $\Tile{1}$ joins opposite sides of $\mathcal{C}$.
    Let $\parentheses[\big]{ \Sigma^{+}_{A_{\vartriangle}}, \sigma_{A_{\vartriangle}} }$ be the one-sided subshift of finite type associated to $f$ and $\mathcal{C}$ defined in Proposition~\ref{prop:one-sided subshift of finite type associated with expanding Thurston map}, and let $\pi_{\vartriangle} \colon \Sigma^{+}_{A_{\vartriangle}} \mapping S^2$ be defined in \eqref{eq:prop:one-sided subshift of finite type associated with expanding Thurston map:factor map}.
    
    Then for each $\delta > 0$ there  exists a constant $D_{\delta} > 0$ such that for all integers $n \geqslant 2$ and $k \in \n$, we have
    \begin{equation}    \label{eq:prop:Ruelle's estimate for the shift associated with expanding Thurston map:holder norm}
        \sum_{X^k \in \Tile{k}} \max_{\colour \in \colours} \norm[\big]{\mathcal{L}^{(k)}_{- s \potential, \colour, X^k}(\indicator{X^k})}_{C^{0, \holderexp}}
        \leqslant D_{\delta} \abs{\Im{s}} \Lambda^{-\holderexp} \myexp{k(\delta + P(f, -\Re{s}\potential))}.
    \end{equation}
    and
    \begin{equation}    \label{eq:prop:Ruelle's estimate for the shift associated with expanding Thurston map:uniform bound for difference}
        \begin{split}
            &\abs[\bigg]{ \shiftpartifun{s}[n][\sigma_{A_{\vartriangle}}][-\potential\circ\pi_{\vartriangle}] - \sum_{\colour \in \colours} \sum_{\substack{X \in \Tile{1} \\ X \subseteq \colourtile}} \mathcal{L}^{(n)}_{-s \potential, \colour, X}(\indicator{X})(x_{X}) }  \\
            &\qquad \leqslant D_{\delta} \abs{ \Im{s} } \sum_{m = 2}^{n} \sploptnorm{ \mathbb{L}^{n - m}_{-s \potential} } \parentheses[\big]{ \Lambda^{-\holderexp} \myexp{\delta + P(f, -\Re{s}\potential)} }^{m}
        \end{split}
    \end{equation}
    for any choice of a point $x_{X} \in \inte{X}$ for each $X \in \Tile{1}$, and for all $s \in \cx$ with $\abs{ \Im{s} } \geqslant 2\rootpressure + 1$ and $\abs{ \Re{s} - \rootpressure } \leqslant \rootpressure$, where $\shiftpartifun{s}[n][\sigma_{A_{\vartriangle}}][-\potential\circ\pi_{\vartriangle}]$ is defined in Definition~\ref{def:partition function}.
\end{proposition}

We record \cite[Theorem~6.3]{li2024prime:split} here, which establishes an exponential bound on the operator norm of the split Ruelle operator.
We remark that by \cite[Proposition~7.3]{li2024prime:split} and \cite[Theorem~F]{li2024prime:dirichlet}, if $\phi$ satisfies the strong non-integrability condition (Definition~\ref{def:strong non-integrability condition}), then $\phi$ is not cohomologous to a constant in $C(S^2)$.

\begin{theorem}[Z.~Li \& T.~Zheng \cite{li2024prime:split}]    \label{thm:exponential bounds for the holder norm of split Ruelle operator}
    Let $f$, $\mathcal{C}$, $d$, $\Lambda$, $\holderexp$, $\potential$, $\rootpressure$ satisfy the Assumptions in Section~\ref{sec:The Assumptions}.
    We assume that $f(\mathcal{C}) \subseteq \mathcal{C}$ and $\potential$ satisfies the $\holderexp$-strong non-integrability condition.
    Then there exists a constant $D' = D'(f, \mathcal{C}, d, \holderexp, \potential) > 0$ such that for each $\varepsilon > 0$ there exist constants $\delta_{\varepsilon} \in (0, \rootpressure)$, $\widetilde{b}_{\varepsilon} \geqslant 2 \rootpressure + 1$, and $\rho_{\varepsilon} \in (0, 1)$ with the following property:

    For each $n \in \n$ and each $s \in \cx$ satisfying $\abs{ \Re{s} - \rootpressure } < \delta_{\varepsilon}$ and $\abs{ \Im{s} } \geqslant \widetilde{b}_{\varepsilon}$, we have
    \begin{equation}    \label{eq:thm:exponential bounds for the holder norm of split Ruelle operator}
        \sploptnorm{ \mathbb{L}^{n}_{-s \potential} } \leqslant D' \abs{ \Im{s} }^{1 + \varepsilon} \rho_{\varepsilon}^{n}.
    \end{equation}
\end{theorem}

We first study the partition function for the symbolic coding of an expanding Thurston map, separating the estimates into unbounded and bounded cases based on the imaginary part of $s$.

In the unbounded case, the following estimate is a direct consequence of Proposition~\ref{prop:Ruelle's estimate for the shift associated with expanding Thurston map} and Theorem~\ref{thm:exponential bounds for the holder norm of split Ruelle operator}.
We adapt the strategy of \cite[Proposition~6.4]{li2024prime:split}; we include the proof for the sake of completeness.

\begin{proposition}    \label{prop:partition function estimate:symbolic coding:unbounded imaginary}
    Let $f$, $\mathcal{C}$, $d$, $\Lambda$, $\holderexp$, $\potential$, $\rootpressure$ satisfy the Assumptions in Section~\ref{sec:The Assumptions}.
    We assume that $\potential$ satisfies the $\holderexp$-strong non-integrability condition, and that $f(\mathcal{C}) \subseteq \mathcal{C}$ and no $1$-tile in $\Tile{1}$ joins opposite sides of $\mathcal{C}$.
    Let $\parentheses[\big]{ \Sigma^{+}_{A_{\vartriangle}}, \sigma_{A_{\vartriangle}} }$ be the one-sided subshift of finite type associated to $f$ and $\mathcal{C}$ defined in Proposition~\ref{prop:one-sided subshift of finite type associated with expanding Thurston map}, and let $\pi_{\vartriangle} \colon \Sigma^{+}_{A_{\vartriangle}} \mapping S^2$ be defined in \eqref{eq:prop:one-sided subshift of finite type associated with expanding Thurston map:factor map}.

    Then for each $\varepsilon > 0$, there exist constants $T > 1$, $C_{\vartriangle} > 0$, and $\rho_{\vartriangle} \in (0, 1)$ such that for each $t \in \real \mysetminus (-T, T)$ and each integer $n \geqslant 2$, we have
    \begin{equation}    \label{eq:prop:partition function estimate:symbolic coding:unbounded imaginary}
        \abs[\Big]{ \shiftpartifun{ -\rootpressure + \imaginary t }[n][\sigma_{A_{\vartriangle}}][\potential\circ\pi_{\vartriangle}]} 
        \leqslant C_{\vartriangle} \abs{t}^{2 + \varepsilon} \rho_{\vartriangle}^{n}.
      \end{equation}  
\end{proposition}
\begin{proof}
    Let $\delta \define \frac{1}{2} \log(\Lambda^{\holderexp}) > 0$.
    Fix an arbitrary $\varepsilon > 0$.
    We choose $T \define \widetilde{b}_{\varepsilon} > 1$, where $\widetilde{b}_{\varepsilon}$ is the constant from Theorem~\ref{thm:exponential bounds for the holder norm of split Ruelle operator} depending only on $f$, $\mathcal{C}$, $d$, $\holderexp$, $\potential$, and $\varepsilon$.

    Fix an arbitrary point $x_{X^{1}} \in \inte{X^{1}}$ for each $X^{1} \in \tile{1}$.
    Recall the definition of partial split Ruelle operators in Definition~\ref{def:partial split Ruelle operator} (see also \cite[Lemmas~5.3~and~5.7]{li2024prime:split}). 
    By applying the estimate \eqref{eq:prop:Ruelle's estimate for the shift associated with expanding Thurston map:holder norm} from Proposition~\ref{prop:Ruelle's estimate for the shift associated with expanding Thurston map}, for each $n \geqslant 2$ and each $t \in \real \mysetminus (-T, T)$, we obtain the following bound:
    \begin{equation}    \label{eq:temp:prop:partition function estimate:symbolic coding:unbounded imaginary:split Ruelle operator}
        \begin{aligned}
        & \uppercase\expandafter{\romannumeral1} \define \abs[\bigg]{ \sum_{\colour \in \colours} \sum_{\substack{X^1 \in \mathbf{X}^1 \\ X^1 \subseteq X_{\colour}^0}} \mathcal{L}_{(-\rootpressure + \imaginary t) \potential, \colour, X^1}^{(n)} \parentheses{ \indicator{X^1} } (x_{X^1}) } \\
        & \qquad \leqslant \sum_{\colour \in\colours} \sum_{\substack{X^1 \in \mathbf{X}^1 \\ X^1 \subseteq X_{\colour}^0}} \abs[\bigg]{ \sum_{\colour^{\prime} \in \colours} \mathcal{L}_{(-\rootpressure + \imaginary t) \potential, \colour, \colour^{\prime}}^{(n-1)} \parentheses[\Big]{ \mathcal{L}_{(-\rootpressure + \imaginary t) \potential, \colour^{\prime}, X^1}^{(1)} \parentheses[\big]{ \indicator{X^1} } } (x_{X^1}) } \\
        & \qquad \leqslant \sploptnorm[\big]{ \mathbb{L}_{(-\rootpressure + \imaginary t) \potential}^{n-1} } \sum_{\colour \in \colours} \sum_{\substack{X^1 \in \mathbf{X}^1 \\ X^1 \subseteq X_{\colour}^0}} \max_{\colour^{\prime} \in\colours} \norm[\Big]{ \mathcal{L}_{(-\rootpressure + \imaginary t) \potential, \colour^{\prime}, X^1}^{(1)} \parentheses{ \indicator{X^1} } }_{C^{0, \holderexp}} \\
        & \qquad \leqslant \sploptnorm[\big]{\mathbb{L}_{(-\rootpressure + \imaginary t) \potential}^{n-1} } \, D_\delta \abs{t} \Lambda^{-\holderexp} \myexp{ \delta + P(f,-\rootpressure \potential) }  \\
        & \qquad = \sploptnorm[\big]{\mathbb{L}_{(-\rootpressure + \imaginary t) \potential}^{n-1} } \, D_\delta \abs{t} \Lambda^{- \holderexp / 2},
        \end{aligned}
    \end{equation}
    where $D_{\delta} > 0$ is the constant given in Proposition~\ref{prop:Ruelle's estimate for the shift associated with expanding Thurston map}, which depends only on $f$, $\mathcal{C}$, $d$, $\holderexp$, $\potential$, and $\delta$.
    
    Hence, by \eqref{eq:def:partition function}, the triangle inequality, \eqref{eq:prop:Ruelle's estimate for the shift associated with expanding Thurston map:uniform bound for difference} in Proposition~\ref{prop:Ruelle's estimate for the shift associated with expanding Thurston map}, \eqref{eq:temp:prop:partition function estimate:symbolic coding:unbounded imaginary:split Ruelle operator}, and Theorem~\ref{thm:exponential bounds for the holder norm of split Ruelle operator}, we deduce that for each $n \geqslant 2$ and each $t \in \real \mysetminus (-T, T)$,
    \begin{align*}
        \abs[\Big]{ \shiftpartifun{ -\rootpressure + \imaginary t }[n][\sigma_{A_{\vartriangle}}][\potential\circ\pi_{\vartriangle}] } 
        & \leqslant \uppercase\expandafter{\romannumeral1}  +  \abs[\Bigg]{ \shiftpartifun{ -\rootpressure + \imaginary t }[n][\sigma_{A_{\vartriangle}}][\potential\circ\pi_{\vartriangle}]  -  \sum_{\colour \in \colours} \sum_{\substack{X^1 \in \mathbf{X}^1 \\ X^1 \subseteq X_{\colour}^0}} \mathcal{L}_{(-\rootpressure + \imaginary t) \potential, \colour, X^1}^{(n)} \parentheses{ \indicator{X^1} } (x_{X^1}) }  \\
        & \leqslant  D_\delta \abs{t} \, \sploptnorm[\big]{\mathbb{L}_{(-\rootpressure + \imaginary t) \potential}^{n-1} } \Lambda^{- \holderexp / 2}  +  D_{\delta} \abs{t} \sum_{m = 2}^{n} \sploptnorm[\big]{ \mathbb{L}^{n - m}_{(-\rootpressure + \imaginary t) \potential} }  \Lambda^{- \holderexp m / 2}  \\
        & \leqslant  D_{\delta} D' \abs{t}^{2 + \varepsilon}  \sum_{m = 1}^{n} \rho_{\varepsilon}^{n - m}  \Lambda^{- \holderexp m / 2}  \\
        & \leqslant  D_{\delta} D' \abs{t}^{2 + \varepsilon} \, n \parentheses[\big]{ \max\set[\big]{\rho_{\varepsilon}, \Lambda^{- \holderexp / 2}} }^{n},
    \end{align*}
    where constants $D' > 0$ and $\rho_{\varepsilon} \in (0, 1)$ are from Theorem~\ref{thm:exponential bounds for the holder norm of split Ruelle operator} depending only on $f$, $\mathcal{C}$, $d$, $\holderexp$, $\potential$, and $\varepsilon$.
    
    Therefore, we may choose constants 
    \[
         \rho_{\vartriangle} \define \frac{1}{2} \parentheses[\big]{ 1 +  \max\set[\big]{\rho_{\varepsilon}, \Lambda^{- \holderexp / 2}} } \in (0, 1)
    \] 
    and 
    \[
        C_{\vartriangle} \define D_{\delta} D' \sup_{n \in \n} \set{n \rho_{\vartriangle}^{n}} < +\infty
    \] 
    to ensure that the estimate \eqref{eq:prop:partition function estimate:symbolic coding:unbounded imaginary} holds.
\end{proof}

\def\errdegpartifun{\mathrm{I}_n(s)}
\def\codingpartifun{\Pi_n(s)}

We record the following result from \cite[Theorem~6.8]{li2024prime:dirichlet}.

\begin{theorem}[Li \& Zheng \cite{li2024prime:dirichlet}] \label{thm:expanding Thurston map symbolic pressure relation}
	Let $f$, $\mathcal{C}$, $d$ satisfy the Assumptions in Section~\ref{sec:The Assumptions}.
	We assume that $f(\mathcal{C}) \subseteq \mathcal{C}$.
	Let $\varphi \in C^{0,\holderexp}(S^2, d)$ be a real-valued \holder continuous function with exponent $\holderexp$.
	Recall the one-sided subshifts of finite type $\parentheses[\big]{ \Sigma_{\ematrix}^+, \sigma_{\ematrix} }$ and $\parentheses[\big]{ \Sigma_{\eematrix}^+, \sigma_{\eematrix} }$ constructed in Subsection~\ref{sub:Symbolic dynamics for expanding Thurston maps}.
	We denote by $(\mathbf{V}^0, f|_{\mathbf{V}^0})$ the dynamical system on $\mathbf{V}^0 \define \Vertex{0} = \post{f}$ induced by $f|_{\mathbf{V}^0} \colon \mathbf{V}^0 \to \mathbf{V}^0$.
	Then the following relations between the topological pressure of these systems hold:
	\begin{enumerate}
        \smallskip

		\item \label{item:thm:expanding Thurston map symbolic pressure relation:vertex}
            $P(f, \varphi) > P(f|_{\mathbf{V}^0}, \varphi|_{\mathbf{V}^0})$,
		
        \smallskip
		
        \item \label{item:thm:expanding Thurston map symbolic pressure relation:Jordan curve and edge}
            $P(f, \varphi) > P(f|_{\mathcal{C}}, \varphi|_{\mathcal{C}}) = P(\sigma_{\ematrix}, \varphi \circ \ecoding) = P(\sigma_{\eematrix}, \varphi \circ \ecoding \circ \eecoding)$.
	\end{enumerate}
\end{theorem}

We record the following result from \cite[Theorem~6.3]{li2024prime:dirichlet}.

\begin{theorem}[Li \& Zheng \cite{li2024prime:dirichlet}] \label{thm:decomposition of local degree by cardinality of the preimages and periodic points of codings}
    Let $f$ and $\mathcal{C}$ satisfy the Assumptions in Section~\ref{sec:The Assumptions}. 
    We assume that $f(\mathcal{C}) \subseteq \mathcal{C}$. 
    Let $\parentheses[\big]{ \Sigma_{A_{\vartriangle}}^+, \sigma_{A_{\vartriangle}} }$ be the one-sided subshift of finite type associated to $f$ and $\mathcal{C}$ defined in Proposition~\ref{prop:one-sided subshift of finite type associated with expanding Thurston map}, and let $\pi_{\vartriangle} \colon \Sigma_{A_{\vartriangle}}^+ \to S^2$ be the factor map defined in~\eqref{eq:prop:one-sided subshift of finite type associated with expanding Thurston map:factor map}.
    Recall the one-sided subshifts of finite type $\parentheses[\big]{ \Sigma_{\ematrix}^+, \sigma_{\ematrix} }$ and $\parentheses[\big]{ \Sigma_{\eematrix}^+, \sigma_{\eematrix} }$ constructed in Subsection~\ref{sub:Symbolic dynamics for expanding Thurston maps}, and the factor maps $\ecoding \colon \Sigma_{\ematrix}^+ \to \mathcal{C}$ and $\eecoding \colon \Sigma_{\eematrix}^+ \to \Sigma_{\ematrix}^+$ defined in Proposition~\ref{prop:edge_symbolic_factors_properties}.
    We denote by $(\mathbf{V}^0, f|_{\mathbf{V}^0})$ the dynamical system on $\mathbf{V}^0 = \Vertex{0} = \post{f}$ induced by $f|_{\mathbf{V}^0} \colon \mathbf{V}^0 \to \mathbf{V}^0$. 
    Then for each $n \in \n$ and each $x \in \myperiodpoint{f}{n}$, we have
    \begin{equation} \label{eq:thm:decomposition of local degree by cardinality of the preimages and periodic points of codings:zeta_function_identity} 
        \deg_{f^n}(x) = M_{A_{\vartriangle}}(x, n) - M_{\eematrix}(x, n) + M_{\ematrix}(x, n) + M_{\bullet}(x, n),
    \end{equation}
    where
    \begin{align*}
        M_{A_{\vartriangle}}( x, n )
            & \define \card[\big]{ \myperiodpoint[\big]{\sigma_{A_{\vartriangle}}}{n} \cap \pi_{\vartriangle}^{-1}( x ) }, \\
        M_{\eematrix}( x, n )
            & \define \card[\big]{ \myperiodpoint[\big]{\sigma_{\eematrix}}{n} \cap ( \ecoding \circ \eecoding )^{-1}( x ) }, \\
        M_{\ematrix}( x, n )
            & \define \card[\big]{ \myperiodpoint[\big]{\sigma_{\ematrix}}{n} \cap \ecoding^{-1}( x ) }, \\
        M_{\bullet}( x, n )
            & \define \card{ \myperiodpoint{ ( f|_{\mathbf{V}^0} ) }{n} \cap \{ x \} }.
    \end{align*}
\end{theorem}

By definition, $M_{\eematrix}(x, n) = M_{\ematrix}(x, n) = 0$ for all $x \notin \mathcal{C}$, and $M_{\bullet}(x, n) = 0$ for all $x \notin \mathbf{V}^0$.

\begin{lemma} \label{lem:factor_map_periodic_points}
	Let $X$ and $Y$ be topological spaces, and let $G \colon Y \to Y$ and $g \colon X \to X$ be continuous maps.
	If $(X, g)$ is a factor of $(Y, G)$ with a factor map $\pi \colon Y \to X$, then for each $n \in \n$, we have $\pi(\myperiodpoint{G}{n}) \subseteq \myperiodpoint{g}{n}$, and the set \( \myperiodpoint{G}{n} \) of periodic points admits the decomposition
	\[
		\myperiodpoint{G}{n} = \bigcup_{ x \in \myperiodpoint{g}{n} } \parentheses[\big]{ \myperiodpoint{G}{n} \cap \pi^{-1}( x ) }.
	\]
    Moreover, for each complex-valued function $\varphi \colon X \to \cx$, we have
    \[  
    	\sum_{ y \in \myperiodpoint{G}{n} } e^{ S_n \varphi (\pi(y)) }
    	= \sum_{ x \in \myperiodpoint{g}{n} } \card[\big]{ \myperiodpoint[\big]{G}{n} \cap \pi^{-1}( x ) } \, e^{ S_n \varphi(x) }
    \]
\end{lemma}
\begin{proof}
    For each \( y \in \myperiodpoint{G}{n} \) we have \( g^n(\pi(y)) = \pi(G^n(y)) = \pi(y) \) since \( g \circ \pi = \pi \circ G \).
    Thus \( \pi(y) \in \myperiodpoint{g}{n} \), which establishes the inclusion \( \pi(\myperiodpoint{G}{n}) \subseteq \myperiodpoint{g}{n} \).

    The decomposition of \( \myperiodpoint{G}{n} \) follows directly.
    The inclusion \( \supseteq \) is trivial.
    For the reverse inclusion, every \( y \in \myperiodpoint{G}{n} \) belongs to \( \myperiodpoint{G}{n} \cap \pi^{-1}(\pi(y)) \).
    Since \( \pi(y) \in \myperiodpoint{g}{n} \), this set is part of the union.
    
    The identity for the sum over periodic points is obtained by grouping terms according to the decomposition:
    \[
        \sum_{ y \in \myperiodpoint{G}{n} } e^{ S_n \varphi ( \pi(y) ) }
        = \sum_{ x \in \myperiodpoint{g}{n} } \sum_{ y \in \myperiodpoint{G}{n} \cap \pi^{-1}(x) } e^{ S_n \varphi ( \pi(y) ) }.
    \]
    Since $\pi(y) = x$ for every $y$ in the inner sum, this simplifies to
    \[
        \sum_{ x \in \myperiodpoint{g}{n} } \sum_{ y \in \myperiodpoint{G}{n} \cap \pi^{-1}(x) } e^{ S_n \varphi(x) }
        = \sum_{ x \in \myperiodpoint{g}{n} } \card[\big]{ \myperiodpoint{G}{n} \cap \pi^{-1}(x) } \, e^{ S_n \varphi(x) }.
    \]
    This completes the proof.
\end{proof}

\begin{lemma} \label{lem:upper bound of shift partition function}
    Consider a finite set of states $S$ and a transition matrix $A \colon S \times S \to \{0, 1\}$. 
    Let $\parentheses[\big]{ \Sigma_{A}^+, \sigma_{A} }$ be the one-sided subshift of finite type defined by $A$, and $\psi \in C\parentheses[\big]{ \Sigma_{A}^+ }$ be a real-valued continuous function.
    Then for each $\varepsilon > 0$ there exists $C = C(\varepsilon) > 0$ such that for all $n \in \n$,
    \begin{equation}    \label{eq:lem:upper bound of shift partition function}
        \sum_{ y \in \myperiodpoint{ \sigma_{A} }{ n } } e^{ S_n \psi( y ) } 
        \leqslant C e^{ n ( P(\sigma_{A}, \psi) + \varepsilon ) }.
    \end{equation}
\end{lemma}
\begin{proof}
    Let $n \in \n$ and $Z_n(\psi) \define \sum_{ y \in \myperiodpoint{ \sigma_{A} }{ n } } e^{ S_n \psi( y ) }$.
    Any two distinct points in $\myperiodpoint{\sigma_{A}}{n}$ must differ in at least one of their first $n$ coordinates, which implies that $\myperiodpoint{\sigma_{A}}{n}$ is an $(n, 1)$-separated set (see Subsection~\ref{sub:thermodynamic formalism}).
    Then it follows directly from the definition of topological pressure \eqref{eq:def:topological pressure} that
    \[
        \limsup_{n \to +\infty} \frac{1}{n} \log Z_n(\psi) \leqslant P(\sigma_A, \psi).
    \]
    Given $\varepsilon > 0$, this implies $Z_n(\psi) \leqslant e^{n(P(\sigma_A, \psi) + \varepsilon)}$ for all sufficiently large $n$.
    The claimed inequality \eqref{eq:lem:upper bound of shift partition function} for all $n \in \n$ then follows by choosing a sufficiently large constant $C = C(\varepsilon) > 0$.
\end{proof}


\def\errdegpartifun{\mathrm{I}_n(s)}
\def\codingpartifun{\Pi_n(s)}

\def\degbirkhoffsum#1{\eta_{#1}}

\def\critprimeorbit{\mathcal{P}^>(f|_{\mathbf{V}^0})}
    
\def\degbirkhoffsum#1{\eta_{#1}}

\begin{lemma} \label{lem:relation between partition function of expanding Thurston map and tile coding}
    Let $f$, $\mathcal{C}$, $d$, $\potential$, $\rootpressure$ satisfy the Assumptions in Section~\ref{sec:The Assumptions}.
    We assume that $f(\mathcal{C}) \subseteq \mathcal{C}$ and no $1$-tile in $\Tile{1}$ joins opposite sides of $\mathcal{C}$.
    Then there exist constants $\kappa \in (0, 1)$ and $C > 0$ such that for each $n \in \n$ and each $s \in \cx$ with $\Re{s} = -\rootpressure$, we have
    \[
        \abs[\big]{ \partifun{s}[n][f][\potential] - \shiftpartifun{s}[n][\sigma_{A_{\vartriangle}}][\potential\circ\pi_{\vartriangle}] } \leqslant C \kappa^n.
    \]
\end{lemma}
\begin{proof}
    Let $s \in \cx$ with $\Re{s} = -\rootpressure$.
    
    We begin by decomposing the partition function $\partifun{s}$ into two terms:
    \begin{equation} \label{eq:proof:lem:relation between partition function of expanding Thurston map and tile coding:partition_function_decomposition_step1}
        \partifun{s}
        = \sum_{ x \in \myperiodpoint{f}{n} } \deg_{ f^n }( x ) e^{ s S_n \potential( x ) } - \sum_{ x \in \myperiodpoint{f}{n} } \parentheses[\big]{ \deg_{ f^n }( x ) - 1 } e^{ s S_n \potential( x ) }.
    \end{equation}
    Substituting the expression \eqref{eq:thm:decomposition of local degree by cardinality of the preimages and periodic points of codings:zeta_function_identity} in Theorem~\ref{thm:decomposition of local degree by cardinality of the preimages and periodic points of codings} for the local degree $\deg_{ f^n }( x )$ transforms the first sum on the right-hand side of \eqref{eq:proof:lem:relation between partition function of expanding Thurston map and tile coding:partition_function_decomposition_step1} into
    \[
        \sum_{ x \in \myperiodpoint{f}{n} } \parentheses[\big]{ M_{A_{\vartriangle}}(x, n) - M_{\eematrix}(x, n) + M_{\ematrix}(x, n) + M_{\bullet}(x, n) } e^{ s S_n \potential( x ) }.
    \]
    By Lemma~\ref{lem:factor_map_periodic_points}, this sum can be expressed in terms of the associated symbolic systems via the factor maps $\pi_ {\vartriangle}$, $\ecoding \circ \eecoding$, and $\ecoding$ (cf.~Propositions~\ref{prop:one-sided subshift of finite type associated with expanding Thurston map} and~\ref{prop:edge_symbolic_factors_properties}), yielding
    \begin{align*}
        \shiftpartifun{s}
            &- \sum_{ y \in \myperiodpoint{ \sigma_{\eematrix} }{ n } } e^{ s S_n \potential( \ecoding \circ \eecoding (y) ) } + \sum_{ y \in \myperiodpoint{ \sigma_{\ematrix} }{ n } } e^{ s S_n \potential( \ecoding( y ) ) }
            + \sum_{ x \in \myperiodpoint{ ( f|_{\mathbf{V}^0} ) }{ n } } e^{ s S_n \potential( x ) }.
    \end{align*}
    Collecting all terms, we obtain the decomposition
    \[
        \partifun{s} = \shiftpartifun{s} - \errdegpartifun - \codingpartifun,
    \]
    where
    \[
        \errdegpartifun \define \sum_{ x \in \myperiodpoint{f}{n} } \parentheses[\big]{ \deg_{ f^n }( x ) - 1 } e^{ s S_n \potential( x ) }
    \]
    and
    \[
        \codingpartifun \define 
            \sum_{ y \in \myperiodpoint{ \sigma_{\eematrix} }{ n } } e^{ s S_n \potential \circ \ecoding \circ \eecoding (y) }
            - \sum_{ y \in \myperiodpoint{ \sigma_{\ematrix} }{ n } } e^{ s S_n \potential \circ \ecoding (y) }
            - \sum_{ x \in \myperiodpoint{ ( f|_{\mathbf{V}^0} ) }{ n } } e^{ s S_n \potential( x ) }.
    \]
    It remains to show that $\abs{\errdegpartifun}$ and $\abs{\codingpartifun}$ are bounded by terms that decay exponentially in $n$.

    We first estimate $\abs{\errdegpartifun}$.
    The only non-zero terms in the sum defining $\errdegpartifun$ correspond to periodic points $x$ with $\deg_{f^n}(x) > 1$, which are necessarily contained in the postcritical set $\mathbf{V}^0 = \post{f}$.
    To analyze these terms, we introduce some notation.
    For a primitive periodic orbit $\tau \in \priorbit$, we write
    \[
        l_{f, \, \potential}(\tau) \define \sum_{y \in \tau} \potential(y) \quad \text{and} \quad \deg_f(\tau) \define \prod_{y \in \tau} \deg_f(y).
    \]
    Define $\critprimeorbit \define \{ \tau \in \priorbit[f|_{\mathbf{V}^0}] \describe \deg_f(\tau) > 1 \}$, which is a finite set since $\mathbf{V}^0$ is finite.
    For each $\tau \in \critprimeorbit$, we define
    \[
        \degbirkhoffsum{\tau} \define \deg_f(\tau) e^{- \rootpressure l_{f, \potential}(\tau)},
    \]
    and let
    \[
        \degbirkhoffsum{} \define \max_{ \tau \in \critprimeorbit } \degbirkhoffsum{\tau}^{ 1 / \abs{\tau} }.
    \]
    A key claim in the proof of Theorem~D in \cite[p.~82]{li2024prime:dirichlet}\footnote{The proof relies on the assumption that no 1-tile in $\Tile{1}$ joins opposite sides of $\mathcal{C}$.} establishes that $\degbirkhoffsum{\tau} < 1$ for all $\tau \in \critprimeorbit$, which implies that $\degbirkhoffsum{} \in (0, 1)$.

    Now, consider a point $x \in \myperiodpoint{f}{n}$ with $\deg_{f^n}(x) > 1$.
    Then $x$ belongs to a primitive periodic orbit $\tau_x \in \critprimeorbit$ of period $k \define \abs{\tau_x}$, and $n$ must be a multiple of $k$, say $n = mk$ for some $m \in \n$.
    Thus $\deg_{f^n}(x) = (\deg_f(\tau_x))^m$ and $S_n\potential(x) = m l_{f, \potential}(\tau_x)$, which yields
    \[
        \deg_{f^n}(x) e^{ -\rootpressure S_n \potential(x) }
        = \parentheses[\big]{ \deg_f(\tau_x) e^{ -\rootpressure l_{f, \potential}(\tau_x) } }^m
        = \degbirkhoffsum{\tau_x}^m
        \leqslant (\degbirkhoffsum{}^k)^m = \degbirkhoffsum{}^n.
    \]
    This implies that
    \begin{align*}
        \abs[\big]{ \errdegpartifun }
        &\leqslant \sum_{ x \in \myperiodpoint{f}{n} } \parentheses[\big]{ \deg_{ f^n }( x ) - 1 } e^{ -\rootpressure S_n \potential( x ) } \\
        &= \sum_{ \substack{  x \in \myperiodpoint{f}{n} \cap \mathbf{V}^{0} \\ \deg_{f^{n}}(x) > 1 } } \parentheses[\big]{ \deg_{ f^n }( x ) - 1 } e^{ -\rootpressure S_n \potential( x ) } 
        \leqslant \card{\mathbf{V}^{0}} \degbirkhoffsum{}^{n}.
    \end{align*}
    This establishes a desired exponential bound for $\abs{\errdegpartifun}$.

    We next estimate $\abs{\codingpartifun}$. 
    Recall from Theorem~\ref{thm:expanding Thurston map symbolic pressure relation} that $P(f|_{\mathbf{V}^0}, - \rootpressure \potential|_{\mathbf{V}^0}) < P(f, - \rootpressure \potential) = 0$ and
    \begin{gather*}
        P(\sigma_{\eematrix}, - \rootpressure \potential \circ \ecoding \circ \eecoding) = P(\sigma_{\ematrix}, - \rootpressure \potential \circ \ecoding) = P(f|_{\mathcal{C}}, - \rootpressure \potential|_{\mathcal{C}}) < P(f, - \rootpressure \potential) = 0.
    \end{gather*}
    This allows us to choose $\delta > 0$ sufficiently small such that
    \[
        P(f|_{\mathcal{C}}, - \rootpressure \potential|_{\mathcal{C}}) + \delta < - \delta \quad \text{and} \quad P(f|_{\mathbf{V}^0}, - \rootpressure \potential|_{\mathbf{V}^0}) + \delta < - \delta.
    \]
    Thus by Lemma~\ref{lem:upper bound of shift partition function} and \eqref{eq:def:topological pressure}, there exist constants $C_1, C_2, C_3 > 0$ such that for all $n \in \n$,
    \begin{align*}
        \abs{ \codingpartifun }
        &\leqslant \sum_{ y \in \myperiodpoint{ \sigma_{\eematrix} }{ n } } e^{ -\rootpressure S_n \potential \circ \ecoding \circ \eecoding (y) }
            + \sum_{ y \in \myperiodpoint{ \sigma_{\ematrix} }{ n } } e^{ -\rootpressure S_n \potential \circ \ecoding (y) }
            + \sum_{ x \in \myperiodpoint{ ( f|_{\mathbf{V}^0} ) }{ n } } e^{ -\rootpressure S_n \potential( x ) }  \\
        &\leqslant C_1 e^{ n ( P(\sigma_{\eematrix}, - \rootpressure \potential \circ \ecoding \circ \eecoding) + \delta ) }
            + C_2 e^{ n ( P(\sigma_{\ematrix}, - \rootpressure \potential \circ \ecoding) + \delta ) }
            + C_3 e^{ n ( P(f|_{\mathbf{V}^0}, - \rootpressure \potential|_{\mathbf{V}^0}) + \delta ) } \\
        &= (C_1+C_2) e^{ n ( P(f|_{\mathcal{C}}, - \rootpressure \potential|_{\mathcal{C}}) + \delta ) } + C_3 e^{ n ( P(f|_{\mathbf{V}^0}, - \rootpressure \potential|_{\mathbf{V}^0}) + \delta ) } \\
        &\leqslant (C_1 + C_2 + C_3) e^{-\delta n}.
    \end{align*}
    This establishes a desired exponential bound for $\abs{\codingpartifun}$.

    Combining the estimates for $\abs{\errdegpartifun}$ and $\abs{\codingpartifun}$, we set $\kappa \define \max\{\degbirkhoffsum, \, e^{-\delta}\} \in (0,1)$ and $C \define \card{\mathbf{V}^{0}} + C_1+C_2+C_3$. 
    It follows that
    \[
        \abs[\big]{ \partifun{s} - \shiftpartifun{s} } \leqslant \abs{ \errdegpartifun } + \abs{ \codingpartifun } \leqslant C \kappa^n,
    \]
    which completes the proof.
\end{proof}

\begin{proposition}    \label{prop:partition function estimate:expanding Thurston map:unbounded imaginary}
    Let $f$, $\mathcal{C}$, $d$, $\Lambda$, $\holderexp$, $\potential$, $\rootpressure$, $\alpha$ satisfy the Assumptions in Section~\ref{sec:The Assumptions}.
    We assume that $\potential$ satisfies the $\holderexp$-strong non-integrability condition, and that $f(\mathcal{C}) \subseteq \mathcal{C}$ and no $1$-tile in $\Tile{1}$ joins opposite sides of $\mathcal{C}$.
    Define $\normpotential \define \phi - \alpha$.
    Then for each $\varepsilon > 0$, there exist constants $T > 1$, $\rho \in (0, 1)$, and $C > 0$ such that for each $t \in \real \mysetminus (-T, T)$ and each integer $n \geqslant 2$, we have
    \[
        \abs[\big]{ \normpartifun{-\rootpressure + \imaginary t} } \leqslant C \abs{t}^{2 + \varepsilon} \rho^{n} e^{n P(f, -\rootpressure \normpotential)}.
    \]    
\end{proposition}
\begin{proof}
    Let $s \define -\rootpressure + \imaginary t$ with $t \in \real$.
    By Definition~\ref{def:partition function}, we have the relation $\normpartifun{s} = e^{- s \alpha n} \partifun{s}$.
    Note that $P(f, -\rootpressure \normpotential) = P(f, -\rootpressure \potential) + \rootpressure \alpha = \rootpressure \alpha$ since $P(f, -\rootpressure \potential) = 0$.
    Thus $\abs[\big]{ \normpartifun{s} } = e^{n P(f, -\rootpressure\normpotential)} \abs[\big]{ \partifun{s} }$ and it suffices to find an appropriate bound for $\abs{\partifun{s}}$.

    Let $\varepsilon > 0$ be arbitrary. 
    Then by Proposition~\ref{prop:partition function estimate:symbolic coding:unbounded imaginary}, there exist $T > 1$, $C_{\vartriangle} > 0$, and $\rho_{\vartriangle} \in (0,1)$ such that for all $n \geqslant 2$ and $t \in \real \mysetminus (-T, T)$,
    \[
        \abs{ \shiftpartifun{s} } \leqslant C_{\vartriangle} \abs{t}^{2+\varepsilon} \rho_{\vartriangle}^n.
    \]
    By Lemma~\ref{lem:relation between partition function of expanding Thurston map and tile coding}, there exist constants $\kappa \in (0, 1)$ and $C_1 > 0$ such that $\abs[\big]{ \partifun{s} - \shiftpartifun{s} } \leqslant C_1 \kappa^n$ for all $n \in \n$ and $t \in \real$.
    Let $\rho \define \max\{ \rho_{\vartriangle}, \kappa \}$.
    Then for all $n \geqslant 2$ and $t \in \real \mysetminus (-T, T)$, we have
    \begin{align*}
        \abs[\big]{ \partifun{s} } 
        &\leqslant \abs[\big]{ \shiftpartifun{s} } + \abs[\big]{ \partifun{s} - \shiftpartifun{s} }  \\
        &\leqslant C_{\vartriangle} \abs{t}^{2+\varepsilon} \rho_{\vartriangle}^n + C_1 \kappa^n 
        \leqslant (C_{\vartriangle} + C_1) \abs{t}^{2+\varepsilon} \rho^n.
    \end{align*}
    Setting $C \define C_{\vartriangle} + C_1$ completes the proof.
\end{proof}



We recall the definition of the non-lattice property, which is needed to apply the complex Ruelle--Perron--Frobenius theorem in \cite{pollicottZetaFunctionsPeriodic1990}.

\begin{definition}    \label{def:non-lattice property}
    Consider a finite set of states $S$ and a transition matrix $A \colon S \times S \to \{0, 1\}$. 
      Let $\parentheses[\big]{ \Sigma_{A}^+, \sigma_{A} }$ be the one-sided subshift of finite type defined by $A$.
      A real-valued function $\psi \colon \Sigma_{A}^+ \mapping \real$ is called \emph{non-lattice} if there exists no continuous function $u \colon \Sigma_{A}^+ \mapping 2\pi \z$ such that
      \[
          \psi = C + u + v \circ \sigma_{A} - v
      \]
      for some constant $C \in \real$ and continuous function $v \in C \parentheses[\big]{ \Sigma_{A}^+ }$.
  \end{definition}
  
  \def\finiteappro{\psi^{(k)}}
  \def\ruelleoperator{\mathcal{L}_{\psi_t}}
  \def\spectrario{\rho(\ruelleoperator)}
  
  The following proposition is a part of \cite[Theorem~F]{li2024prime:dirichlet}.
  
  \begin{proposition}[Li \& Zheng \cite{li2024prime:dirichlet}] \label{prop:cohomologous_equivalence}
      Let \( f \colon S^2 \to S^2 \) be an expanding Thurston map and \( d \) a visual metric on \( S^2 \) for \( f \).
      Let \( \psi \in C^{0,\holderexp}\parentheses[\big]{ (S^2, d), \mathbb{C} } \) be a complex-valued \holder continuous function with an exponent \( \holderexp \in (0, 1] \).
      Then the following statements are equivalent:
      \begin{enumerate}
          \smallskip
  
          \item   \label{item:prop:cohomologous_equivalence:cohomologous}
          The function \( \psi \) is cohomologous to a constant in \( C(S^2, \mathbb{C}) \), i.e., \( \psi = K + u \circ f - u \) for some \( K \in \mathbb{C} \) and \( u \in C(S^2, \mathbb{C}) \).
  
          \smallskip
  
          \item   \label{item:prop:cohomologous_equivalence:lattice}
          There exists \( n \in \n \) and a Jordan curve \( \mathcal{C} \subseteq S^2 \) with \( f^n(\mathcal{C}) \subseteq \mathcal{C} \) and \( \post{f} \subseteq \mathcal{C} \) such that the following statement holds for \( F \define f^n \), \( \Psi \define S_n^f \psi \), the one-sided subshift of finite type \( \parentheses[\big]{ \Sigma_{A_{\vartriangle}}^+, \sigma_{A_{\vartriangle}} } \) associated to \( F \) and \( \mathcal{C} \) defined in Proposition~\ref{prop:one-sided subshift of finite type associated with expanding Thurston map}, and the factor map \( \pi_{\vartriangle} \colon \Sigma_{A_{\vartriangle}}^+ \to S^2 \) defined in~\eqref{eq:prop:one-sided subshift of finite type associated with expanding Thurston map:factor map}:
  
          The function \( \Psi \circ \pi_{\vartriangle} \) is cohomologous to a constant multiple of an integer-valued continuous function in \( C\parentheses[\big]{ \Sigma_{A_{\vartriangle}}^+, \mathbb{C} } \), i.e., \( \Psi \circ \pi_{\vartriangle} = KM + \varpi \circ \sigma_{A_{\vartriangle}} - \varpi \) for some \( K \in \mathbb{C} \), \( M \in C\parentheses[\big]{ \Sigma_{A_{\vartriangle}}^+, \mathbb{Z} } \), and \( \varpi \in C\parentheses[\big]{ \Sigma_{A_{\vartriangle}}^+, \mathbb{C} } \).
      \end{enumerate}
  \end{proposition}
  
  This proposition is used in Lemma~\ref{lem:partition function estimate:symbolic coding:bounded imaginary bounded cases}~\ref{item:lem:partition function estimate:symbolic coding:bounded imaginary bounded cases:intermediate part} to apply the complex Ruelle--Perron--Frobenius theorem \cite[Theorem~2]{pollicott1984complex}.

  
  \begin{theorem}[Parry \& Pollicott {\cite[Theorem~2(ii)]{pollicottZetaFunctionsPeriodic1990}}] \label{thm:complex Ruelle--Perron--Frobenius}
      Consider a finite set of states $S$ and a transition matrix $A \colon S \times S \to \{0, 1\}$. 
      Let $\parentheses[\big]{ \Sigma_{A}^+, \sigma_{A} }$ be the one-sided subshift of finite type defined by $A$.
      Fix $\metricexpshiftspace \in (0, 1)$ and equip the space $\Sigma_{A}^+$ with the metric $\metriconshiftspace$ defined in \eqref{eq:def:metric on shift space}.
  
      Consider a complex-valued \holder continuous function $\psi = u + iv \in C^{0, \holderexp}\parentheses[\big]{ \parentheses[\big]{ \Sigma_A^+, \metriconshiftspace }, \mathbb{C} }$ with $u, v \in C^{0, \holderexp}\parentheses[\big]{ \Sigma_A^+, \metriconshiftspace }$.
      If $\parentheses[\big]{ \Sigma_{A}^+, \sigma_{A} }$ is topologically mixing and $v$ is non-lattice, then the spectrum of $\mathcal{L}_{\psi}$ is contained in a disc of radius strictly smaller than $e^{P(\sigma, u)}$.
  \end{theorem}

  Applying the complex Ruelle--Perron--Frobenius theorem (Theorem~\ref{thm:complex Ruelle--Perron--Frobenius}) to the subshift $\parentheses[\big]{ \Sigma^{+}_{A_{\vartriangle}}, \sigma_{A_{\vartriangle}} }$ yields the following estimates for a bounded imaginary part.
  
  \begin{lemma} \label{lem:partition function estimate:symbolic coding:bounded imaginary bounded cases}
      Let $f$, $\mathcal{C}$, $d$, $\Lambda$, $\holderexp$, $\potential$, $\rootpressure$ satisfy the Assumptions in Section~\ref{sec:The Assumptions}.
      We assume that $f(\mathcal{C}) \subseteq \mathcal{C}$.
      Let $\parentheses[\big]{ \Sigma^{+}_{A_{\vartriangle}}, \sigma_{A_{\vartriangle}} }$ be the one-sided subshift of finite type associated to $f$ and $\mathcal{C}$ defined in Proposition~\ref{prop:one-sided subshift of finite type associated with expanding Thurston map}, and let $\pi_{\vartriangle} \colon \Sigma^{+}_{A_{\vartriangle}} \mapping S^2$ be defined in \eqref{eq:prop:one-sided subshift of finite type associated with expanding Thurston map:factor map}.
      Fix $\metricexpshiftspace \in (0, 1)$ and equip the space $\Sigma^{+}_{A_{\vartriangle}}$ with the metric $\metriconshiftspace$ defined in \eqref{eq:def:metric on shift space}.
      Then the following statements hold:
      \begin{enumerate}
          \smallskip
  
          \item     \label{item:lem:partition function estimate:symbolic coding:bounded imaginary bounded cases:local part}
              There exist constants $t_{0} > 0$, $\theta \in (0, 1)$, and $C \geqslant 0$ such that $P(\sigma_{A_{\vartriangle}}, (-\rootpressure + \imaginary t)\potential\circ\pi_{\vartriangle})$ is well-defined and
              \[
                   \abs[\Big]{ \shiftpartifun{ -\rootpressure + \imaginary t }[n][\sigma_{A_{\vartriangle}}][\potential\circ\pi_{\vartriangle}] - \myexp[\big]{ n P(\sigma_{A_{\vartriangle}}, (-\rootpressure + \imaginary t)\potential\circ\pi_{\vartriangle}) } }  \leqslant C \theta^{n}
              \]
              for all $t \in (-t_{0}, t_{0})$ and all $n \in \n$.
          \item     \label{item:lem:partition function estimate:symbolic coding:bounded imaginary bounded cases:intermediate part}
              If $\potential$ is not cohomologous to a constant in $C(S^{2})$, then for each $\varepsilon > 0$ and each compact set $K \subseteq \real$, there exist $\vartheta \in (0, 1)$ and $C \geqslant 0$ such that
              \[
                  \abs[\Big]{ \shiftpartifun{ -\rootpressure + \imaginary t }[n][\sigma_{A_{\vartriangle}}][\potential\circ\pi_{\vartriangle}] }
                  \leqslant C \vartheta^{n}
              \]
              for all $t \in K \mysetminus (-\varepsilon, \varepsilon)$ and all $n \in \n$.
      \end{enumerate}
  \end{lemma}
  
  
  \begin{proof}
      \ref{item:lem:partition function estimate:symbolic coding:bounded imaginary bounded cases:local part}
      Note that $P(\sigma_{A_{\vartriangle}}, (-\rootpressure + \imaginary t)\potential\circ\pi_{\vartriangle})$ is well-defined when $\abs{t}$ is sufficiently small.
      The statement follows from the argument in the proof of \cite[Theorem~5.5~(ii)]{pollicottZetaFunctionsPeriodic1990}.
  
      \ref{item:lem:partition function estimate:symbolic coding:bounded imaginary bounded cases:intermediate part}
      \def\expspectrum{\rho}
      Assume that $\potential$ is not cohomologous to a constant in $C(S^{2})$. 
      Consider $t \in K \mysetminus (-\varepsilon, \varepsilon)$ and denote $\psi_t \define (-\rootpressure + \imaginary t) \potential\circ\pi_{\vartriangle}$.
      Set $\delta \define \frac{1}{2} \log(\metricexpshiftspace^{-\holderexp}) > 0$.
  
      Since $P(\sigma_{A_{\vartriangle}}, -\rootpressure \potential\circ\pi_{\vartriangle}) = P(f, -\rootpressure \potential) = 0$, by applying Corollary~\ref{coro:bound of partition function for subshift of finite type via Ruelle lemma} to the subshift $\parentheses[\big]{ \Sigma^{+}_{A_{\vartriangle}}, \sigma_{A_{\vartriangle}} }$ and the potential $\potential\circ\pi_{\vartriangle}$, we obtain that
      \begin{equation} \label{eq:lem:partition function estimate:symbolic coding:bounded imaginary bounded cases:intermediate part:temporary estimate}
          \begin{aligned}
          \abs[\Big]{ \shiftpartifun{ -\rootpressure + \imaginary t }[n][\sigma_{A_{\vartriangle}}][\potential\circ\pi_{\vartriangle}] }
          &\leqslant C_\delta \abs{t} \sum_{m = 1}^{n} \norm[\big]{ \mathcal{L}^{n - m}_{\psi_t} }_{C^{0,\holderexp}} \parentheses[\big]{ \metricexpshiftspace^{\holderexp} e^{P(\sigma_{A_{\vartriangle}}, - \rootpressure \potential\circ\pi_{\vartriangle}) + \delta} }^{m} \\
          &= C_\delta \abs{t} \sum_{m = 1}^{n} \norm[\big]{ \mathcal{L}^{n - m}_{\psi_t} }_{C^{0,\holderexp}} \metricexpshiftspace^{\holderexp m / 2}
          \end{aligned}
      \end{equation}
      for some constant $C_\delta > 0$ depending only on $f$, $\mathcal{C}$, $d$, $\holderexp$, $\potential$, \( b_0 \), \( \metricexpshiftspace \), and \( \delta \).
  
      We now bound the operator norm $\norm[\big]{ \mathcal{L}^{k}_{\psi_t} }_{C^{0,\holderexp}}$ for $k \in \n$ and $t \in K \mysetminus (-\varepsilon, \varepsilon)$ by establishing a uniform estimate on the spectral radius of $\mathcal{L}_{\psi_t}$.
      Since $\phi$ is not cohomologous to a constant, Proposition~\ref{prop:cohomologous_equivalence} and Theorem~\ref{thm:complex Ruelle--Perron--Frobenius} together imply that the spectral radius of $\mathcal{L}_{\psi_t}$, denoted by $\rho(\mathcal{L}_{\psi_t})$, is strictly less than $e^{P(\sigma_{A_{\vartriangle}}, \Re{\psi_t})} = 1$ for each $t \in K \mysetminus (-\varepsilon, \varepsilon)$.
      Straightforward calculations show that $t \mapsto \mathcal{L}_{\psi_t}$ is continuous (as a map from $\real$ to the space of bounded linear operators on $C^{0,\holderexp} \parentheses[\big]{ \parentheses[\big]{ \Sigma_{A}^{+}, \metriconshiftspace }, \cx }$).
      A classical result from operator theory asserts that the spectrum is upper semi-continuous on the space of bounded linear operators on a Banach space (cf.~\cite[Remark~3.3, pp.~208--209]{katoPerturbationTheoryLinear1995}).
      It follows that the function $t \mapsto \rho(\mathcal{L}_{\psi_t})$ is upper semi-continuous, and hence attains its maximum on the compact set $K \mysetminus (-\varepsilon, \varepsilon)$.
      Thus we have
      \[
          \rho_0 \define \max_{t \in K \mysetminus (-\varepsilon, \varepsilon)} \rho(\mathcal{L}_{\psi_t}) < 1.
      \]
      According to the spectral radius formula $\rho(\mathcal{L}_{\psi_t}) = \lim_{n \to \infty} \norm[\big]{ \mathcal{L}^{n}_{\psi_t} }_{C^{0,\holderexp}}^{1/n}$, we can find constants $\expspectrum \in (\rho_0, 1)$ and $C_{\expspectrum} \geqslant  1$ such that
      \begin{equation} \label{eq:lem:partition function estimate:symbolic coding:bounded imaginary bounded cases:intermediate part:operator norm estimate}
          \norm[\big]{ \mathcal{L}^{k}_{\psi_t} }_{C^{0,\holderexp}} 
          \leqslant C_{\expspectrum} \, \expspectrum^k \quad \text{for all } k \in \n \text{ and } t \in K \mysetminus (-\varepsilon, \varepsilon).
      \end{equation}
      
      Substituting the estimate \eqref{eq:lem:partition function estimate:symbolic coding:bounded imaginary bounded cases:intermediate part:operator norm estimate} into \eqref{eq:lem:partition function estimate:symbolic coding:bounded imaginary bounded cases:intermediate part:temporary estimate}, we get
      \begin{align*}
          \abs[\Big]{ \shiftpartifun{ -\rootpressure + \imaginary t }[n][\sigma_{A_{\vartriangle}}][\potential\circ\pi_{\vartriangle}] }
          \leqslant C_\delta C_{\expspectrum} |t| \sum_{m = 1}^{n} \expspectrum^{n-m} \metricexpshiftspace^{\holderexp m / 2} 
          \leqslant C' n \parentheses[\big]{ \max\set[\big]{ \expspectrum, \metricexpshiftspace^{\holderexp / 2} } }^{n},
      \end{align*}
      where $C' \define C_\delta C_{\expspectrum} \sup_{t \in K}|t|$. 
      Therefore, by setting $\vartheta \define \frac{1}{2} \parentheses[\big]{ 1 +  \max\set[\big]{ \expspectrum, \metricexpshiftspace^{\holderexp / 2} } } \in (0, 1)$ and $C \define C' \sup_{n \in \n} \set{n \vartheta^{n}} < +\infty$, we establish the desired estimate.
  \end{proof}

Combining Lemmas~\ref{lem:partition function estimate:symbolic coding:bounded imaginary bounded cases}~and~\ref{lem:relation between partition function of expanding Thurston map and tile coding}, we obtain the following estimates for the partition function in the case of bounded imaginary part.

\begin{lemma} \label{lem:partition function estimate:expanding Thurston map:bounded imaginary bounded cases}
    Let $f$, $d$, $\phi$, $\rootpressure$, $\alpha$ satisfy the Assumptions in Section~\ref{sec:The Assumptions}.
    We assume that $f(\mathcal{C}) \subseteq \mathcal{C}$ and no $1$-tile in $\Tile{1}$ joins opposite sides of $\mathcal{C}$.
    Define $\normpotential \define \phi - \alpha$.
    Then the following statements hold:
    \begin{enumerate}
        \smallskip
        
        \item     \label{item:lem:partition function estimate:expanding Thurston map:bounded imaginary bounded cases:local part}
            There exist constants $t_{0} > 0$, $\theta \in (0, 1)$, and $C \geqslant 0$ such that $P(\sigma_{A_{\vartriangle}}, (-\rootpressure + \imaginary t)\normpotential)$ is well-defined and
            \[
                 \abs[\Big]{ \normpartifun{-\rootpressure + \imaginary t} - e^{n P(f, (- \rootpressure + \imaginary t)\normpotential)} }
                \leqslant C \theta^{n} e^{ n P(f, -\rootpressure \normpotential) }
            \]
            for all $t \in (-t_{0}, t_{0})$ and all $n \in \n$.

        \smallskip

        \item     \label{item:lem:partition function estimate:expanding Thurston map:bounded imaginary bounded cases:intermediate part}
            If $\phi$ is not cohomologous to a constant in $C(S^2)$, then for each $\varepsilon > 0$ and each compact set $K \subseteq \real$, there exist constants $\vartheta \in (0, 1)$ and $C \geqslant 0$ such that
            \[
                \abs[\big]{ \normpartifun{-\rootpressure + \imaginary t} }
                \leqslant C \vartheta^{n} e^{ n P(f, -\rootpressure \normpotential) }.
            \]
            for all $t \in K \mysetminus (-\varepsilon, \varepsilon)$ and all $n \in \n$.
    \end{enumerate}
\end{lemma}

The proof of Lemma~\ref{lem:partition function estimate:expanding Thurston map:bounded imaginary bounded cases} parallels that of Proposition~\ref{prop:partition function estimate:expanding Thurston map:unbounded imaginary}, substituting the estimates from Lemma~\ref{lem:partition function estimate:symbolic coding:bounded imaginary bounded cases} for those from Proposition~\ref{prop:partition function estimate:symbolic coding:unbounded imaginary}; the details are omitted.

%% file: section/Proof.tex

\section{Proof of the main theorem}
\label{sec:Proof of the main theorem}

In this section, we complete the proof of Theorem~\ref{thm:main theorem} using the results from the previous section.
In Subsection~\ref{sub:Notation and assumptions}, we introduce the notation and assumptions that are used throughout this section.
In Subsection~\ref{sub:Auxiliary estimates}, we establish some auxiliary estimates (Lemma~\ref{lem:error of shorter primitive periodic orbits} and Proposition~\ref{prop:estimate of cardinality periodic orbits}).
In Subsection~\ref{sub:Approximation argument}, we use these estimates to establish our main result.

This section is dedicated to the proof of Theorem~\ref{thm:main theorem}. The argument follows a standard strategy for such counting problems and is organized into three main steps. First, in Subsection~\ref{sub:Auxiliary estimates}, we replace the sharp count of orbits with a smoothed version using a test function and relate this to a sum over all periodic points. Second, we employ Fourier analysis and the decay estimates from Section~\ref{sec:Pressure function and partition function estimates} to establish the asymptotic behavior of this smoothed count. Finally, in Subsection~\ref{sub:Approximation argument}, we complete the proof by using an approximation argument, where the sharp indicator function is bounded by smooth functions from above and below.

\subsection{Notation and assumptions}%
\label{sub:Notation and assumptions}

Throughout this section, let $f$, $\mathcal{C}$, $d$, $\potential$, $\rootpressure$, $\alpha$ satisfy the Assumptions in Section~\ref{sec:The Assumptions}.
We assume that $f(\mathcal{C}) \subseteq \mathcal{C}$ and no $1$-tile in $\Tile{1}$ joins opposite sides of $\mathcal{C}$.
Suppose that $\potential$ is not cohomologous to a constant in $C(S^2)$.
We define $\normpotential \define \potential - \alpha$.
Let $\sigma > 0$ be defined as in \eqref{eq:def:variance of potential}.
Let $K \subseteq \real$ be a compact set and $\sequen{I_{n}}$ be a sequence of intervals contained in $K$.

For each $n \in \n$, we denote by $p_{n}$ the midpoint of the interval $I_{n}$ and by $\ell_{n}$ the length of $I_{n}$.
Moreover, we assume that $\sequen{\ell_{n}^{-1}}$ has sub-exponential growth.
Then we can write
\begin{equation}    \label{eq:cardinality of periodic orbits with constraints}
	\primeorbitcard
	= \sum_{\tau \in \primeorbit} \indicator{I_{n}}\parentheses[\big]{ l_{f, \, \potential}(\tau) - n \alpha }  
	= \sum_{\tau \in \primeorbit} \indicator{[- \frac{1}{2}, \frac{1}{2}]}\parentheses[\big]{ \ell_{n}^{-1}\parentheses[\big]{ l_{f, \, \potential}(\tau) - n \alpha - p_{n} }  },
\end{equation}
where $l_{f, \, \potential}(\tau) = \sum_{y \in \tau} \potential(y)$.

\subsection{Auxiliary estimates}%
\label{sub:Auxiliary estimates}

In this subsection, we fix a non-negative function $\psi \in C^{4}(\real, \nonnegreal)$ with compact support.
For each $n \in \n$ we consider the auxiliary counting number
\begin{equation}    \label{eq:def:cardinality prime orbits}
	\auxprimeorbitcounting \define \sum_{\tau \in \primeorbit} \psi \parentheses[\big]{ \ell_{n}^{-1}\parentheses[\big]{ l_{f, \, \potential}(\tau) - n \alpha - p_{n} } }. 
\end{equation}

We study the asymptotic behavior of $\auxprimeorbitcounting$ to establish our main result by using an approximation argument in Subsection~\ref{sub:Approximation argument}.

We start by transforming the summation over $\primeorbit$, which represents the primitive periodic orbits of period $n$, into a summation over the set of fixed points of the iterated map $f^{n}$. 
Each primitive periodic orbit corresponds to $n$ distinct points in this set. 
However, this set also contains points belonging to primitive periodic orbits of shorter lengths. 
The following lemma establishes a bound for the error introduced by these shorter primitive periodic orbits. 
For each $n \in \n$, define
\begin{equation}    \label{eq:def:cardinality periodic orbits}
	\auxperiodicorbitcounting \define \frac{1}{n} \sum_{f^{n}(x) = x} \psi \parentheses[\big]{ \ell_{n}^{-1}\parentheses[\big]{ S_{n}^{f}\potential(x) - n \alpha - p_{n} } }.
\end{equation}

\begin{lemma} \label{lem:error of shorter primitive periodic orbits}
	Consider a non-negative function $\psi \in C^{4}(\real, \nonnegreal)$ with compact support.
	Following the assumptions in Subsection~\ref{sub:Notation and assumptions}, we have that for each $\eta > 0$,
	\[
		\auxprimeorbitcounting = \auxperiodicorbitcounting + \mathcal{O}\parentheses[\Big]{ e^{(P(f, -\rootpressure\normpotential) + \eta) n/2} }     \quad \text{as } n \to +\infty.
	\]
	Here $\auxprimeorbitcounting$ and $\auxperiodicorbitcounting$ are defined by \eqref{eq:def:cardinality prime orbits} and \eqref{eq:def:cardinality periodic orbits}, respectively.
\end{lemma}
The proof follows that of \cite[Lemma~4.1]{sharpStatisticsMultipliersHyperbolic2022}; we include it for completeness.
\begin{proof}
	We say that a fixed point $x$ of $f^{n}$ is non-primitive if there exists a proper divisor $q$ of $n$ such that $f^{q}(x) = x$.
	Then we have
	\begin{align*}
		\auxperiodicorbitcounting - \auxprimeorbitcounting 
		&= \frac{1}{n} \sum_{ \substack{ f^{n}(x) = x \\ \text{non-primitive} } } \psi \parentheses[\big]{ \ell_{n}^{-1}\parentheses[\big]{ S_{n}^{f}\potential(x) - n \alpha - p_{n} } } \\
		&= \frac{1}{n} \sum_{\substack{ q \mid n \\ q \leqslant n/2 } } \sum_{f^{q}(x) = x} \psi \parentheses[\big]{ \ell_{n}^{-1}\parentheses[\big]{ S_{n}^{f}\potential(x) - n \alpha - p_{n} } } \\
		&= \frac{1}{n} \sum_{\substack{ q \mid n \\ q \leqslant n/2 } } \sum_{f^{q}(x) = x} \frac{ \psi \parentheses[\big]{ \ell_{n}^{-1}\parentheses[\big]{ S_{n}^{f}\normpotential(x) - p_{n} } } }{ e^{-\rootpressure S_{q}^{f}\normpotential(x)} } e^{-\rootpressure S_{q}^{f}\normpotential(x)}.
	\end{align*}
	Since $\psi$ has compact support, it suffices to consider periodic points which satisfy $\ell_{n}^{-1}\parentheses[\big]{ S_{n}^{f}\normpotential(x) - p_{n} } \in \supp{\psi}$, i.e., $S_{n}^{f}\normpotential(x) \in p_{n} + \ell_{n} \supp{\psi}$.
	In particular, $S_{q}^{f}\normpotential(x)$ is bounded for these periodic points (recall from Subsection~\ref{sub:Notation and assumptions} that the intervals $I_{n}$ are contained in a compact set $K \subseteq \real$).
	Hence, for a non-primitive periodic point $x$ satisfying $f^{q}(x) = x$ for $q$ in the summation above, we get that $S_{q}^{f}\normpotential(x) = \frac{q}{n} S_{n}^{f}\normpotential(x)$ and thus $e^{-\rootpressure S_{q}^{f}\normpotential(x)}$ is uniformly bounded away from $0$ for $q$ in the summation above.
	Applying Lemma~\ref{lem:partition function estimate:expanding Thurston map:bounded imaginary bounded cases}~\ref{item:lem:partition function estimate:expanding Thurston map:bounded imaginary bounded cases:local part}, we deduce that for each $\eta > 0$,
	\begin{align*}
		&\frac{1}{n} \sum_{\substack{ q \mid n \\ q \leqslant n/2 } } \sum_{f^{q}(x) = x} \frac{ \psi \parentheses[\big]{ \ell_{n}^{-1}\parentheses[\big]{ S_{n}^{f}\normpotential(x) - p_{n} } } }{ e^{-\rootpressure S_{q}^{f}\normpotential(x)} } e^{-\rootpressure S_{q}^{f}\normpotential(x)}
		= \mathcal{O}\parentheses[\bigg]{ \frac{1}{n} \uniformnorm{\psi} \sum_{ q \leqslant n/2 } \sum_{f^{q}(x) = x} e^{-\rootpressure S_{q}^{f}\normpotential(x)} }  \\
		&= \mathcal{O}\parentheses[\bigg]{ \frac{1}{n} \uniformnorm{\psi} \sum_{ q \leqslant n/2 } \normpartifun{-\rootpressure }[q] } 
		= \mathcal{O}\parentheses[\bigg]{ \frac{1}{n} \sum_{ q \leqslant n/2 } e^{ (P(f, -\rootpressure \normpotential) + \eta) q } } 
		= \mathcal{O}\parentheses[\Big]{ e^{ (P(f, -\rootpressure\normpotential) + \eta) n / 2 } }
	\end{align*}
	as $n \to +\infty$.
	This completes the proof.
\end{proof}

For each $n \in \n$, we define
\begin{equation}    \label{eq:def:normalized testin function}
	\psi_{n}(x) \define \psi \parentheses[\big]{ \ell_{n}^{-1} (x - p_{n}) } e^{\rootpressure (x - p_n)}.
\end{equation}
Note that $\psi_n \in C^{4}(\real, \nonnegreal)$ has compact support.
Recall that $\normpotential = \potential - \alpha$.
In this notation we have
\begin{equation}    \label{eq:cardinality periodic orbits reduced form}
	\auxperiodicorbitcounting = \frac{1}{n} \sum_{f^{n}(x) = x} \psi_{n}\parentheses[\big]{ S_{n}^{f}\normpotential(x) } e^{-\rootpressure \parentheses[\big]{ S_{n}^{f}\normpotential(x) - p_{n} } }.
\end{equation}

We use Fourier transform to relate $\auxperiodicorbitcounting$ to partition functions so that we can apply the estimates established in Subsection~\ref{sub:Decay estimates of the partition function}.

\begin{proposition}    \label{prop:estimate of cardinality periodic orbits}
	Consider a non-negative function $\psi \in C^{4}(\real, \nonnegreal)$ with compact support.
	Under the assumptions in Subsection~\ref{sub:Notation and assumptions}, we have that
	\[
		\auxperiodicorbitcounting \sim e^{\rootpressure p_n} \frac{\int_{\real} \! \psi_{n}(x) \,\mathrm{d}x}{\sqrt{2\pi} \, \sigma} \frac{e^{P(f, -\rootpressure \normpotential)n}}{n^{3/2}} \qquad \text{as } n \to +\infty.
	\]
	Here $\auxperiodicorbitcounting$ is defined by \eqref{eq:def:cardinality periodic orbits} and $\psi_{n}$ is defined by \eqref{eq:def:normalized testin function}.
\end{proposition}
\begin{proof}
	For each $n \in \n$ we define \[
		A(n) \define \abs[\bigg]{ \frac{ \ell_{n}^{-1} e^{-\rootpressure p_{n}} \sigma \sqrt{2 \pi n^{3}} }{ e^{P(f, -\rootpressure \normpotential)n} } \auxperiodicorbitcounting \, - \, \ell_{n}^{-1} \! \int_{\real} \! \psi_{n}(x) \,\mathrm{d}x }.
	\]
	The integral $\ell_{n}^{-1} \int_{\real} \! \psi_{n}(x) \,\mathrm{d}x = \int_{\real} \! \psi(y) e^{\rootpressure \ell_{n} y} \,\mathrm{d}y$ is uniformly bounded away from $0$ and $+\infty$ for $n \in \n$ as $\psi$ has compact support. 
	Hence it suffices to show that $A(n) \to 0$ as $n \to +\infty$.

	By applying the Fourier inversion theorem, for each $n \in \n$ we obtain
	\begin{equation}    \label{eq:prop:estimate of partition function with periodic orbits:Fourier inversion formula}
		\psi_{n}(x) e^{-\rootpressure (x - p_{n})} = e^{\rootpressure p_{n}} \int_{\real} \! \widehat{\psi}_{n}(t) e^{(-\rootpressure + 2 \pi \imaginary t) x} \,\mathrm{d}t.
	\end{equation}

	\smallskip

	\emph{Claim~1.} For each $n \in \n$, \[
		A(n) \leqslant \frac{1}{\sqrt{2 \pi}} \int_{\real} \, \abs[\bigg]{ \frac{\ell_{n}^{-1}}{e^{P(f, -\rootpressure \normpotential)n}} \widehat{\psi}_{n}\parentheses[\Big]{ \frac{t}{2 \pi \sigma \sqrt{n}} }  \normpartifun[\Big]{-\rootpressure + \frac{\imaginary t}{\sigma \sqrt{n}}} \, - \, \ell_{n}^{-1} e^{-\frac{t^{2}}{2}} \! \int_{\real} \! \psi_{n}(x) \,\mathrm{d}x } \,\mathrm{d}t.
	\]

	\smallskip

	By \eqref{eq:cardinality periodic orbits reduced form} and \eqref{eq:prop:estimate of partition function with periodic orbits:Fourier inversion formula}, for each $n \in \n$ we have
	\begin{align*}
		\frac{ \ell_{n}^{-1} e^{-\rootpressure p_{n}} \sigma \sqrt{2 \pi n^{3}} }{ e^{P(f, -\rootpressure \normpotential)n} } \auxperiodicorbitcounting 
		&= \frac{ \ell_{n}^{-1} e^{-\rootpressure p_{n}} \sigma \sqrt{2 \pi n} }{ e^{P(f, -\rootpressure \normpotential)n} } \sum_{f^{n}(x) = x} \psi_{n}\parentheses[\big]{ S_{n}^{f}\normpotential(x) } e^{-\rootpressure \parentheses[\big]{ S_{n}^{f}\normpotential(x) - p_{n} } } \\
		&= \frac{ \ell_{n}^{-1} \sigma \sqrt{2 \pi n} }{ e^{P(f, -\rootpressure \normpotential)n} } \sum_{f^{n}(x) = x}  \int_{\real} \! \widehat{\psi}_{n}(t) e^{(-\rootpressure + 2 \pi \imaginary t) S_{n}^{f}\normpotential(x)} \,\mathrm{d}t\\
		&= \frac{ \ell_{n}^{-1} }{ \sqrt{2 \pi} e^{P(f, -\rootpressure \normpotential)n} } \int_{\real} \! \widehat{\psi}_{n}\parentheses[\Big]{ \frac{t}{ 2 \pi \sigma \sqrt{n} }  } \sum_{f^{n}(x) = x} e^{\parentheses[\big]{ -\rootpressure + \frac{\imaginary t}{\sigma \sqrt{n}} }  S_{n}^{f}\normpotential(x)} \,\mathrm{d}t \\
		&= \frac{ \ell_{n}^{-1} }{ \sqrt{2 \pi} e^{P(f, -\rootpressure \normpotential)n} } \int_{\real} \! \widehat{\psi}_{n}\parentheses[\Big]{ \frac{t}{ 2 \pi \sigma \sqrt{n} }  } \normpartifun[\Big]{-\rootpressure + \frac{\imaginary t}{\sigma \sqrt{n}}} \,\mathrm{d}t.
	\end{align*}
	Here $\normpartifun{\cdot}$ is defined in Definition~\ref{def:partition function}.
	Using the identity $\sqrt{2\pi} = \int_{\real} \! e^{-t^{2}/2} \,\mathrm{d}t$ we deduce that
	\[
		A(n) = \frac{1}{\sqrt{2\pi}}  \abs[\bigg]{ \frac{ \ell_{n}^{-1} }{ e^{P(f, -\rootpressure \normpotential)n} } \int_{\real} \! \widehat{\psi}_{n}\parentheses[\Big]{ \frac{t}{ 2 \pi \sigma \sqrt{n} }  } \normpartifun[\Big]{-\rootpressure + \frac{\imaginary t}{\sigma \sqrt{n}}} \,\mathrm{d}t  \, - \, \ell_{n}^{-1} \! \int_{\real} \! \psi_{n}(x) \,\mathrm{d}x \int_{\real} \! e^{-t^{2}/2} \,\mathrm{d}t }.
	\]
	Then Claim~1 follows immediately from the integral inequality for absolute values.

	\smallskip

	We now consider the following three quantities:
	\begin{align*}
		A_1(n) &\define \int_{ \abs{t} < \varepsilon \sigma \sqrt{n} } \, \abs[\bigg]{ \frac{\ell_{n}^{-1}}{e^{P(f, -\rootpressure \normpotential)n}} \widehat{\psi}_{n}\parentheses[\Big]{ \frac{t}{2 \pi \sigma \sqrt{n}} }  \normpartifun[\Big]{-\rootpressure + \frac{\imaginary t}{\sigma \sqrt{n}}} \, - \, \ell_{n}^{-1} e^{-\frac{t^{2}}{2}} \! \int_{\real} \! \psi_{n}(x) \,\mathrm{d}x } \,\mathrm{d}t,  \\
		A_2(n) &\define \int_{ \abs{t} \geqslant \varepsilon \sigma \sqrt{n} } \, \abs[\bigg]{ \frac{\ell_{n}^{-1}}{e^{P(f, -\rootpressure \normpotential)n}} \widehat{\psi}_{n}\parentheses[\Big]{ \frac{t}{2 \pi \sigma \sqrt{n}} }  \normpartifun[\Big]{-\rootpressure + \frac{\imaginary t}{\sigma \sqrt{n}}} } \,\mathrm{d}t,  \\
		A_3(n) &\define \int_{ \abs{t} \geqslant \varepsilon \sigma \sqrt{n} } \, \abs[\bigg]{ \ell_{n}^{-1} e^{-\frac{t^{2}}{2}} \! \int_{\real} \! \psi_{n}(x) \,\mathrm{d}x } \,\mathrm{d}t.
	\end{align*}
	Here $\varepsilon \in (0, 1)$ is chosen to be smaller than $\min\set[\Big]{ \delta , \frac{\sigma^{2}}{4 C_{\delta}}, t_0 }$, where the constants $\delta$ and $C_{\delta}$ are given by Lemma~\ref{lem:Taylor extension of pressure function in imaginary part}, and the constant $t_0$ is given by Lemma~\ref{lem:partition function estimate:expanding Thurston map:bounded imaginary bounded cases}~\ref{item:lem:partition function estimate:expanding Thurston map:bounded imaginary bounded cases:local part}.
	It follows from Claim~1 that 
	\[
		A(n) \leqslant \frac{1}{\sqrt{2\pi}} \parentheses[\big]{ A_1(n) + A_2(n) + A_3(n)} \qquad \text{for each } n \in \n.
	\]
	Thus it suffices to show that $\lim_{n \to +\infty} A_{i}(n) = 0$ for each $i \in \{1, 2, 3\}$.

	\smallskip
	
	\emph{Claim~2.}  $\lim_{n \to +\infty} A_{1}(n) = 0$.
	
	\smallskip

	For each $n \in \n$, we define \[
		\widetilde{A}_{1}(n) \define \int_{ \abs{t} < \varepsilon \sigma \sqrt{n} } \, \abs[\bigg]{ \ell_{n}^{-1} \widehat{\psi}_{n}\parentheses[\Big]{ \frac{t}{2 \pi \sigma \sqrt{n}} } e^{n\parentheses[\big]{ P\parentheses[\big]{ f, \parentheses{ -\rootpressure + \frac{\imaginary t}{\sigma\sqrt{n}} } \normpotential } - P(f, -\rootpressure \normpotential) } }  \, - \, \ell_{n}^{-1} e^{-\frac{t^{2}}{2}} \! \int_{\real} \! \psi_{n}(x) \,\mathrm{d}x } \,\mathrm{d}t.
	\]	
	Then by Lemma~\ref{lem:partition function estimate:expanding Thurston map:bounded imaginary bounded cases}~\ref{item:lem:partition function estimate:expanding Thurston map:bounded imaginary bounded cases:local part}, there exists $\theta \in (0, 1)$ such that as $n \to +\infty$,
	\[
		\begin{split}
			A_1(n) &= \widetilde{A}_{1}(n) + \mathcal{O}\parentheses[\big]{ \varepsilon \sigma \sqrt{n} \, \ell_{n}^{-1} \uniformnorm{\widehat{\psi}_{n}} \theta^{n} },
		\end{split}
	\]
	where $\ell_{n}^{-1} \uniformnorm{\widehat{\psi}_{n}} \leqslant \ell_{n}^{-1} \int_{\real} \! \psi_{n}(x) \,\mathrm{d}x = \int_{\real} \! \psi(y) e^{\rootpressure \ell_{n} y} \,\mathrm{d}y$ is uniformly bounded for $n \in \n$ since $\psi$ has compact support. 
	Hence, to establish Claim~2, it suffices to show that $\lim_{n \to +\infty} \widetilde{A}_{1}(n) = 0$.
	
	By the triangle inequality we have that for each $n \in \n$, 
	\[
		\widetilde{A}_{1}(n) \leqslant \int_{ \abs{t} < \varepsilon \sigma \sqrt{n} } \,  g_{n}(t) \,\mathrm{d}t + \int_{ \abs{t} < \varepsilon \sigma \sqrt{n} } \,  h_{n}(t) \,\mathrm{d}t,
	\]
	where for each $t \in \real$,
	\begin{align*}
		g_{n}(t) &\define \abs[\bigg]{ \ell_{n}^{-1} \widehat{\psi}_{n}\parentheses[\Big]{ \frac{t}{2 \pi \sigma \sqrt{n}} } e^{n\parentheses[\big]{ P\parentheses[\big]{ f, \parentheses{ -\rootpressure + \frac{\imaginary t}{\sigma\sqrt{n}} } \normpotential } - P(f, -\rootpressure \normpotential) } }  \, - \, \ell_{n}^{-1} \widehat{\psi}_{n}\parentheses[\Big]{ \frac{t}{2 \pi \sigma \sqrt{n}} } e^{-\frac{t^{2}}{2}} },  \\
		h_{n}(t) &\define \abs[\bigg]{ e^{-\frac{t^{2}}{2}} \ell_{n}^{-1} \widehat{\psi}_{n}\parentheses[\Big]{ \frac{t}{2 \pi \sigma \sqrt{n}} } \, - \, e^{-\frac{t^{2}}{2}} \ell_{n}^{-1} \int_{\real} \! \psi_{n}(x) \,\mathrm{d}x }.
	\end{align*}
	By \eqref{eq:def:normalized testin function} we have that
	\[
		\ell_{n}^{-1} \widehat{\psi}_{n}\parentheses[\Big]{ \frac{t}{2 \pi \sigma \sqrt{n}} } = \ell_{n}^{-1} \int_{\real} \! \psi_{n}(x) e^{\frac{\imaginary x t}{\sigma \sqrt{n}} } \,\mathrm{d} x = \int_{\real} \! \psi(y) e^{\rootpressure \ell_{n} y} e^{\frac{\imaginary t}{\sigma\sqrt{n}} (p_{n} + \ell_{n} y) } \,\mathrm{d} y,
	\]
	which is uniformly bounded for $n \in \n$ and $t \in \real$ since $\psi$ has compact support.
	Then it follows from Lebesgue's dominated convergence theorem that for each $t \in \real$, \[
		0 \leqslant \lim_{n \to +\infty} h_{n}(t) \leqslant \lim_{n \to +\infty} e^{-\frac{t^{2}}{2}} \int_{\real} \! \psi(y) e^{\rootpressure \ell_{n} y} \abs[\Big]{ e^{\frac{\imaginary t}{\sigma\sqrt{n}} (p_{n} + \ell_{n} y) } - 1 }  \,\mathrm{d} y = 0.
	\]
	Since $\int_{\real} \! e^{-t^{2}/2 } \,\mathrm{d}t = \sqrt{2\pi} < +\infty$, by Lebesgue's dominated convergence theorem we conclude that $\lim_{n \to +\infty} \int_{\real} h_{n}(t) \,\mathrm{d}t = 0$.
	By \eqref{eq:lem:Taylor extension of pressure function in imaginary part:Taylor expansion of pressure function} in Lemma~\ref{lem:Taylor extension of pressure function in imaginary part}, for each $t \in \real$, \[
		\lim_{n \to +\infty} e^{n\parentheses[\big]{ P\parentheses[\big]{ f, \parentheses{ -\rootpressure + \frac{\imaginary t}{\sigma\sqrt{n}} } \normpotential } - P(f, -\rootpressure \normpotential) } } = e^{- \frac{t^2}{2}}.
	\]
	This implies that for each $t \in \real$, $\lim_{n \to +\infty} g_{n}(t) = 0$.
	Moreover, since $\varepsilon < \min\set[\Big]{\delta, \frac{\sigma^{2}}{4 C_{\delta}}}$, it follows from Lemma~\ref{lem:Taylor extension of pressure function in imaginary part} that if $\abs{t} < \varepsilon \sigma\sqrt{n}$, then 
	\[
		\abs[\Big]{ e^{n\parentheses[\big]{ P\parentheses[\big]{ f, \parentheses{ -\rootpressure + \frac{\imaginary t}{\sigma\sqrt{n}} } \normpotential } - P(f, -\rootpressure \normpotential) } } - e^{- \frac{t^2}{2}} } 
		\leqslant e^{ - \frac{t^{2}}{2}\parentheses[\big]{ 1 - \frac{2 C_{\delta} t}{\sigma^{3} \sqrt{n}} }  } + e^{- \frac{t^2}{2}}
		< e^{- \frac{t^2}{4}} + e^{- \frac{t^2}{2}}.
	\]
	Hence, by Lebesgue's dominated convergence theorem, $\lim_{n \to +\infty} \int_{ \abs{t} < \varepsilon \sigma \sqrt{n} } \,  g_{n}(t) \,\mathrm{d}t = 0$.
	This implies $\lim_{n \to +\infty} \widetilde{A}_{1}(n) = 0$ and establishes Claim~2.

	\smallskip
	
	\emph{Claim~3.}  $\lim_{n \to +\infty} A_{2}(n) = 0$.
	
	\smallskip


	Let $T > 1$ be the constant given by Proposition~\ref{prop:partition function estimate:expanding Thurston map:unbounded imaginary}.
	By Lemma~\ref{lem:partition function estimate:expanding Thurston map:bounded imaginary bounded cases}~\ref{item:lem:partition function estimate:expanding Thurston map:bounded imaginary bounded cases:intermediate part}, there exists $\vartheta \in (0, 1)$ such that
	\begin{equation}    \label{eq:temp:prop:estimate of partition function with periodic orbits:Claim3:medium case}
		\int_{ \varepsilon \leqslant \frac{ \abs{t}}{\sigma \sqrt{n} } \leqslant T } \abs[\bigg]{ \frac{\ell_{n}^{-1}}{e^{P(f, -\rootpressure \normpotential)n}} \widehat{\psi}_{n}\parentheses[\Big]{ \frac{t}{2 \pi \sigma \sqrt{n}} }  \normpartifun[\Big]{-\rootpressure + \frac{\imaginary t}{\sigma \sqrt{n}}} } \,\mathrm{d}t
		= \mathcal{O}\parentheses[\big]{ \sigma \sqrt{n} \, \ell_{n}^{-1} \uniformnorm{\widehat{\psi}_{n}} \vartheta^{n} }
	\end{equation}
	as $n \to +\infty$, where $\ell_{n}^{-1} \uniformnorm{\widehat{\psi}_{n}} \leqslant \ell_{n}^{-1} \int_{\real} \! \psi_{n}(x) \,\mathrm{d}x = \int_{\real} \! \psi(y) e^{\rootpressure \ell_{n} y} \,\mathrm{d}y$ is uniformly bounded for $n \in \n$ since $\psi$ has compact support. 
	Since $\psi_n \in C^{4}(\real, \real)$, we have that for all $n \in \n$ and $t \in \real$,
	\[
		(2 \pi \imaginary t)^{4} \widehat{\psi}_{n}(t) = \widehat{\psi_{n}^{(4)}}(t).
	\]	
	Moreover, since $\psi$ has compact support, we have that 
	\[
		\begin{split}
			\ell_{n}^{3} \, \uniformnorm[\Big]{ \widehat{\psi_{n}^{(4)}} } 
			\leqslant \ell_{n}^{3} \! \int_{\real} \abs[\big]{ \psi_{n}^{(4)}(x) }  \,\mathrm{d}x 
			&= \int_{\real} \, \abs[\bigg]{ \sum_{k=0}^{4} \binom{4}{k} \ell_{n}^{k} \, \psi^{(4 - k)}(y) (\rootpressure )^{k} e^{\rootpressure \ell_{n} y} }  \,\mathrm{d}y \leqslant C'
		\end{split}
	\]
	for some constant $C' \geqslant 0$, which is independent of $n$. 
	Then, by Proposition~\ref{prop:partition function estimate:expanding Thurston map:unbounded imaginary}, there exist $C > 0$ and $\rho \in (0, 1)$ such that for each integer $n \geqslant 2$,
	\begin{align*}
		&\int_{ \abs{t} > T \sigma \sqrt{n} } \, \abs[\bigg]{ \frac{\ell_{n}^{-1}}{e^{P(f, -\rootpressure \normpotential)n}} \widehat{\psi}_{n}\parentheses[\Big]{ \frac{t}{2 \pi \sigma \sqrt{n}} }  \normpartifun[\Big]{-\rootpressure + \frac{\imaginary t}{\sigma \sqrt{n}}} } \,\mathrm{d}t \\
		&\qquad \leqslant \int_{ \abs{t} > T \sigma \sqrt{n} } \, \abs[\bigg]{ \ell_{n}^{-1} \widehat{\psi}_{n}\parentheses[\Big]{ \frac{t}{2 \pi \sigma \sqrt{n}} } C \abs[\Big]{ \frac{t}{\sigma\sqrt{n}}  }^{2 + \varepsilon} \rho^{n} } \,\mathrm{d}t\\
		&\qquad = \int_{ \abs{t} > T } C \sigma \sqrt{n} \, \ell_{n}^{-1} \rho^{n} \abs{t}^{2 + \varepsilon} \abs[\Big]{ \widehat{\psi}_{n}\parentheses[\Big]{ \frac{t}{2 \pi} } } \,\mathrm{d}t \\
		&\qquad = C \sigma \sqrt{n} \, \ell_{n}^{-4} \rho^{n} \int_{ \abs{t} > T } \abs{t}^{-2 + \varepsilon} \, \ell_{n}^{3} \abs[\Big]{ \widehat{\psi^{(4)}_{n}} \parentheses[\Big]{ \frac{t}{2 \pi} } }  \,\mathrm{d}t \\
		&\qquad \leqslant C C' \sigma \sqrt{n} \, \ell_{n}^{-4} \rho^{n} \int_{ \abs{t} > T } \abs{t}^{-2 + \varepsilon}  \,\mathrm{d}t.
	\end{align*}
	Combining this with \eqref{eq:temp:prop:estimate of partition function with periodic orbits:Claim3:medium case} and recalling that the sequences $\sequen{\ell_{n}^{-1}}$ are of sub-exponential growth, we establish Claim~3.

	\smallskip

	Finally, it is clear that $\lim_{n \to +\infty} A_{3}(n) = 0$. 
	We conclude the proof by combining the estimates in Claims~2~and~3.
\end{proof}

The following corollary is an immediate consequence of Proposition~\ref{prop:estimate of cardinality periodic orbits} and Lemma~\ref{lem:error of shorter primitive periodic orbits}.
Recall that $P(f, -\rootpressure \normpotential) = \rootpressure \alpha > 0$.

\begin{corollary} \label{coro:estimate of cardinality prime orbits}
	Consider a non-negative function $\psi \in C^{4}(\real, \nonnegreal)$ with compact support.
	Under the assumptions in Subsection~\ref{sub:Notation and assumptions}, we have that
	\[
		\auxprimeorbitcounting \sim e^{\rootpressure p_n} \frac{\int_{\real} \! \psi_{n}(x) \,\mathrm{d}x}{\sigma \sqrt{2\pi}} \frac{e^{\rootpressure \alpha n}}{n^{3/2}} \quad \text{as } n \to +\infty.
	\]
	Here $\auxprimeorbitcounting$ is defined by \eqref{eq:def:cardinality prime orbits} and $\psi_{n}$ is defined by \eqref{eq:def:normalized testin function}.
\end{corollary}

\subsection{Approximation argument}%
\label{sub:Approximation argument}

In this subsection we prove Theorem~\ref{thm:main theorem} by using the auxiliary estimates from Subsection~\ref{sub:Auxiliary estimates} through an approximation argument.

\begin{proof}[Proof of Theorem~\ref{thm:main theorem}]
	By Lemma~\ref{lem:invariant_Jordan_curve not join opposite sides}, it suffices to prove the theorem for the case where $f(\mathcal{C}) \subseteq \mathcal{C}$ and no $1$-tile in $\Tile{1}$ joins opposite sides of $\mathcal{C}$.
	Hence all the assumptions in Subsection~\ref{sub:Notation and assumptions} are satisfied.

	For each $n \in \n$, we define 
	\[
		B(n) \define \ell_{n}^{-1} e^{-\rootpressure p_{n}} \frac{ \sigma \sqrt{2 \pi n^{3}} }{ e^{ \rootpressure \alpha n } } \primeorbitcard \, - \, \ell_{n}^{-1} \int_{I_{n}} \! e^{\rootpressure (z - p_{n})} \,\mathrm{d}z.
	\]
	The integral $\ell_{n}^{-1} \int_{I_{n}} \! e^{\rootpressure (z - p_{n})} \,\mathrm{d}z$ is uniformly bounded away from $0$ and $+\infty$ for $n \in \n$ as $K$ is compact. 
	Thus, in order to prove this theorem, it suffices to show that $\lim_{n \to +\infty} B(n) = 0$.

	Fix an arbitrary $\varepsilon \in (0, 1)$.

	We first construct a non-negative function $\psi \in C^{4}(\real, \nonnegreal)$ with compact support satisfying the following properties:
	\[
		\indicator{[-\frac{1}{2}, \frac{1}{2}]} \leqslant \psi \leqslant 1 + \varepsilon, \quad 
		\supp{\psi} \subseteq \squarebracket[\Big]{ - \frac{1 + \varepsilon}{2}, \frac{1 + \varepsilon}{2} }, \quad \text{and} \quad
		\int_{\real} \! \psi(x) \,\mathrm{d}x \leqslant 1 + \varepsilon.
	\]
	Let $\eta \in C^{\infty}(\real)$ be a non-negative mollifier defined by 
	\[
		\eta(x) \define 
		\begin{cases}
			C\myexp[\big]{ -\frac{1}{1 - x^{2}} } & \text{if } \abs{x} < 1; \\
			0 & \text{otherwise},
		\end{cases}
	\]
	where $C > 0$ is a normalization constant such that $\int_{\real} \eta(x) \,\mathrm{d}x = 1$.
	For each $\delta > 0$, set $\eta_{\delta}(x) \define \delta^{-1} \eta(x/\delta)$. 
	Consider $G \define (1 + \varepsilon / 4) \indicator{\squarebracket{ - (1 + \varepsilon / 4)/ 2, \, (1 + \varepsilon / 4)/ 2 }}$.
	Then there exists a sufficiently small $\delta > 0$ such that the function $\psi \define G * \eta_{\delta}$ satisfies the desired properties. 
	Here $*$ denotes the convolution of functions.

	By \eqref{eq:cardinality of periodic orbits with constraints}, \eqref{eq:def:cardinality prime orbits}, and Corollary~\ref{coro:estimate of cardinality prime orbits}, we have that
	\begin{align*}
		&\limsup_{n \to +\infty} B(n) \\
		&=\limsup_{n \to +\infty} \ell_{n}^{-1} e^{-\rootpressure p_{n}} \frac{ \sigma \sqrt{2 \pi n^{3}}}{e^{ \rootpressure \alpha n }} \sum_{\tau \in \primeorbit} \indicator{[- \frac{1}{2}, \frac{1}{2}]}\parentheses[\big]{ \ell_{n}^{-1}\parentheses[\big]{ l_{f, \, \potential}(\tau) - n \alpha - p_{n} }  } \, - \, \ell_{n}^{-1} \int_{I_{n}} \! e^{\rootpressure (z - p_{n})} \,\mathrm{d}z \\
		&\leqslant \limsup_{n \to +\infty} \ell_{n}^{-1} e^{-\rootpressure p_{n}} \frac{ \sigma \sqrt{2 \pi n^{3}}}{e^{ \rootpressure \alpha n }} \sum_{\tau \in \primeorbit} \psi\parentheses[\big]{ \ell_{n}^{-1}\parentheses[\big]{ l_{f, \, \potential}(\tau) - n \alpha - p_{n} }  }  \, - \, \ell_{n}^{-1} \int_{I_{n}} \! e^{\rootpressure (z - p_{n})} \,\mathrm{d}z.\\
		&= \limsup_{n \to +\infty} \ell_{n}^{-1} \int_{\real} \! \psi_{n}(x) \,\mathrm{d}x \, - \, \ell_{n}^{-1} \int_{I_{n}} \! e^{\rootpressure (z - p_{n})} \,\mathrm{d}z,
	\end{align*}
	where $\psi_{n}$ is defined by \eqref{eq:def:normalized testin function} and $\ell_{n}^{-1} \int_{\real} \! \psi_{n}(x) \,\mathrm{d}x = \int_{\real} \! \psi(y) e^{\rootpressure \ell_{n} y} \,\mathrm{d}y$ is uniformly bounded for $n \in \n$ since $\psi$ has compact support.
	For each $n \in \n$, 
	\begin{align*}
		\ell_{n}^{-1} \int_{\real} \! \psi_{n}(x) \,\mathrm{d}x 
		&= \int_{\real} \! \psi(y) e^{\rootpressure \ell_{n} y} \,\mathrm{d}y 
		= \int_{- \frac{1 + \varepsilon}{2}}^{\frac{1 + \varepsilon}{2} } \! \psi(y) e^{\rootpressure \ell_{n} y} \,\mathrm{d}y \\
		&\leqslant \int_{- \frac{1}{2}}^{\frac{1}{2} } \! \psi(y) e^{\rootpressure \ell_{n} y} \,\mathrm{d}y + \varepsilon (1 + \varepsilon) e^{ \rootpressure \abs{K} }  \\
		&\leqslant (1 + \varepsilon) \int_{- \frac{1}{2}}^{\frac{1}{2} } \! e^{\rootpressure \ell_{n} y} \,\mathrm{d}y + \varepsilon (1 + \varepsilon) e^{ \rootpressure \abs{K} }  \\
		&= (1 + \varepsilon) \ell_{n}^{-1} \int_{I_{n}} \! e^{\rootpressure (z - p_{n})} \,\mathrm{d}z + \varepsilon (1 + \varepsilon) e^{ \rootpressure \abs{K} }  \\
		&\leqslant \ell_{n}^{-1} \int_{I_{n}} \! e^{\rootpressure (z - p_{n})} \,\mathrm{d}z + \varepsilon (2 + \varepsilon) e^{ \rootpressure \abs{K} },
	\end{align*}
	where $\abs{K}$ denotes the diameter of the compact set $K$.
	Thus, we obtain that
	\begin{align*}
		\limsup_{n \to +\infty} B(n) 
		&\leqslant \varepsilon (2 + \varepsilon) e^{ \rootpressure \abs{K} } 
		\leqslant C_1 \varepsilon,
	\end{align*}
	where the constant $C_1 \define 3e^{ \rootpressure \abs{K} }$ depends only on $\rootpressure $ and $K$.
	
	Similarly, one can show that
	\[
		\liminf_{n \to +\infty} B(n) \geqslant - C_{2} \varepsilon
	\]
	for some constant $C_2 > 0$ that depends only on $-\rootpressure $ and $K$.

	Since the choice of $\varepsilon \in (0, 1)$ was arbitrary, we have that $\lim_{n \to +\infty} B(n) = 0$.
	This completes the proof.
\end{proof}

%% file: section/Appendix/Ruelle_lemma.tex


\section{Ruelle lemma}%
\label{sec:Appendix:Ruelle lemma}

This appendix presents a proof of the Ruelle lemma for one-sided subshifts of finite type (see Lemma~\ref{lem:Ruelle lemma for subshift of finite type}).
Estimates of this type first appeared implicitly in \cite{ruelle1990extension}, and variations were established in \cite{pollicottExponentialErrorTerms1998,pollicottErrorTermsClosed2001,naudExpandingMapsCantor2005}.
The lemma provides an estimate of the difference between the partition function and a sum involving iterates of the Ruelle operator acting on characteristic functions of cylinders.
This estimate is crucial for analyzing the analytic properties of dynamical zeta functions and establishing statistical limit theorems.

Throughout this appendix, we fix the following notation.
Consider a finite set of states $S$ and a transition matrix $A \colon S \times S \to \{0, 1\}$.
Denote by $\parentheses[\big]{ \Sigma_{A}^+, \sigma_{A} }$ the one-sided subshift of finite type defined by $A$.
Fix $\metricexpshiftspace \in (0, 1)$ and equip the space $\Sigma_{A}^+$ with the metric $\metriconshiftspace$ defined in \eqref{eq:def:metric on shift space}.
Let $\holderexp \in (0, 1]$ and $\phi \in \holderspace[\big][][\Sigma_{A}^{+}][\metriconshiftspace]$ be a real-valued \holder continuous function.
We refer the reader to Subsection~\ref{sub:Symbolic dynamics for expanding Thurston maps} for background on symbolic dynamics and the Ruelle operator.

We introduce definitions needed to state and prove the Ruelle lemma.

\begin{definition}[Admissible word, cylinder set, and concatenation]
    For $n \in \n$ and a finite word \( \omega = \omega_0 \cdots \omega_{n - 1} \) with \( \omega_i \in S \) and \( A(\omega_i, \omega_{i + 1}) = 1 \) for every \( 0 \leqslant i < n - 1 \), we say \( \omega \) is an \emph{admissible word of length \( n \)}, denoted \( \abs{\omega} = n \).
    The \emph{cylinder set} associated to \( \omega \) is
    \[
        C_\omega = [\omega] \define \set[\big]{ x = \set{x_{i}}_{i \in \n_{0}} \in \Sigma_{A}^{+} \describe x_i = \omega_i \text{ for } 0 \leqslant i < n } .
    \]
    We denote by \( \chi_\omega \) the characteristic function of the cylinder set \( C_\omega \).
    For a point \( x \in \Sigma_{A}^{+} \) satisfying \( A(\omega_{n - 1}, x_0) = 1 \), we write \( \omega x \in \Sigma_{A}^{+} \) for the concatenation obtained by prepending \( \omega \) to \( x \); explicitly, \( (\omega x)_i = \omega_i \) for \( 0 \leqslant i < n \) and \( (\omega x)_i = x_{i - n} \) for \( i \geqslant n \).
\end{definition}

We now state the Ruelle lemma.

\begin{lemma}[Ruelle lemma]      \label{lem:Ruelle lemma for subshift of finite type}
    For each \( a_0 > 0 \), each \( b_0 > 0 \), and each \( \varepsilon > 0 \), there exists a constant \( C_{\varepsilon} > 0 \) such that for all $k \in \n$, \( n \in \n \) with $n \geqslant 2$, and \( s \in \cx \) with \( \abs{\Re{s}} \leqslant a_0 \) and \( \abs{\Im{s}} \geqslant b_0 \), we have
    \begin{equation} \label{eq:lem:Ruelle lemma for subshift of finite type:bound on iterates of Ruelle operator on characteristic functions of cylinder sets}
        \sum_{ \abs{\omega} = k } \bigl\| \mathcal{L}_{s\phi}^{k} \chi_\omega \bigr\|_{C^{0,\holderexp}} 
        \leqslant C_\varepsilon \abs{\Im{s}} \metricexpshiftspace^{\holderexp} e^{k (P(\sigma, \Re{s} \phi) + \varepsilon)}
    \end{equation}
    and
    \begin{equation} \label{eq:lem:Ruelle lemma for subshift of finite type:Ruelle lemma}
        \abs[\bigg]{ \sum_{\sigma^n x = x} e^{ s S_n\phi(x)} - \sum_{j \in S} \mathcal{L}^n_{s \phi} \chi_{j}(x_{j}) }
        \leqslant C_\varepsilon \abs{\Im{s}} \sum_{m=2}^{n} \norm[\big]{ \mathcal{L}^{n - m}_{s \phi} }_{C^{0,\holderexp}} \parentheses[\big]{ \metricexpshiftspace^{\holderexp} e^{P(\sigma, \Re{s} \phi) + \varepsilon} }^{m}
    \end{equation}
    for any choice of a point \( x_j \in C_j \) for each cylinder set \( C_j \).
    Here the sum in \eqref{eq:lem:Ruelle lemma for subshift of finite type:bound on iterates of Ruelle operator on characteristic functions of cylinder sets} is taken over all admissible words \( \omega \) of length \( k \), and \( \norm[\big]{ \mathcal{L}^{n-m}_{s \phi} }_{C^{0,\holderexp}} \) denotes the operator norm of \( \mathcal{L}^{n-m}_{s \phi} \) on \( C^{0,\holderexp} \parentheses[\big]{ \parentheses[\big]{ \Sigma_{A}^{+}, \metriconshiftspace }, \cx } \).
\end{lemma}
\begin{rmk}
    The constant \( C_\varepsilon \) in Lemma~\ref{lem:Ruelle lemma for subshift of finite type} only depends on \( \varepsilon \), \( a_0 \), \( b_0 \), \( \metricexpshiftspace \), \( \holderexp \), \( \phi \), and the ambient parameters of the subshift; it is independent of \( n \) and the specific value of \( s \) (as long as \( \abs{\Re{s}} \leqslant a_0 \) and \( \abs{\Im{s}} \geqslant b_0 \)).
\end{rmk}

The following lemma expresses the partition function as a sum over the Ruelle operator acting on the characteristic functions of the cylinder sets.

\begin{lemma}    \label{lem:periodic point representation}
    For each $\psi \in C(X, \cx)$ and each \( n \in \n \) we have
    \[
        \sum_{x \in \fixpoint{\sigma^n}} e^{S_n \psi(x)} 
        = \sum_{\abs{\omega} = n} (\mathcal{L}_\psi^n \chi_\omega)(x_\omega) ,
    \]
    where for each admissible word \( \omega \) the point \( x_\omega \in C_\omega \) is defined as follows: if the cylinder set \( C_\omega \) contains a fixed point of \( \sigma^n \), we take \( x_\omega \) to be that point; otherwise we choose \( x_\omega \) arbitrarily in \( C_\omega \).
\end{lemma}
\begin{proof}
    Fix an arbitrary $n \in \n$.
    Since $\Sigma_{A}^{+} = \bigcup\limits_{\abs{\omega} = n} C_{\omega}$, we have
    \[
        \sum_{x \in \fixpoint{\sigma^n}} e^{S_n \psi(x)} 
        = \sum_{\abs{\omega} = n} \, \sum_{x \in C_{\omega} \cap \fixpoint{\sigma^n}} e^{S_n \psi(x)}.
    \]
    By definition of the Ruelle operator,
    \[
        (\mathcal{L}_\psi^n \chi_\omega)(x_\omega) 
        = \sum_{y \in C_\omega \cap \sigma^{-n}(x_{\omega}) } e^{S_n \psi(y)}.
    \]
    Thus it suffices to show that $C_{\omega} \cap \fixpoint{\sigma^n} = C_\omega \cap \sigma^{-n}(x_{\omega})$ for each admissible word \( \omega \) with \( |\omega| = n \).

    Let $\omega$ be an arbitrary admissible word of length $n$.
    Note that the cylinder set $C_\omega$ can contain at most one point from each of the sets $\fixpoint{\sigma^n}$ and $\sigma^{-n}(x_{\omega})$. 
    We consider two cases.

    If $A(\omega_{n-1}, \omega_0) = 1$, the word $\omega$ can be followed by itself, which forms the point $x_{\omega}^{*} \define \omega\omega\cdots$. 
    This point lies in $C_\omega$ and is a fixed point of $\sigma^n$. 
    This implies that $x_{\omega} = x_{\omega}^{*}$ and
    \[
        C_{\omega} \cap \fixpoint{\sigma^n} = C_{\omega} \cap \sigma^{-n}(x_{\omega}) = \{x_{\omega}\}.
    \]
    On the other hand, if $A(\omega_{n-1}, \omega_0) = 0$, then the word $\omega$ cannot be followed by the symbol $\omega_0$. 
    This implies that 
    \[
        C_{\omega} \cap \fixpoint{\sigma^n} = C_{\omega} \cap \sigma^{-n}(x_{\omega}) = \emptyset.
    \]
    In either case, the equality $C_{\omega} \cap \fixpoint{\sigma^n} = C_{\omega} \cap \sigma^{-n}(x_{\omega})$ holds. 
    This completes the proof.
\end{proof}


We now prove Lemma~\ref{lem:Ruelle lemma for subshift of finite type}.
The proof proceeds in three steps: establishing a telescoping sum identity, estimating the \holder norms, and bounding the sum over admissible words.

\begin{proof}[Proof of Lemma~\ref{lem:Ruelle lemma for subshift of finite type}]
    \def\wordtruncation{\widehat{\omega}}
    Let $a_0 > 0$, $b_0 > 0$, and $\varepsilon > 0$ be arbitrary.
    Fix an arbitrary integer $n \geqslant 2$.
    Consider $s = a + \imaginary b \in \cx$ with $\abs{a} \leqslant a_0$ and $\abs{b} \geqslant b_0$, where $a \define \Re{s}$ and $b \define \Im{s}$.

    For each admissible word \( \omega \) of length \( n \), we choose \( x_\omega \in C_\omega \) in the following way:
    if the cylinder set \( C_\omega \) contains a fixed point of \( \sigma^n \), we take \( x_\omega \) to be that point; otherwise we choose \( x_\omega \) arbitrarily in \( C_\omega \).
    Next, for each $m \in \oneton[n - 1]$ and each admissible word \( \omega \) of length $m$, we fix an arbitrary point $x_{\omega} \in C_{\omega}$.

    Denote
    \[
        Z_n(s) \define \sum_{x \in \fixpoint{\sigma^n}} e^{s S_n \phi(x)} \quad \text{and} \quad T_n(s) \define \sum_{j \in S} (\mathcal{L}_{s\phi}^n \chi_{j})(x_{j}).
    \]

    \emph{Claim.}
    \[
        Z_n(s) - T_n(s) = \sum_{m=2}^n \sum_{\abs{\omega} = m} \parentheses[\big]{ (\mathcal{L}_{s\phi}^n \chi_\omega)(x_\omega) - (\mathcal{L}_{s\phi}^n \chi_\omega)(x_{\wordtruncation}) },
    \]
    where \( \wordtruncation \) denotes the word obtained by removing the last symbol from \( \omega \).
    \smallskip

    \emph{Proof of the claim.}
    By Lemma~\ref{lem:periodic point representation}, $Z_n(s) = \sum_{\abs{\omega} = n} (\mathcal{L}_{s\phi}^n \chi_\omega)(x_\omega)$.
    For \( m \in \oneton \), define
    \[
        A_m \define \sum_{\abs{\omega} = m} (\mathcal{L}_{s\phi}^n \chi_\omega)(x_\omega) .
    \]
    Then $Z_n(s) - T_n(s) = A_n - A_1 = \sum_{m=2}^n (A_m - A_{m-1})$.
    Note that for each $m \in \atob{2}{n}$ and each admissible word \( \eta \) of length \( m - 1 \), the characteristic function satisfies
    \[
        \chi_\eta = \sum_{j \in S \summationdescribe A(\eta_{m-2}, j) = 1} \chi_{\eta j} ,
    \]
    where \( \eta j \) denotes the word obtained by appending \( j \) to \( \eta \).
    Thus for each $m \in \atob{2}{n}$, we have
    \begin{align*}
        A_{m-1} &= \sum_{\abs{\eta} = m-1} (\mathcal{L}_{s\phi}^n \chi_\eta)(x_\eta) \\
        &= \sum_{\abs{\eta} = m-1} \sum_{j \in S \summationdescribe A(\eta_{m-2}, j) = 1} (\mathcal{L}_{s\phi}^n \chi_{\eta j})(x_\eta) \\
        &= \sum_{\abs{\omega} = m} (\mathcal{L}_{s\phi}^n \chi_\omega)(x_{\wordtruncation}) ,
    \end{align*}
    where we set $\omega = \eta j$, so that $\wordtruncation = \eta$.
    It follows that
    \[
        A_m - A_{m-1} = \sum_{\abs{\omega} = m} \parentheses[\big]{ (\mathcal{L}_{s\phi}^n \chi_\omega)(x_\omega) - (\mathcal{L}_{s\phi}^n \chi_\omega)(x_{\wordtruncation}) }.
    \]
    Summing this equality over $m \in \atob{2}{n}$ yields the desired identity, establishing the claim.

    \smallskip

    By the claim, we have
    \begin{align*}
        \abs{ Z_n(s) - T_n(s) } 
        &\leqslant \sum_{m=2}^n \sum_{\abs{\omega} = m} \abs{ (\mathcal{L}_{s\phi}^n \chi_\omega)(x_\omega) - (\mathcal{L}_{s\phi}^n \chi_\omega)(x_{\wordtruncation}) }  \\
        &\leqslant \sum_{m=2}^n \sum_{\abs{\omega} = m} | \mathcal{L}_{s\phi}^n \chi_\omega |_{\holderexp} \, \metriconshiftspace(x_\omega, x_{\wordtruncation})^\holderexp \\
        &\leqslant \sum_{m=2}^n \sum_{\abs{\omega} = m} \| \mathcal{L}_{s\phi}^n \chi_\omega \|_{C^{0,\holderexp}} \, \metricexpshiftspace^{\holderexp(m-1)}.
    \end{align*}
    Since \( \mathcal{L}_{s\phi}^{n - m} \) is bounded on $C^{0, \holderexp} \parentheses[\big]{ \parentheses[\big]{ \Sigma_{A}^{+}, \metriconshiftspace }, \cx }$ (cf.~\cite[Proposition~2.1]{pollicottZetaFunctionsPeriodic1990}), we have
    \[
        \| \mathcal{L}_{s\phi}^n \chi_\omega \|_{C^{0,\holderexp}} 
        = \| \mathcal{L}_{s\phi}^{n-m} \parentheses[\big]{ \mathcal{L}_{s\phi}^m \chi_\omega } \|_{C^{0,\holderexp}} 
        \leqslant \norm[\big]{ \mathcal{L}_{s\phi}^{n-m} }_{C^{0,\holderexp}} \, \| \mathcal{L}_{s\phi}^m \chi_\omega \|_{C^{0,\holderexp}} ,
    \]
    where \( \norm[\big]{ \mathcal{L}_{s\phi}^{n-m} }_{C^{0,\holderexp}} \) denotes the operator norm of \( \mathcal{L}_{s\phi}^{n-m} \) on \( C^{0,\holderexp} \parentheses[\big]{ \parentheses[\big]{ \Sigma_{A}^{+}, \metriconshiftspace }, \cx } \).
    Hence,
    \begin{equation}     \label{eq:temp:lem:Ruelle lemma for subshift of finite type:auxiliary bound}
        \abs{ Z_n(s) - T_n(s) } \leqslant \metricexpshiftspace^{-\holderexp} \sum_{m = 2}^n \metricexpshiftspace^{\holderexp m} \, \norm[\big]{ \mathcal{L}_{s\phi}^{n-m} }_{C^{0,\holderexp}} \sum_{\abs{\omega} = m} \| \mathcal{L}_{s\phi}^m \chi_\omega \|_{C^{0,\holderexp}}.
    \end{equation}

    We now bound the sum \( \sum_{\abs{\omega} = m} \| \mathcal{L}_{s\phi}^m \chi_\omega \|_{C^{0,\holderexp}} \) for all $m \in \n$.
    For each admissible word \( \omega \) of length \( m \), we have
    \[
        (\mathcal{L}_{s\phi}^m \chi_\omega)(x) = \sum_{y \in C_\omega \summationdescribe \sigma^m(y) = x} e^{s S_m \phi(y)} \qquad \text{for } x \in \Sigma_{A}^{+}.
    \]
    This implies that \( \uniformnorm{ \mathcal{L}_{s\phi}^m \chi_\omega } \leqslant \sup_{y \in C_\omega} e^{a S_m \phi(y)} \).
    For each \( x \in \Sigma_{A}^{+} \), there is at most one \( y \in C_\omega \) satisfying \( \sigma^m(y) = x \).
    When such a \( y \) exists, we denote it by \( y_x^\omega \).
    Thus,
    \[
        (\mathcal{L}_{s\phi}^m \chi_\omega)(x) =
        \begin{cases}
            e^{s S_m \phi(y_x^\omega)} & \text{if } y_x^\omega \text{ exists}; \\
            0 & \text{otherwise}.
        \end{cases}
    \]

    To estimate the \holder seminorm $| \mathcal{L}_{s\phi}^m \chi_\omega |_{\holderexp}$, we consider distinct points \( x = \set{x_i}_{i \in \n_{0}} \) and \( z = \set{z_i}_{i \in \n_{0}} \) in \( \Sigma_{A}^{+} \).
    According to the existence of \( y_x^\omega \) and \( y_z^\omega \), we consider the following three cases.
    
    \smallskip
    \emph{Case 1: Neither \( y_x^\omega \) nor \( y_z^\omega \) exists.} 
    Then \( (\mathcal{L}_{s\phi}^m \chi_\omega)(x) = (\mathcal{L}_{s\phi}^m \chi_\omega)(z) = 0 \), and \( \abs{ (\mathcal{L}_{s\phi}^m \chi_\omega)(x) - (\mathcal{L}_{s\phi}^m \chi_\omega)(z) } = 0 \) holds trivially.

    \smallskip
    \emph{Case 2: Exactly one of \( y_x^\omega \) and \( y_z^\omega \) exists.}
    In this case, we must have $\metriconshiftspace(x, z) = 1$ (i.e., $x_0 \ne z_0$); otherwise the condition $x_0 = z_0$ would imply that either both $y_x^\omega$ and $y_z^\omega$ exist, or neither exists, contradicting the assumption that exactly one of them exists.
    Then one of \( (\mathcal{L}_{s\phi}^m \chi_\omega)(x) \) and \( (\mathcal{L}_{s\phi}^m \chi_\omega)(z) \) vanishes, and we have
    \[
        \abs[\big]{ (\mathcal{L}_{s\phi}^m \chi_\omega)(x) - (\mathcal{L}_{s\phi}^m \chi_\omega)(z) } 
        \leqslant \uniformnorm{ \mathcal{L}_{s\phi}^m \chi_\omega } 
        \leqslant \metriconshiftspace(x, z)^{\holderexp} \sup_{y \in C_\omega} e^{a S_m \phi(y)}.
    \]

    \emph{Case 3: Both \( y_x^\omega \) and \( y_z^\omega \) exist.}
    Since $y_x^\omega = \omega x$ and $y_z^\omega = \omega z$ by definition, we have \( \metriconshiftspace(y_x^\omega, y_z^\omega) =  \metriconshiftspace(x, z) \metricexpshiftspace^m \).
    Using the \holder continuity of \( \phi \), we deduce that
    \begin{align*}
        \abs{ S_m \phi(y_x^\omega) - S_m \phi(y_z^\omega) } 
        &\leqslant \sum_{i = 0}^{m - 1} |\phi(\sigma^{i}(y_x^\omega)) - \phi(\sigma^{i}(y_z^\omega))| \\ 
        &\leqslant \sum_{i = 0}^{m - 1} \| \phi \|_{C^{0, \holderexp}} \metriconshiftspace\parentheses[\big]{\sigma^{i}(y_x^\omega), \sigma^{i}(y_z^\omega)}^\holderexp \\ 
        &= \| \phi \|_{C^{0, \holderexp}} \metriconshiftspace(y_x^\omega, y_z^\omega)^{\holderexp} \sum_{i = 0}^{m - 1} \metricexpshiftspace^{- i \holderexp} \\ 
        &= \| \phi \|_{C^{0, \holderexp}} \metriconshiftspace(x, z)^{\holderexp} \metricexpshiftspace^{\holderexp m} \frac{ \metricexpshiftspace^{- \holderexp m} - 1 }{\metricexpshiftspace^{-\holderexp} - 1} \\
        &\leqslant \frac{ \| \phi \|_{C^{0, \holderexp}} }{\metricexpshiftspace^{-\holderexp} - 1} \metriconshiftspace(x, z)^{\holderexp}.
    \end{align*}
    Applying the inequality \( \abs{e^{z} - e^{w}} \leqslant \abs{z - w} \, e^{\max\set{\Re{z}, \,\Re{w}}} \) for $z, w \in \cx$, which is a consequence of the Fundamental Theorem of Calculus for line integrals, we obtain that
    \begin{align*}
        \abs[\big]{ (\mathcal{L}_{s\phi}^m \chi_\omega)(x) - (\mathcal{L}_{s\phi}^m \chi_\omega)(z) }
        &= \abs[\big]{ e^{s S_m \phi(y_x^\omega)} - e^{s S_m \phi(y_z^\omega)} } \\
        &\leqslant \abs{s} \, \abs{ S_m \phi(y_x^\omega) - S_m \phi(y_z^\omega) } \, e^{a \max\set{\Re{S_m \phi(y_x^\omega)}, \, \Re{S_m \phi(y_z^\omega)}}} \\
        &\leqslant \frac{ \| \phi \|_{C^{0, \holderexp}} }{\metricexpshiftspace^{-\holderexp} - 1} \abs{s} \, \metriconshiftspace(x, z)^{\holderexp} \sup_{y \in C_\omega} e^{a S_m \phi(y)}.
    \end{align*}
    Since $\abs{s} = \sqrt{a^2 + b^2} \leqslant \abs{a} + \abs{b} \leqslant a_0 + \abs{b} \leqslant \parentheses[\big]{ 1 + \frac{a_0}{b_0} } \abs{\Im{s}}$, it follows that
    \[
        \abs[\big]{ (\mathcal{L}_{s\phi}^m \chi_\omega)(x) - (\mathcal{L}_{s\phi}^m \chi_\omega)(z) } \leqslant \frac{ \parentheses[\big]{ 1 + \frac{a_0}{b_0} } \| \phi \|_{C^{0, \holderexp}} }{\metricexpshiftspace^{-\holderexp} - 1} \, \abs{\Im{s}} \, \metriconshiftspace(x, z)^{\holderexp} \sup_{y \in C_\omega} e^{a S_m \phi(y)}.
    \]

    Combining the estimates from these three cases, we conclude that
    \begin{align*}
        | \mathcal{L}_{s\phi}^m \chi_\omega |_{\holderexp} 
        &\leqslant \max \set[\bigg]{1, \, \frac{ \parentheses[\big]{ 1 + \frac{a_0}{b_0} } \| \phi \|_{C^{0, \holderexp}} }{\metricexpshiftspace^{-\holderexp} - 1} \, \abs{\Im{s}} } \sup_{y \in C_\omega} e^{a S_m \phi(y)}  
        \leqslant C_0 \abs{\Im{s}} \sup_{y \in C_\omega} e^{a S_m \phi(y)},
    \end{align*}
    where $C_0 \define \max \set[\Big]{ \frac{1}{b_0}, \, \frac{ \parentheses[\big]{ 1 + \frac{a_0}{b_0} } \| \phi \|_{C^{0, \holderexp}} }{\metricexpshiftspace^{-\holderexp} - 1} }$.    
    Thus we have
    \begin{align*}
        \| \mathcal{L}_{s\phi}^m \chi_\omega \|_{C^{0,\holderexp}}
        &= \uniformnorm{ \mathcal{L}_{s\phi}^m \chi_\omega } + | \mathcal{L}_{s\phi}^m \chi_\omega |_{\holderexp} \\
        &\leqslant \parentheses[\big]{ 1 + C_0 \abs{\Im{s}} } \sup_{y \in C_\omega} e^{a S_m \phi(y)} \\
        &\leqslant C \abs{\Im{s}} \sup_{y \in C_\omega} e^{a S_m \phi(y)} ,
    \end{align*}
    where the constant \( C \define \frac{1}{b_0} + C_0 \) only depends on \( a_0 \), \( b_0 \), \( \metricexpshiftspace \), \( \holderexp \), and \( \phi \).
    Summing over all admissible words \( \omega \) of length \( m \) yields
    \begin{equation} \label{eq:temp:lem:Ruelle lemma for subshift of finite type:sum of holder norm of Ruelle operator acting on characteristic functions}
        \sum_{\abs{\omega} = m} \| \mathcal{L}_{s\phi}^m \chi_\omega \|_{C^{0,\holderexp}} \leqslant C \abs{\Im{s}} \sum_{\abs{\omega} = m} \sup_{y \in C_\omega} e^{a S_m \phi(y)}.
    \end{equation}

    Denote
    \[
        Z_m(a\phi) \define \sum_{\abs{\omega} = m} \sup_{y \in C_\omega} e^{a S_m \phi(y)} .
    \]
    It is a classical result that $\lim_{m \to +\infty} \frac{1}{m} \log Z_m(a\phi) = P(\sigma, a\phi)$ (see~e.g.~\cite[Theorem~9.6]{walters1982introduction}).
    We now show that this convergence is uniform for $a \in [-a_0, a_0]$.
    Indeed, a straightforward calculation shows that for each $m \in \n$, the function $a \mapsto p_m(a) \define \frac{1}{m} \log Z_m(a\phi)$ is Lipschitz continuous on $\real$ with a Lipschitz constant $\uniformnorm{\phi}$, and is bounded on $[-a_0, a_0]$ by $a_0 \uniformnorm{\phi} + \log (\card{S})$.
    This implies that the sequence of functions $\set{ p_m }_{m \in \n}$ is equicontinuous and uniformly bounded on $[-a_0, a_0]$.
    By the \aalem theorem, every subsequence $\set{ p_{m_k} }_{k \in \n}$, which is equicontinuous and uniformly bounded on $[-a_0, a_0]$, has a further subsequence $\set[\big]{ p_{m_{k_j}} }_{j \in \n}$ that converges uniformly on $[-a_0, a_0]$.
    Note that such a subsequential limit must be $P(\sigma, \cdot \, \phi)$ since $\lim_{m \to +\infty} p_m(a) = P(\sigma, a\phi)$.
    A standard result in topology states that if every subsequence of a sequence in a topological space has a further subsequence converging to the same point, then the original sequence itself converges to that point.
    Therefore, the convergence $\lim_{m \to +\infty} p_m(a) = P(\sigma, a\phi)$ is uniform in $a \in [-a_0, a_0]$.
    Hence, there exists a constant \( K_\varepsilon > 0 \) such that for all $m \in \n$ and $a \in [-a_0, a_0]$, we have
    \begin{equation} \label{eq:temp:lem:Ruelle lemma for subshift of finite type:upper bound of partition function}
        Z_m(a\phi) \leqslant K_\varepsilon \, e^{m (P(\sigma, a\phi) + \varepsilon)},
    \end{equation}
    and \eqref{eq:lem:Ruelle lemma for subshift of finite type:bound on iterates of Ruelle operator on characteristic functions of cylinder sets} follows immediately from \eqref{eq:temp:lem:Ruelle lemma for subshift of finite type:sum of holder norm of Ruelle operator acting on characteristic functions} and \eqref{eq:temp:lem:Ruelle lemma for subshift of finite type:upper bound of partition function} by setting \( C_\varepsilon \define \metricexpshiftspace^{-\holderexp} C K_\varepsilon \).

    Finally, by substituting the estimate \eqref{eq:lem:Ruelle lemma for subshift of finite type:bound on iterates of Ruelle operator on characteristic functions of cylinder sets} into \eqref{eq:temp:lem:Ruelle lemma for subshift of finite type:auxiliary bound}, we obtain
    \begin{align*}
        \abs{ Z_n(s) - T_n(s) }
        &\leqslant \metricexpshiftspace^{-\holderexp} \sum_{m = 2}^n \metricexpshiftspace^{\holderexp m} \, \norm[\big]{ \mathcal{L}_{s\phi}^{n-m} }_{C^{0,\holderexp}} \sum_{\abs{\omega} = m} \| \mathcal{L}_{s\phi}^m \chi_\omega \|_{C^{0,\holderexp}}  \\
        &\leqslant C_\varepsilon \abs{\Im{s}} \sum_{m = 2}^{n}  \norm[\big]{ \mathcal{L}_{s\phi}^{n - m} }_{C^{0,\holderexp}} \metricexpshiftspace^{\holderexp m} e^{m (P(\sigma, a\phi) + \varepsilon)},
    \end{align*}
    establishing \eqref{eq:lem:Ruelle lemma for subshift of finite type:Ruelle lemma}.
\end{proof}

The following estimate for the partition function is an immediate consequence of Lemma~\ref{lem:Ruelle lemma for subshift of finite type}.

\begin{corollary} \label{coro:bound of partition function for subshift of finite type via Ruelle lemma}
    For all \( a_0 > 0 \), \( b_0 > 0 \), \( \varepsilon > 0 \), \( n \in \n\), and \( s \in \cx \) satisfying \( \abs{\Re{s}} \leqslant a_0 \) and \( \abs{\Im{s}} \geqslant b_0 \), we have
    \begin{equation} \label{eq:coro:bound of partition function for subshift of finite type via Ruelle lemma:bound on partition function}
        \abs[\bigg]{ \sum_{\sigma^n x = x} e^{ s S_n\phi(x)} }
        \leqslant C_\varepsilon \abs{\Im{s}} \sum_{m = 1}^{n} \norm[\big]{ \mathcal{L}^{n - m}_{s \phi} }_{C^{0,\holderexp}} \parentheses[\big]{ \metricexpshiftspace^{\holderexp} e^{P(\sigma, \Re{s} \phi) + \varepsilon} }^{m},
    \end{equation}
    where \( C_\varepsilon \) is the constant from Lemma~\ref{lem:Ruelle lemma for subshift of finite type}, which depends only on \( \varepsilon \), \( a_0 \), \( b_0 \), \( \metricexpshiftspace \), \( \holderexp \), \( \phi \), and the ambient parameters of the subshift $\parentheses[\big]{ \sft, \sopt }$.
\end{corollary}
\begin{proof}
    When $n = 1$, \eqref{eq:coro:bound of partition function for subshift of finite type via Ruelle lemma:bound on partition function} follows immediately from Lemma~\ref{lem:periodic point representation} and \eqref{eq:lem:Ruelle lemma for subshift of finite type:bound on iterates of Ruelle operator on characteristic functions of cylinder sets} in Lemma~\ref{lem:Ruelle lemma for subshift of finite type}.
    
    We now consider the case $n \geqslant 2$.
    Let $x_j \in C_j$ be an arbitrary point for each $j \in S$.
    By \eqref{eq:lem:Ruelle lemma for subshift of finite type:bound on iterates of Ruelle operator on characteristic functions of cylinder sets}, we have
    \begin{align*}
        \uppercase\expandafter{\romannumeral1} \define 
        \abs[\bigg]{\sum_{j \in S} \mathcal{L}^n_{s \phi} \chi_{j}(x_{j})} 
        &\leqslant \sum_{j \in S} \abs[\big]{ \mathcal{L}^{n - 1}_{s \phi} (\mathcal{L}_{s \phi} \chi_{j}) (x_{j}) }
        \leqslant \norm[\big]{ \mathcal{L}^{n - 1}_{s \phi} }_{C^{0, \holderexp}} 
            \sum_{j \in S}  \norm[\big]{ \mathcal{L}_{s \phi} \chi_{j} }_{C^{0, \holderexp}} \\
        &\leqslant \norm[\big]{ \mathcal{L}^{n - 1}_{s \phi} }_{C^{0, \holderexp}} \, C_\varepsilon \abs{\Im{s}} \metricexpshiftspace^{\holderexp} e^{P(\sigma, \Re{s} \phi) + \varepsilon},
    \end{align*}
    where $C_\varepsilon$ is the constant from Lemma~\ref{lem:Ruelle lemma for subshift of finite type}.
    Hence, using the triangle inequality and \eqref{eq:lem:Ruelle lemma for subshift of finite type:Ruelle lemma} in Lemma~\ref{lem:Ruelle lemma for subshift of finite type}, we deduce that
    \begin{align*}
        &\abs[\bigg]{ \sum_{\sigma^n x = x} e^{ s S_n\phi(x)} }
        \leqslant \uppercase\expandafter{\romannumeral1} + \abs[\bigg]{ \sum_{\sigma^n x = x} e^{ s S_n\phi(x)} - \sum_{j \in S} \mathcal{L}^n_{s \phi} \chi_{j}(x_{j}) }  \\
        &\qquad \leqslant C_\varepsilon \abs{\Im{s}} \norm[\big]{ \mathcal{L}^{n - 1}_{s \phi} }_{C^{0, \holderexp}} \metricexpshiftspace^{\holderexp} e^{P(\sigma, \Re{s} \phi) + \varepsilon} 
            + C_\varepsilon \abs{\Im{s}} \sum_{m=2}^{n} \norm[\big]{ \mathcal{L}^{n - m}_{s \phi} }_{C^{0,\holderexp}} \parentheses[\big]{ \metricexpshiftspace^{\holderexp} e^{P(\sigma, \Re{s} \phi) + \varepsilon} }^{m} \\
        &\qquad = C_\varepsilon \abs{\Im{s}} \sum_{m = 1}^{n} \norm[\big]{ \mathcal{L}^{n - m}_{s \phi} }_{C^{0,\holderexp}} \parentheses[\big]{ \metricexpshiftspace^{\holderexp} e^{P(\sigma, \Re{s} \phi) + \varepsilon} }^{m}.
    \end{align*}
    This establishes \eqref{eq:coro:bound of partition function for subshift of finite type via Ruelle lemma:bound on partition function}.
\end{proof}